\documentclass[10pt,oneside,reqno]{amsart}
\usepackage{hyperref}
\usepackage{amsmath}
\usepackage{amssymb}
\usepackage{amsxtra}
\usepackage{amsthm}

\allowdisplaybreaks[4]
\usepackage{stmaryrd}
\usepackage{bm}
\usepackage{bbm}
\usepackage{upgreek}
\usepackage{thmtools}
\usepackage{mathtools}
\usepackage{geometry}
\usepackage{enumitem}
\usepackage{comment}
\usepackage{microtype}

\usepackage{mathpazo}
\usepackage{domitian}
\usepackage[T1]{fontenc}

\AtBeginDocument{
 \DeclareSymbolFont{AMSb}{U}{msb}{m}{n}
 \DeclareSymbolFontAlphabet{\mathbb}{AMSb}
 }
 
\let\SSec\S

\newcommand\lopen{\mathopen{]}}
\newcommand\ropen{\mathclose{[}}

\newcommand\rmd{\mathrm{d}}
\newcommand\rme{\mathrm{e}}
\newcommand\rmi{\mathrm{i}}

\newcommand\cK{\mathcal{K}}

\newcommand\cO{\mathcal{O}}

\newcommand\cX{\mathcal{X}}

\renewcommand\AA{\mathbb{A}}

\newcommand\CC{\mathbb{C}}

\newcommand\QQ{\mathbb{Q}}
\newcommand\RR{\mathbb{R}}

\newcommand\ZZ{\mathbb{Z}}

\newcommand\mf[1]{\mathfrak{#1}}

\newcommand\A{\mathsf{A}}
\newcommand\B{\mathsf{B}}

\newcommand\G{\mathsf{G}}
\renewcommand\H{\mathsf{H}}

\newcommand\M{\mathsf{M}}
\newcommand\N{\mathsf{N}}

\renewcommand\S{\mathsf{S}}
\newcommand\T{\mathsf{T}}

\newcommand\Z{\mathsf{Z}}

\newcommand\SL{\mathsf{SL}}
\newcommand\GL{\mathsf{GL}}

\newcommand\SO{\mathsf{SO}}

\newcommand\triv{\mathbbm{1}}
\newcommand\dpii{2\uppi\rmi}
\newcommand\bs{\backslash}
\newcommand\diff{\partial}

\renewcommand\Re{\mathop{\mathrm{Re}}}

\DeclareMathOperator\Tr{Tr}
\DeclareMathOperator\sgn{sgn}

\DeclareMathOperator\res{res}
\DeclareMathOperator\fp{f.p.}
\DeclareMathOperator\pv{p.v.}
\DeclareMathOperator\ord{ord}

\DeclareMathOperator\Res{Res}
\DeclareMathOperator\Ind{Ind}

\newcommand\fin{\mathrm{fin}}
\newcommand\id{\mathrm{id}}
\newcommand\el{\mathrm{ell}}
\newcommand\hyp{\mathrm{hyp}}
\newcommand\unip{\mathrm{unip}}
\newcommand\spec{\mathrm{spec}}
\newcommand\cont{\mathrm{cont}}
\newcommand\cusp{\mathrm{cusp}}

\newcommand\reg{\mathrm{reg}}

\DeclareMathOperator\vol{vol}
\DeclareMathOperator\orb{orb}
\DeclareMathOperator\worb{worb}
\renewcommand\leq{\leqslant}
\renewcommand\geq{\geqslant}
\let\oldsslash\sslash
\renewcommand\sslash{{\oldsslash}}

\newcommand\legendresymbol[2]{\genfrac{(}{)}{}{}{#1}{#2}}

\makeatletter
\def\subsection{\@startsection{subsection}{2}%
  \z@{3pt\@plus0pt}{-.5em}%
  {\normalfont\bfseries}}
\makeatother

\makeatletter
\def\subsubsection{\@startsection{subsubsection}{2}
  \z@{0pt\@plus0pt}{-.5em}
  {\normalfont\itshape}}
\makeatother

\makeatletter
\def\@seccntformat#1{
  \protect\textup{\protect\@secnumfont
    \ifnum\pdfstrcmp{subsection}{#1}=0 \bfseries\fi
    \csname the#1\endcsname
    \protect\@secnumpunct
  }%
}  
\makeatother

\newtheoremstyle{THEOREM}
{2.5pt}
{2pt}
{\itshape}
{}
{\bfseries}
{.}
{.5em}
{\thmname{#1}\thmnumber{ #2}\thmnote{ (#3)}}

\newtheoremstyle{DEFINITION}
{2.5pt}
{2pt}
{}
{}
{\bfseries}
{.}
{.5em}
{\thmname{#1}\thmnumber{ #2}\thmnote{ (#3)}}

\newtheoremstyle{EXERCISE}
{2pt}
{2pt}
{}
{}
{\scshape}
{.}
{.5em}
{\thmname{#1}\thmnumber{ #2}\thmnote{ (#3)}}

\theoremstyle{THEOREM}
\newtheorem{theorem}{Theorem}[section]
\newtheorem{lemma}[theorem]{Lemma}
\newtheorem{proposition}[theorem]{Proposition}
\newtheorem{corollary}[theorem]{Corollary}
\newtheorem{conjecture}[theorem]{Conjecture}

\theoremstyle{DEFINITION}

\theoremstyle{EXERCISE}
\newtheorem{remark}[theorem]{Remark}

\numberwithin{equation}{section}

\makeatletter
\renewenvironment{proof}[1][\proofname]{\par
  \vspace{-6pt}
  \pushQED{\qed}
  \normalfont \topsep6\p@\@plus6\p@\relax
  \trivlist
  \item[\hskip\labelsep\rmfamily\bfseries
    #1\@addpunct{:}]\ignorespaces
}{
  \popQED\endtrivlist\@endpefalse
  \vspace{-6pt}
}
\makeatother



\begin{document}
\setlength \lineskip{3pt}
\setlength \lineskiplimit{3pt}
\setlength \parskip{1pt}
\setlength \partopsep{0pt}

\title{Beyond endoscopy for $\GL_2$ over $\QQ$ with ramification 4: contribution of non-elliptic parts}
\author{Yuhao Cheng}
\address{Qiuzhen College, Tsinghua University, 100084, Beijing, China}
\email{chengyuhaomath@gmail.com}
\date{20 May, 2026}
\begin{abstract}
We continue our work on $\GL_2$ over $\QQ$ in the ramified setting for \emph{Beyond Endoscopy}. We establish asymptotic formulas for each term of the trace formula when summing over $n<X$, using arbitrary smooth test functions at the places in $S=\{\infty,q_1,\dots, q_r\}$ where $2\in S$, for the standard representation, up to an error of $o(X)$. This yields an identity depending on a parameter $X$, leading to certain identities that can be regarded as a limit form of the trace formula for $\GL_2$ over $\QQ$. On the spectral side, we employ the contour shift method and the Riemann-Lebesgue lemma. On the geometric side, both the identity part and the unipotent part contribute $o(X)$. The elliptic part was reduced to the hyperbolic part in the previous paper. Finally, using hyperbolic Poisson summation, we relate the hyperbolic part back to the spectral side and determine its contribution.
\end{abstract}

\maketitle
 
\tableofcontents

\section{Introduction}
\subsection{General Philosophy} 
The \emph{Beyond Endoscopy} program was proposed by Langlands near 2000 \cite{langlands2004}. Roughly speaking, it is a two-step process. The first one is to determine whether a representation $\pi$ of a quasi-split algebraic group $\G$ comes from a representation $\sigma$ of a "smaller" group $\H$ via the $L$-group morphism $\phi\colon \prescript{L}{}\H\to \prescript{L}{}\G$. Moreover, we want to find the smallest $\H$ (which we call \emph{primitive}). The second step is to compare these data for different groups.

For simplicity we work on the field $\QQ$ of rational numbers.
One way to study the first step is to consider the order of $L$-functions. Let $\rho$ be a complex representation of $^L\G$. For a finite set $S$ of places of $\QQ$ containing the archimedean place $\infty$, the partial $L$-function $L^{S}(s,\pi,\rho)$ can be defined and expressed as a Dirichlet series
\[
L^{S}(s,\pi,\rho)=\sum_{\substack{n=1\\ \gcd(n,S)=1}}^{+\infty}\frac{a_{\pi,\rho}(n)}{n^s}
\]
for $\Re s$ sufficiently large.

If $\pi$ is induced by $\sigma$ for some automorphic representation $\sigma$ of a $\H$ via $\phi\colon \prescript{L}{}\H\to \prescript{L}{}\G$, then we expect to have
\[
\ord_{s=1}L^S(s,\pi,\rho)=\ord_{s=1}L^S(s,\sigma,\rho\circ\phi)\geq m(\triv,\rho\circ\phi)
\]
and the equality holds when $\H$ is primitive, where $m(\triv,\rho\circ\phi)$ denotes the multiplicity of the trivial representation in $\rho\circ\phi$, which is a complex representation of $^L\H$.

Unfortunately, we know little about the analytic continuation and the order of pole at $s=1$. However, by using the trace formula, we are able to compute the average 
\begin{equation}\label{eq:beyondendoscopy}
\sum_{\pi}m_{\pi}\ord_{s=1}L^{S}(s,\pi,\rho)\prod_{v\in S}\Tr(\pi_v(f_v))
\end{equation}
in a reasonable way.
Here, $\pi$ runs over all cuspidal representations and $f_v$ are nice functions on $\G(\QQ_v)$ for each $v\in S$. $m_\pi$ denotes the multiplicity of $\pi$ in $L^2_{\mathrm{cusp}}(\G(\QQ)\bs \G(\AA)^1)$.
We have
\[
\sum_{\pi}m_\pi a_{\pi,\rho}(p)\prod_{v\in S}\Tr(\pi_v(f_v))=I_{\mathrm{cusp}}(f^{p,\rho}),
\]
where $f^{p,\rho}=\bigotimes_{v\in S}f_v\otimes\bigotimes_{v\notin S}'f_v^{p,\rho}$, with $f_v^{p,\rho}$ spherical for $v\notin S$, and $\Tr(\pi_p(f_p^{p,\rho}))=a_{\pi,\rho}(p)$  and $\Tr(\pi_\ell(f_\ell^{p,\rho}))=1$ if $\ell\notin S\cup\{p\}$.

If $L^S(s,\pi,\rho)$ admits meromorphic continuation on $\Re s>1-\delta$ with neither zeros nor poles on the vertical line $s=1+\rmi t$ except at $s=1$, then by the Ikehara theorem, we expect the following asymptotic formula:
\[
\lim_{X\to +\infty}\frac{1}{X}\sum_{\substack{p<X\\ p\notin S}}a_{\pi,\rho}(p)\log p=\ord_{s=1}L^S(s,\pi,\rho),
\]
where $p$ runs over all primes less than $X$.

If such asymptotic formula holds, \eqref{eq:beyondendoscopy} can be rewritten as
\[
\lim_{X\to +\infty}\frac{1}{X}\sum_{\pi}m_\pi\prod_{v\in S}\Tr(\pi_v(f_v))\sum_{\substack{p<X\\ p \notin S}}a_{\pi,\rho}(p)\log p .
\] 
and thus we expect that 
\[
\lim_{X\to +\infty}\frac{1}{X}\sum_{\substack{p<X\\ p \notin S}}\log pI_{\mathrm{cusp}}(f^{p,\rho})=\sum_{\pi}m_\pi \prod_{v\in S}\Tr(\pi_v(f_v))\ord_{s=1}L^S(s,\pi,\rho).
\]

We can consider a more general setting, which was proposed by Sarnak \cite{sarnak2001}. For $\gcd(n,S)=1$, we have
\[
\sum_{\pi}m_\pi a_{\pi,\rho}(n)\prod_{v\in S}\Tr(\pi_v(f_v))=I_{\mathrm{cusp}}(f^{n,\rho}),
\]
where $f^{n,\rho}=\bigotimes_{v\in S}f_v\otimes\bigotimes_{v\notin S}'f_v^{n,\rho}$, with $f_v^{n,\rho}$ spherical for $v\notin S$, and $\Tr(\pi_p(f_p^{n,\rho}))=a_{\pi,\rho}(p^{v_p(n)})$ for all $p\notin S$. By the Ikehara theorem, we expect to have
\[
\lim_{X\to +\infty}\frac{1}{X}\sum_{\substack{n<X\\ \gcd(n,S)=1}}a_{\pi,\rho}(n)=\res_{s=1}L^S(s,\pi,\rho),
\]
and we expect that
\[
\lim_{X\to +\infty}\frac{1}{X}\sum_{\substack{n<X\\ \gcd(n,S)=1}}I_{\mathrm{cusp}}(f^{n,\rho})=\sum_{\pi}m_\pi \prod_{v\in S}\Tr(\pi_v(f_v))\res_{s=1}L^S(s,\pi,\rho),
\]
where $\pi$ runs over certain representations of $\G(\AA)$ depending on $f^{n,\rho}$. For example, if $\G=\GL_2$ and $\rho$ is the standard representation, then $\res_{s=1}L^S(s,\pi,\rho)=0$ for all cuspidal $\pi$. Hence the expectation is
\begin{equation}\label{eq:standardrepresentation}
\lim_{X\to +\infty}\frac{1}{X}\sum_{\substack{n<X\\ \gcd(n,S)=1}}I_{\mathrm{cusp}}(f^{n,\rho})=0.
\end{equation}

We now assume that $\G=\GL_2$ and let $\rho=\mathrm{Sym}^1$ be the standard representation. 
For any prime number $p$ and $m\in \ZZ_{\geq 0}$, we define
\[
\cX_p^{m}=\{X\in \M_2(\ZZ_p)\,|\, \mathopen{|}\det X\mathclose{|}_p = p^{-m}\}.
\]
For example, if $m=0$, $\cX_p^{m}$ is just $\cK_p=\GL_2(\ZZ_p)$. By Hecke operator theory, we can choose $f^{n}$ to be $\bigotimes_{v\in \mf{S}}f^{n}_v$, where $\mf{S}$ denotes the set of places of $\QQ$, and
\begin{enumerate}[itemsep=0pt,parsep=0pt,topsep=2pt,leftmargin=0pt,labelsep=3pt,itemindent=9pt,label=\textbullet]
  \item If $v=p$, $f^{n}_p=p^{-n_p/2}\triv_{\cX_p^{n_p}}$, where $n_p=v_p(n)$.
  \item If $v=\infty$,  $f^{n}_\infty=f_\infty\in C^\infty(Z_+\bs \G(\RR))$ such that the orbital integrals are compactly supported modulo $Z_+$, and other than this condition they are arbitrary.
\end{enumerate}
In this case, $\pi$ runs over all unramified cuspidal representations. Venkatesh \cite{venkatesh2004} established an asymptotic formula for the residue case for $k\leq 2$, using the Petersson-Kuznetsov trace formula. As a new approach, Altu\u{g} \cite{altug2015,altug2017,altug2020} proved \eqref{eq:standardrepresentation} by using the Arthur-Selberg trace formula for $f_{\infty,m}$ that is a matrix coefficient of a weight $m$ discrete series for $m>2$ even.  
Altu\u{g} actually proved that
\[
\sum_{n<X}\Tr(T_m(n))\ll_{m,\varepsilon} X^{\frac{31}{32}+\varepsilon},
\]
where $T_m(n)$ denotes the (normalized) $n^{\mathrm{th}}$ Hecke operator acting on $S_m(\SL_2(\ZZ))$, the space of holomorphic cusp forms of weight $m$ for the modular group $\Gamma=\SL_2(\ZZ)$.

\subsection{Main results in this paper and proof strategy}
In this paper, we will approach the formula \eqref{eq:standardrepresentation} in the case over $\QQ$ with ramification and $\rho=\mathrm{Std}$ the standard representation.
We consider $S=\{\infty,q_1,\dots,q_r\}$ for primes $q_1,\dots,q_r$ such that $2\in S$, and a corresponding function $f^{n}=\bigotimes_{v\in \mf{S}}'f^n_v$ such that the local components at places in $S$ are arbitrary. 
Specifically, for $\gcd(n,S)=1$, the function $f^n$ is defined as follows:
\begin{enumerate}[itemsep=0pt,parsep=0pt,topsep=2pt,leftmargin=0pt,labelsep=3pt,itemindent=9pt,label=\textbullet]
  \item If $v=p\notin S$, we choose $f^n_v=p^{-n_p/2}\triv_{\cX_p^{n_p}}$, where $n_p=v_p(n)$.
  \item If $v=q_i\in S$ and is a prime, we choose $f^n_v=f_{q_i}$ to be an arbitrary function in $C_c^\infty(\G(\QQ_{q_i}))$.
  \item If $v=\infty$, we choose $f^n_v=f_\infty\in C_c^\infty(Z_+\bs \G(\RR))$. 
\end{enumerate}
Note that we need $f_\infty\in C_c^\infty(Z_+\bs \G(\RR))$, which was explained in \cite{cheng2025b}.

In this case, $f_v^{n}$ spherical for $v\notin S$, and $\Tr(\pi_p(f_p^{n}))=a_{\pi}(p^{n_p})$ for all $p\notin S$.

To prove \eqref{eq:standardrepresentation}, it suffices to give an asymptotic formula of 
\[
\sum_{\substack{n<X\\ \gcd(n,S)=1}}I_\cusp(f^n)
\]
and show that it is $o(X)$. Using Arthur-Selberg trace formula we may split $I_\cusp(f^n)$ into various terms.

The main result provides asymptotic formulas for all the terms that occur in the trace formula. (See \autoref{sec:finalresult} for the final result.)

We first consider the spectral side, which consists of discrete terms and a continuous term. For the discrete terms, we apply classical methods from analytic number theory. For the continuous term, we use the Perron formula along with the contour shift method (or the Riemann-Lebesgue lemma) to estimate its contribution.

Next, we consider the elliptic part. By Theorem 1.1 of \cite{cheng2025c}, the estimation of the elliptic part reduces to estimating the contributions from the modified hyperbolic parts. Moreover, it is easily seen that the identity part and the unipotent part make no contribution to the asymptotic formula.

Finally, we deal with the various hyperbolic parts by using hyperbolic Poisson summation (see \autoref{sec:hyperbolicpoisson}) to relate them to the continuous term. Then, we apply the contour shift method as in the estimation of the spectral side to obtain the hyperbolic contribution.

\eqref{eq:standardrepresentation} is a known result (see, for example, \cite[Theorem 11.7.1]{getz2024}). Hence, by comparing the coefficients of $X$, $X\log X$, and $X^{3/2}$, we obtain several identities which we call the \emph{limit form of the trace formula}. The equalities for the coefficients of $X^{3/2}$ and $X\log X$ are trivial. However, for the coefficient of $X$ it seems highly nontrivial. See \autoref{thm:limittraceformula} and the theorems in \autoref{sec:limittracelocal} for the precise versions. Conversely, if we prove the limiting form of the trace formula directly, we will get a new proof of \eqref{eq:standardrepresentation}. More precisely, what we will prove in this paper is that
\[
\eqref{eq:standardrepresentation} \Leftrightarrow \autoref{thm:limittraceformula} \Leftrightarrow \autoref{thm:limittraceformulalocal1} +\autoref{thm:limittraceformulaarchimedean}+\autoref{thm:limittraceformulalocal2}.
\]

\subsection{Notations}
\begin{enumerate}[itemsep=0pt,parsep=0pt,topsep=0pt,leftmargin=0pt,labelsep=3pt,itemindent=9pt,label=\textbullet]
  \item $\# X$ denotes the number of elements in a set $X$.
  \item For $A\subseteq X$, $\triv_A$ denotes the characteristic function on $X$, defined by $\triv_A(x)=1$ for $x\in A$ and $\triv_A(x)=0$ for $x\notin A$.
  \item $\triv$ also denotes the trivial character or the trivial representation.
  \item For $x\in \RR$, $\lfloor x\rfloor$ denotes the greatest integer that is less than or equal to $x$, $\lceil x\rceil$ denotes the smallest integer that is greater than or equal to $x$.
  \item We often use the notation $a\equiv b\,(n)$ to denote $a\equiv b\pmod n$.
  \item If $R$ is a ring (which we \emph{always} assume to be commutative with $1$), $R^\times$ denotes its group of units.
  \item $\mf{S}$ denotes the set of places of $\QQ$.
  \item For $S=\{\infty, q_1,\dots,q_r\}$, $\ZZ^S$ denotes the ring of $S$-integers in $\QQ$. That is
    \[
    \ZZ^S=\{\alpha\in \QQ\ |\ v_p(\alpha)\geq 0\ \text{for all}\ p\notin S\}.
    \]
    Additionally, we define
    \[
    \QQ_S=\prod_{v\in S}\QQ_v=\RR\times \QQ_{q_1}\times\dots\times\QQ_{q_r}\quad\text{and}\quad\QQ_{S_\fin}=\prod_{v\in S_\fin}\QQ_v=\QQ_{q_1}\times\dots\times\QQ_{q_r}
    \]
  and
    \[
    \ZZ_{S_\fin}=\prod_{v\in S_\fin}\ZZ_v=\ZZ_{q_1}\times\dots\times\ZZ_{q_r}.
    \]
  \item For $n\in \ZZ$, we write $\gcd(n,S)=1$, or $n\in \ZZ_{(S)}$, if 
  $p\nmid n$ for all $p\in S$. 
  We write $n\in \ZZ_{(S)}^{>0}$ if additionally $n>0$.
  \item Let $p$ be a prime. For $a\in\QQ$, we define $a_{(p)}$ to be the $p$-part of $a$, that is, $a_{(p)}=p^{v_p(a)}$, and we define $a^{(p)}=a/a_{(p)}$. Moreover, we define 
  \[
    a_{(q)}=\prod_{i=1}^{r}q_i^{v_{q_i}(a)}\quad\text{and}\quad a^{(q)}=\frac{a}{a_{(q)}}=\prod_{p\notin S}p^{v_{p}(a)}.
  \]
  \item For $n\in \ZZ_{>0}$, ${\bm d}(n)$ denotes the number of divisors of $n$, ${\bm \phi}(n)$ denotes the Euler totient function. $\bm{\mu}(n)$ denotes the M\"obius function, defined by
  \[
  \bm{\mu}(n)=\begin{cases}
    1, & \text{if $n=1$}, \\
    (-1)^m, & \text{if $n$ is a product of $m$ distinct primes}, \\
    0, & \text{otherwise}.
  \end{cases}
  \]
  $\bm{\Lambda}(n)$ denotes the von Mangoldt function, defined by
  \[
  \bm{\Lambda}(n)=\begin{cases}
    \log p, & \text{if $n=p^m$ with $m\in \ZZ_{>0}$}, \\
    0, & \text{otherwise}.
  \end{cases}
  \]
  \item $\Gamma(s)$ denotes the gamma function, defined by
  \[
  \Gamma(s)=\int_{0}^{+\infty}\rme^{-x}x^s\frac{\rmd x}{x},
  \]
  for $\Re s>0$, analytically continued to $\CC$. $\zeta(s)$ denotes the Riemann zeta function, defined by
  \[
    \zeta(s)=\sum_{n=1}^{+\infty}\frac{1}{n^s}=\prod_{p}\frac{1}{1-p^{-s}}
  \]
  for $\Re s>1$, analytically continued to $\CC$.
  \item For any Dirichlet character $\chi$, we define
  \[
  L(s,\chi)=\sum_{n=1}^{+\infty}\frac{\chi(n)}{n^s}=\prod_{p}\frac{1}{1-\chi(p)p^{-s}}.
  \]
  for $\Re s>1$, analytically continued to $\CC$. Moreover, we define
  \[
  L^S(s,\chi)=\sum_{\substack{n=1\\\gcd(n,S)=1}}^{+\infty}\frac{\chi(n)}{n^s}=\prod_{p\notin S}\frac{1}{1-\chi(p)p^{-s}}.
  \]
  for $\Re s>1$, analytically continued to $\CC$. Finally, we set
  \[
  \zeta^S(s)=L^S(s,\triv)=\sum_{\substack{n=1\\\gcd(n,S)=1}}^{+\infty}\frac{1}{n^s}=\prod_{p\notin S}\frac{1}{1-p^{-s}}=\prod_{i=1}^{r}(1-q_i^{-s})\zeta(s).
  \]
  \item $(\sigma)$ denotes the vertical contour from $\sigma-\rmi\infty$ to $\sigma+\rmi\infty$.
  \item Let $f$ be a meromorphic function near $z=z_0$ and suppose that 
  \[
  f(z)=\sum_{n\in \ZZ}a_n(z-z_0)^n
  \]
  is its Laurent expansion near $z_0$, we denote $\fp_{z=z_0}f(z)=a_0$.
  \item We use $f(x)=O(g(x))$ or $f(x)\ll g(x)$ to denote that there exists a constant $C$ such that $|f(x)|\leq C|g(x)|$ for all $x$ in a specified set. If the constant depends on other variables, they will be subscripted under $O$ or $\ll$. 
  \item We use $f(x)=o(g(x))$ to denote $f(x)/g(x)\to 0$ as $x$ tends to a certain limit.
  \item The notation $f(x)\asymp g(x)$ indicates that $f(x)\ll g(x)$ and $g(x)\ll f(x)$. If the constant depends on other variables, they will be subscripted under $\asymp$.
\end{enumerate}

\section{Preliminaries}\label{sec:preliminaries}
In this section, we recall relevant definitions and results from previous work \cite{cheng2025,cheng2025b,cheng2025c}. To avoid notational conflicts between \cite{cheng2025b} and \cite{cheng2025c}, we adopt the conventions of the latter. Readers familiar with these papers may skip this section.

\subsection{The Arthur-Selberg trace formula for $\GL_2$}\label{subsec:traceformula}
Let $\mf{S}$ be the set of places of $\QQ$ and let $\AA$ be the ad\`ele ring of $\QQ$. The Arthur-Selberg trace formula of $\G=\GL_2$ takes the following form:
\begin{equation}\label{eq:traceformulagl2}
J_{\mathrm{geom}}(f)=J_{\mathrm{spec}}(f),
\end{equation}
where the geometric side is
\begin{equation}\label{eq:geometric}
J_{\mathrm{geom}}(f)=I_\id(f)+I_\el(f)+J_\hyp(f)+J_\unip(f)
\end{equation}
and the spectral side is
\begin{equation}\label{eq:spectral}
J_{\mathrm{spec}}(f)=I_{\mathrm{cusp}}(f)+J_{\mathrm{cont}}(f)+\sum_{\mu}\Tr(\mu(f))+\frac14\sum_{\mu}\Tr(M(0,\mu)(\xi_0\otimes\mu)(f)).
\end{equation}
The sum $\mu=\mu_1\boxtimes\mu_2$ runs over all $1$-dimensional representations of $\A(\QQ)\bs \A(\AA)^1= \QQ^\times\bs (\AA^\times)^1\times \QQ^\times\bs (\AA^\times)^1$ with $\mu_1=\mu_2$ (equivalently, runs over all $1$-dimensional representations of $\G(\QQ)\bs \G(\AA)^1$), where $\A$ denotes the diagonal torus of $\G$.

The Arthur-Selberg trace formula appears in numerous references. For example, \cite[Chapter 16]{langlands1970}, \cite[Theorem 6.33]{gelbart1979}, \cite[Theorem 7.14]{knapp1997} and \cite[Theorem 1]{finis2011}. 

We first explain the geometric side. The identity part is
\[
I_\id(f)=\sum_{z\in \Z(\QQ)}\vol(\G(\QQ)\bs\G(\AA)^1)f(z).
\]

The elliptic part is
\[
I_\el(f)=\sum_{\gamma\in \G(\QQ)^\#_\el}\vol(\gamma)\orb(f;\gamma),
\]
where $ \G(\QQ)^\#_\el$ denotes the set of elliptic conjugacy classes in $\G(\QQ)$, and
\[
\orb(f;\gamma)=\int_{\G_\gamma(\AA)\bs \G(\AA)} f(g^{-1}\gamma g)\rmd g,
\qquad
\vol(\gamma)=\int_{Z_+\G_\gamma(\QQ)\bs \G_\gamma(\AA)} \rmd g,
\]
where $Z_+$ denotes the connected component of the identity matrix in the center $\Z(\RR)$ of $\G(\RR)$.

The hyperbolic part is
\[
J_\hyp(f)=-\frac{1}{2}\sum_{\gamma\in \A(\QQ)_{\reg}}\int_{\A(\AA)\bs \G(\AA)}f(g^{-1}\gamma g)\alpha(H_\B(wg)+H_\B(g))\rmd g,
\]
where $\A$ is the diagonal torus, $\A(\QQ)_\reg$ denotes all the regular elements in $\A(\QQ)$, $w$ denotes the nontrivial element in the Weyl group of $(\G,\A)$, $\alpha$ denotes the positive root in $\mf{sl}_2$, and $H_\B$ denotes the Harish-Chandra map. 

The unipotent part is
\[
J_\mathrm{unip}(f)=\sum_{z\in \Z(\QQ)}\fp_{s=1}Z(s,\triv,F_z),
\]
where $\fp$ denotes the finite part, $Z$ denotes the zeta integral in Tate's thesis and for $t\in \AA$,
\[
F_z(t)=\int_{K}f\left(k^{-1}z\begin{pmatrix}
                                 1 & t \\
                                 0 & 1 
                               \end{pmatrix}k\right)\rmd k,
\]
where $K$ denotes the standard maximal compact subgroup of $\G(\AA)$.

Next we provide an explicit description of the spectral side.
We follow the notations of \cite{gelbart1979}. The continuous part is the sum of the two following terms
\begin{equation}\label{eq:continuouspart1}
-\frac{1}{4\uppi\rmi}\sum_{\mu}\int_{(0)}\frac{m'(s,\mu)}{m(s,\mu)} \Tr\left(\Ind_{\B(\AA)}^{\G(\AA)}(s,\mu)(f)\right)\rmd s
\end{equation}
and
\begin{equation}\label{eq:continuouspart2}
-\frac{1}{4\uppi\rmi}\sum_{v\in \mf{S}}\sum_{\mu}\int_{(0)}\Tr \left(R_v(s,\mu_v)^{-1}R_v'(s,\mu_v)\Ind_{\B(\QQ_v)}^{\G(\QQ_v)}(s,\mu_v)(f_v)\right)\cdot \prod_{w\neq v}\Tr\left(\Ind_{\B(\QQ_w)}^{\G(\QQ_w)}(s,\mu_w)(f_w)\right)\rmd s,
\end{equation}
where $\mu=\mu_1\boxtimes\mu_2$ runs over all $1$-dimensional representations of $\A(\QQ)\bs \A(\AA)^1$. We do \emph{not} require $\mu_1=\mu_2$ for the sum in the continuous part.

The space of induced representations is defined as follows:
As a vector space, $\Ind_{\B(\QQ_v)}^{\G(\QQ_v)}(s,\mu_v)$ is the space of all smooth functions $\psi$ on $\G(\QQ_v)$ that satisfy
\[
\psi(bg)=\mu_1(x)\mu_2(y)\left|\frac{x}{y}\right|_v^{s+1/2}\psi(g)
\] 
for any $g\in \G(\QQ_v)$ and $b=(\begin{smallmatrix} x & z \\ 0 & y \end{smallmatrix})\in \B(\QQ_v)$, and $\G(\QQ_v)$ acts by right translation. The global induced representation is defined similarly. 
$\mu=\mu_1\boxtimes\mu_2$ runs over all $1$-dimensional representations of $\A(\QQ)\bs \A(\AA)^1$, and
\[
m(s,\mu)=\prod_{v}m_v(s,\mu_v)=\prod_{v}\frac{L_v(2s,\mu_1/\mu_2)}{L_v(1+2s,\mu_1/\mu_2) \varepsilon_v(1-2s,\mu_2/\mu_1,\rme_v)} =\frac{\Lambda(1-2s,\mu_2/\mu_1)}{\Lambda(1+2s,\mu_1/\mu_2)},
\]
where $\varepsilon_v$ denotes the $\varepsilon$-factor.
$M_v(s,\mu_v)$ denotes the intertwining operator from the space $\Ind_{\B(\QQ_v)}^{\G(\QQ_v)}(s,\mu_v)$ to $\Ind_{\B(\QQ_v)}^{\G(\QQ_v)}(-s,w(\mu_v))$, where $w(\mu)=\mu_2\boxtimes\mu_1$ if $\mu=\mu_1\boxtimes\mu_2$, and
\[
M(s,\mu)=\prod_{v\in \mf{S}}M_v(s,\mu_v).
\] 
$R_v(s,\mu_v)$ is the normalized intertwining operator 
\[
R_v(s,\mu_v)=m_v(s,\mu)^{-1}M_v(s,\mu_v).
\]
If $s=0$ and $\mu$ is trivial, we also denote this representation by $\xi_0$.

\subsection{Measure normalizations}
We will give the measure normalizations of $\G_\gamma(\QQ_v)$ and $\G(\QQ_v)$, where $\gamma$ is a regular semisimple element in $\G(\QQ)$.
Let $Z_+$ be the connected component of the identity matrix in the center $\Z(\RR)$ of $\G(\RR)$. The measures of $\G_\gamma(\QQ_v)$ and $\G(\QQ_v)$ are normalized as follows. For $\G(\QQ_v)$, 
\begin{enumerate}[itemsep=0pt,parsep=0pt,topsep=0pt,leftmargin=0pt,labelsep=3pt,itemindent=9pt,label=\textbullet]
  \item If $v=p$ is a prime, we normalize the Haar measure on $\G(\QQ_p)$ such that the volume of $\G(\ZZ_p)$ is $1$.
  \item If $v=\infty$, we choose the Haar measure on $\G(\RR)$ that is compatible with the Iwasawa decomposition $\G(\RR)=\B(\RR)K$, where $K=\SO(2)$. (The measure on $K$ can be chosen arbitrarily.)
\end{enumerate}

By Proposition 2.1 of \cite{cheng2025}, for elliptic $\gamma$, there exists a quadratic extension $E/\QQ$ such that $\G_\gamma\cong\Res_{E/\QQ}\GL_1$. For hyperbolic $\gamma$, we may take $E=\QQ\times \QQ$ so that such identification still holds. Hence $\G_\gamma(\QQ_v)\cong E_v^\times$.
\begin{enumerate}[itemsep=0pt,parsep=0pt,topsep=0pt,leftmargin=0pt,labelsep=3pt,itemindent=9pt,label=\textbullet]
  \item If $v=p$ is a prime, we know that $E_p$ is either $\QQ_p^2$ or a quadratic extension of $\QQ_p$. In the first case, we let $\G_\gamma(\ZZ_p)=\ZZ_p^\times \times \ZZ_p^\times$. In the second case, we let $\G_\gamma(\ZZ_p)=\cO_{E_p}^\times$.
  In each case, we normalize the Haar measure on $\G_\gamma(\QQ_p)=E_p^\times$ such that the volume of $\G_\gamma(\ZZ_p)$ is $1$.
  \item If $v=\infty$, we know that $E_v$ is either $\RR^2$ or $\CC$. If $E_v=\RR^2$, we define the measure on $E_v^\times=\RR^\times \times \RR^\times$ as $\rmd x\rmd y/|xy|$, where $(x,y)$ is the coordinate of $\RR^2$. The action of $Z_+$ on $E_v$ is $a\cdot(x,y)=(ax,ay)$, and we define the measure on $Z_+\bs E_v$ to be the measure of the quotient of $E_v$ by the measure $\rmd a/a$ on $Z_+$. If $E_v=\CC$, we choose the  measure on $E_v^\times=\CC^\times$ to be $2\rmd r\rmd \theta/r$, where we use the polar coordinate $(r,\theta)$ on $\CC^\times$. The action of $Z_+$ on $E_v$ is $a\cdot z=az$, and we define the measure on $Z_+\bs E_v$ to be the measure of the quotient of $E_v$ by the measure $\rmd a/a$ on $Z_+$.
\end{enumerate} 
\begin{remark}
These measure normalizations are taken from \cite{finis2011} which is used in the previous papers \cite{cheng2025,cheng2025b,cheng2025c}, instead of that in \cite{langlands2004}. This is not important for the previous three papers, but is crucial in this paper. 
\end{remark}

\subsection{The modified $p$-adic norm}\label{subsec:modifiednorm}
For any prime $p$, the \emph{modified norm} $|\cdot|_p'$ is defined as follows: 

For $p\neq 2$ and $y\in \QQ_p$, we define $|y|_p' = p^{-2\lfloor v_p(y)/2\rfloor}$.  
This satisfies $|y|_p' = |y|_p$ if $v_p(y)$ is even, and $|y|_p' = p |y|_p$ if $v_p(y)$ is odd.

For $p=2$ and $y\in \QQ_p$, we define
\[
|y|_p' = \begin{cases}
  p^{-v_p(y)}=|y|_p & \text{if $v_p(y)$ is even, $y_0\equiv 1\,(4)$}, \\
  p^{-v_p(y)+2}=p^2|y|_p & \text{if $v_p(y)$ is even, $y_0\equiv 3\,(4)$}, \\
  p^{-v_p(y)+3}=p^3|y|_p & \text{if $v_p(y)$ is odd},
\end{cases}
\]
where $y_0=yp^{-v_p(y)}$.
Clearly $|a^2y|_p'=|a|_p^2|y|_p'$ for any $a\in \QQ_p$.

Also, for any regular element $\gamma\in\G(\QQ_p)$, we denote $T_\gamma=\Tr\gamma$ and $N_\gamma=\det\gamma$. $k_\gamma$ is defined such that $p^{k_\gamma}=|T_\gamma^2-4N_\gamma|_p'^{-1/2}$. This coincides with the original definition of \cite{cheng2025} (see Proposition 2.6 of loc. cit.).

\subsection{Singularities of the orbital integrals}\label{subsec:singularities}
We follow the notations in \cite{cheng2025c}. Define
\[
\omega_\infty(x)=\begin{cases}
             0, & x>0, \\
             1, & x<0
           \end{cases}
\]
for $x\in \RR$ with $x\neq 0$ and
\[
\omega_p(y)=\legendresymbol{y|y|_{p}'}{p}
\]
for prime $p$ and $y\in \QQ_p$ with $y\neq 0$. When $p=q_i$, we also write $\omega_p=\omega_i$.
For $\iota\in \{0,1\}$ we define
\[
X_{\iota}=\{x\in \RR\ |\ \omega_\infty(x)=\iota\}
\]
and for $\epsilon_i\in \{0,\pm 1\}$, we define
\[
Y_{\epsilon_i}=\{y_i\in \QQ_{q_i}\ |\ \omega_i(y_i)=\epsilon_i\}. 
\]

For any prime $p\in S$, we define
\[
\theta_{p}(\gamma)=\frac{1}{\mathopen{|}\det\gamma\mathclose{|}_p^{1/2}}\left(1-\frac{\chi(p)}{p}\right)^{-1}p^{-{k_\gamma}}\orb(f_{p};\gamma),
\]
where $\chi(p)=\omega_p(T_\gamma^2-4N_\gamma)$. By Corollary 2.12 of \cite{cheng2025}, the local behavior of $\theta_{p}$ at $z=aI$ is
\begin{equation}\label{eq:shalikalocal}
\theta_{p}(\gamma)=\lambda_1\left(1-\frac{\chi(p)}{p}\right)^{-1}p^{-{k_\gamma}} \frac{1-\chi(p)}{1-p}+\lambda_2.
\end{equation}

Clearly $\theta_{p}(\gamma)$ is invariant under conjugation. Thus $\theta_{p}(\gamma)$ can be parametrized by $T=\Tr\gamma$ and $N=\det\gamma$, i.e., $\theta_{p}(\gamma)=\theta_p(T,N)$. 
Since $\theta_p(\gamma)$ is smooth away from the center, $\theta_p(T,N)$ is smooth except at $T^2=4N$. 

For the archimedean orbital integral, we recall the following theorem (cf. \cite[Theorem 2.13]{cheng2025}).
\begin{theorem}\label{thm:archimedeanintegral}
For any $f_\infty\in C^\infty(\G(\RR))$, any maximal torus $\T(\RR)$ in $\G(\RR)$ and any $z$ in the center of $\G(\RR)$, there exists a neighborhood $N$ in $\T(\RR)$ of $z$ and smooth functions $g_1,g_2\in C^\infty(N)$ (depending on $f_\infty$ and $z$) such that
\begin{equation}\label{eq:archimedeanintegral}
\orb(f_\infty;\gamma)=g_1(\gamma)+\frac{|\gamma_1\gamma_2|^{1/2}}{|\gamma_1-\gamma_2|}g_2(\gamma)
\end{equation}
for any $\gamma\in \T(\RR)$, where $\gamma_1$ and $\gamma_2$ are the eigenvalues of $\gamma$.  Moreover, $g_1$ and $g_2$ can be extended smoothly to all split and elliptic elements, remaining invariant under conjugation, with $g_1(\gamma)=0$ if $\T(\RR)$ is split, and $g_2$ can further be extended smoothly to the center. If $f_\infty$ is $Z_+$-invariant, then $g_1$ and $g_2$ are also $Z_+$-invariant.
\end{theorem}

We define
\[
\theta_\infty(\gamma)=\frac{|\gamma_1-\gamma_2|}{|\gamma_1\gamma_2|^{1/2}}\orb(f_\infty;\gamma)= \frac{|\gamma_1-\gamma_2|}{|\gamma_1\gamma_2|^{1/2}}g_1(\gamma)+g_2(\gamma).
\]
Since $g_1,g_2$ and $\theta_\infty$ are invariant under conjugation, we can parametrize them by $T_\gamma$ and $N_\gamma$ as in the nonarchimedean case.

Since $T_{z\gamma}=aT_\gamma$ and  $N_{z\gamma}=a^2N_\gamma$ for $z=aI$ with $a>0$, we have $g_i(T_\gamma,N_\gamma)=g_i(aT_\gamma,a^2N_\gamma)$ and $\theta_\infty(T_\gamma,N_\gamma)=\theta_\infty(aT_\gamma,a^2N_\gamma)$ for $i=1,2$ and any $a>0$. 

Also, we set $\theta_\infty^\pm(x)=\theta_\infty(x,\pm 1/4)$, which coincides with the notation in \cite{cheng2025} and \cite{cheng2025b}.

In \cite{cheng2025c} we have the following definitions. Let
\[
\Theta_\infty^\pm(x)=\theta_\infty\left(\pm 1,\frac{1-x}{4}\right).
\]
and write $\widehat{\Theta}_\infty(x)=\Theta_\infty^+(x)+\Theta_\infty^-(x)$. Also, for any prime $p$, we define
\[
\widehat{\Theta}_p(y)=\int_{\QQ_p^\times}\theta_p\left(z,\frac{z^2(1-y)}{4}\right)\frac{\rmd z}{|z|_p}.
\]

\subsection{Global notations} For $\nu\in \ZZ^r$, we usually denote $\nu_i$ by its $i^{\mathrm{th}}$ component. We define
\[
q^\nu=q_1^{\nu_1}\dots q_r^{\nu_r}.
\]

For any $y=(y_1,\dots,y_r)\in \QQ_{S_\fin}=\QQ_{q_1}\times\dots\times\QQ_{q_r}$ and $N\in \QQ$, we define
\[
\theta_{q}(y,N)=\prod_{i=1}^{r}\theta_{q_i}(y_i,N),\qquad|y|_q'=\prod_{i=1}^{r}|y_i|_{q_i}',\qquad \rme_q(y)=\prod_{i=1}^{r}\rme_{q_i}(y_i)
\]
and
\[
\widehat{\Theta}_q(y)=\prod_{i=1}^{r}\widehat{\Theta}_{q_i}(y_i).
\]
We usually embed $\QQ$ in $\QQ_S$ or $\QQ_{S_\fin}$ diagonally. Moreover, we define
\[
\theta_q(\gamma)=\prod_{i=1}^{r}\theta_{q_i}(\gamma)
\]
for $\gamma\in \G(\QQ_S)$.

\section{Contribution of the spectral side}\label{sec:spectral}
Recall that the spectral side of the Arthur-Selberg trace formula is of the form
\[
J_\spec(f^n)=I_\cusp(f^n)+J_\cont(f^n)+\sum_{\mu}\Tr(\mu(f^n))+\frac{1}{4}\sum_{\mu}\Tr(M(0,\mu)(\xi_0\otimes\mu)(f^n)).
\]
The third term on the right hand side is called the \emph{$1$-dimensional term} or the \emph{residual part}. The fourth term on the right hand side is called the \emph{discrete part in the continuous spectrum}. We call these terms \emph{discrete}.
\subsection{Contribution of discrete terms}
In this section we will give asymptotic formulas of
\[
\sum_{\substack{n<X\\\gcd(n,S)=1}}\sum_{\mu}\Tr(\mu(f^n))\quad\text{and}\quad \sum_{\substack{n<X\\\gcd(n,S)=1}}\frac{1}{4}\sum_{\mu}\Tr(M(0,\mu)(\xi_0\otimes\mu)(f^n)),
\]
where $\mu$ runs over $1$-dimensional representations (equivalently, characters on $\QQ^\times\bs(\AA^\times)^1$). We also consider such $\mu$ as Dirichlet characters. In \cite{cheng2025b}, we have proved that the sum over $\mu$ is finite in both terms.

First we consider the $1$-dimensional term.
\begin{theorem}\label{thm:1dimestimate}
We have
\[
\sum_{\substack{n<X\\\gcd(n,S)=1}}\sum_{\mu}\Tr(\mu(f^n))=\mf{A}X^{\frac32}+O(X^{\frac12}\log X),
\]
where $\mf{A}$ is defined in \cite[Theorem 7.10]{cheng2025c}. 
\end{theorem}
\begin{proof}
This theorem is precisely the content of Theorem 7.10 in \cite{cheng2025c}.
\end{proof}

Next we give an asymptotic formula for the discrete part in the continuous spectrum. Note that such term is closely related to the \emph{Eisenstein term}
\[
\sum_{\substack{n<X\\\gcd(n,S)=1}}\frac{1}{4}\sum_{\mu}\Tr((\xi_0\otimes\mu)(f^n)).
\]
The only difference is the intertwining operator $M(0,\mu)$. The first thing we need to do is to analyze the intertwining operator $M(0,\mu)$.
\begin{proposition}\label{prop:intertwining}
We have $M(0,\mu)=s(\mu)\id$ with $s(\mu)\in \{\pm 1\}$. Moreover, $s(\triv)=-1$.
\end{proposition}
\begin{proof}
By Schur's lemma we know that $M(0,\mu)$ is a scalar. Moreover, by functional equation of the intertwining operator, $M(0,\mu)^2=\id$. Hence $M(0,\mu)=\pm\id$.

Now it suffices to prove that $M(0,\triv)=-\id$. 
Since the trivial representation is unramified at all places of $\QQ$, we have
\[
M(s,\triv)=m(s,\triv)\id=\frac{\Lambda(1-2s)}{\Lambda(1+2s)}\id
\]
for $\Re s$ sufficiently large. Thus it also holds for $s=0$ by analytic continuation. Since $\Lambda(s)$ has a simple pole at $s=1$, we obtain
\[
\left.\frac{\Lambda(1-2s)}{\Lambda(1+2s)}\right|_{s=0}=-\lim_{s\to 0}\frac{(1-2s)\Lambda(1-2s)}{(1+2s)\Lambda(1+2s)}=-1.
\] 
Hence $M(0,\triv)=-\id$.
\end{proof}

The computation of the discrete part in the continuous spectrum then reduces to computing 
\[
\sum_{\substack{n<X\\\gcd(n,S)=1}}\Tr((\xi_0\otimes\mu)(f^n)).
\]

\begin{proposition}\label{prop:unramifiedtraceeisenstein}
Suppose that $\mu$ is unramified outside $S$. Then we have
\[
\Tr((\xi_0\otimes\mu)(f^n))=\prod_{v\in S}\Tr((\xi_0\otimes\mu_v)(f_v))\bm{d}(n)\mu(n),
\]
where $\bm{d}(n)$ denotes the number of the divisors of $n$. 
\end{proposition}
\begin{proof}
By Proposition 8.4 of \cite{cheng2025} we have
\begin{align*}
  \Tr((\xi_0\otimes\mu)(f^n)) & =\prod_{v\in S}\Tr((\xi_0\otimes\mu_v)(f_v))\prod_{p\notin S}p^{-n_p/2}\mu_p(p^{n_p})p^{n_p/2}(n_p+1)\\
  &=\prod_{v\in S}\Tr((\xi_0\otimes\mu_v)(f_v))\bm{d}(n)\mu(n).\qedhere
\end{align*}
\end{proof}

\begin{lemma}
We have
\[
\sum_{\substack{n<X\\ \gcd(n,S)=1}}\frac{1}{n}=\prod_{i=1}^{r}(1-q_i^{-1})\left(\log X +\upgamma_S\right)+O(X^{-1}),
\]
where
\begin{equation}\label{eq:defgammas}
\upgamma_S=\upgamma+\sum_{i=1}^{r}\frac{q_i^{-1}\log q_i}{1-q_i^{-1}}
\end{equation}
and $\upgamma$ is the Euler-Mascheroni constant.
\end{lemma}
\begin{proof}
Let $k=q_1\cdots q_r$. Then we have
\begin{align*}
\sum_{\substack{n<X\\ \gcd(n,S)=1}}\frac{1}{n}=&\sum_{n<X}\sum_{\substack{d\mid n\\ d\mid k}}\frac{\bm{\mu}(d)}{n}=\sum_{d\mid k}\sum_{md<X}\frac{\bm{\mu}(d)}{md}=\sum_{d\mid k}\frac{\bm{\mu}(d)}{d}\sum_{m<X/d}\frac{1}{m}\\
=&\sum_{d\mid k}\frac{\bm{\mu}(d)}{d}(\log X+\upgamma)-\sum_{d\mid k}\frac{\bm{\mu}(d)}{d}\log d+O(X^{-1}),
\end{align*}
where $\bm{\mu}(n)$ denotes the M\"obius function.

It is easy to see that
\[
\sum_{d\mid k}\frac{\bm{\mu}(d)}{d}=\prod_{i=1}^{r}\left(1-\frac{1}{q_i}\right)
\]
and
\[
-\sum_{d\mid k}\frac{\bm{\mu}(d)}{d}\log d=\sum_{i=1}^{r}\frac{\log q_i}{q_i}\prod_{j\neq i}\left(1-\frac{1}{q_j}\right).
\]
Hence we obtain the desired result.
\end{proof}

\begin{proposition}\label{prop:dnestimate}
Let $\chi$ be a Dirichlet character such that $\chi(m)=0$ if and only if $\gcd(m,S)\neq 1$. Then
\[
\sum_{n<X}\bm{d}(n)\chi(n)=\delta(\chi)\prod_{i=1}^{r}(1-q_i^{-1})^2(X\log X+\left(2\upgamma_S-1\right)X)+O(X^{1/2}),
\]
where $\delta(\chi)=1$ if $\chi$ is trivial, and  $\delta(\chi)=0$ otherwise.
\end{proposition}
\begin{proof}
It is just a consequence of the Dirichlet hyperbola method by using the above lemma. The details are left to the reader.
\end{proof}

\begin{corollary}\label{cor:eisensteinasymptotic}
We have
\begin{align*}
&\sum_{\substack{n<X\\\gcd(n,S)=1}}\frac{1}{4}\sum_{\mu}\Tr(M(0,\mu)(\xi_0\otimes\mu)(f^n))\\
=&-\frac{1}{4}\prod_{v\in S}\Tr(\xi_0(f_v))\prod_{i=1}^{r}(1-q_i^{-1})^2(X\log X+\left(2\upgamma_S-1\right)X)+O(X^{1/2}),
\end{align*}
where the implied constant only depends on $f_\infty$ and $f_{q_i}$.
\end{corollary}
\begin{proof}
Fix $\mu$, we consider
\begin{equation}\label{eq:intertwiningestimate}
\sum_{\substack{n<X\\\gcd(n,S)=1}}\frac{1}{4}\Tr(M(0,\mu)(\xi_0\otimes\mu)(f^n)).
\end{equation}
By \autoref{prop:intertwining} we have
\[
\Tr(M(0,\mu)(\xi_0\otimes\mu)(f^n))=s(\mu)\Tr((\xi_0\otimes\mu)(f^n)).
\]
Hence by \autoref{prop:unramifiedtraceeisenstein} and \autoref{prop:dnestimate} we obtain
\begin{align*}
\eqref{eq:intertwiningestimate}=&\frac{1}{4}s(\mu)\sum_{\substack{n<X\\\gcd(n,S)=1}} \Tr((\xi_0\otimes\mu)(f^n))=\frac{1}{4}s(\mu)\prod_{v\in S}\Tr((\xi_0\otimes\mu_v)(f_v)) \sum_{\substack{n<X\\\gcd(n,S)=1}} \bm{d}(n)\mu(n)\\
=&\frac{1}{4}s(\mu)\prod_{v\in S}\Tr((\xi_0\otimes\mu_v)(f_v)) \delta(\mu)\prod_{i=1}^{r}(1-q_i^{-1})^2(X\log X+\left(2\upgamma_S-1\right)X)+O(X^{\frac12}).
\end{align*}

By summing over $\mu$ and using \autoref{prop:dnestimate} and that $s(\triv)=-1$, we obtain the corollary since the sum over $\mu$ is finite.
\end{proof}

The remaining work is to compute $\Tr(\xi_0(f_v))$ for $v\in S$.
\begin{proposition}\label{prop:archimedeantraceeisenstein}
We have
\[
\Tr(\xi_0(f_\infty))=2\int_{X_0}\frac{\widehat{\Theta}_\infty(x)}{|1-x||x|^{1/2}}\rmd x,
\]
where $X_0=\lopen 0,+\infty\ropen$.
\end{proposition}
\begin{proof}
We have
\begin{equation}\label{eq:thetarelation}
\Theta_\infty^\pm(x)=\theta_\infty\left(\pm 1,\frac{1-x}{4}\right)=\begin{dcases}
                                                                     \theta_\infty^+\left(\pm \frac{1}{\sqrt{1-x}}\right), & x<1, \\
                                                                     \theta_\infty^-\left(\pm \frac{1}{\sqrt{x-1}}\right), & x>1.
                                                                   \end{dcases}
\end{equation}
Also, recall that $\widehat{\Theta}_\infty(x)=\Theta_\infty^+(x)+\Theta_\infty^-(x)$. Hence
\[
\int_{X_0}\frac{\widehat{\Theta}_\infty(x)}{|1-x||x|^{1/2}}\rmd x=\int_{0}^{1}\sum_{\pm}\frac{\theta_\infty^+\left(\pm \frac{1}{\sqrt{1-x}}\right)}{(1-x)|x|^{1/2}}\rmd x+\int_{1}^{+\infty}\sum_{\pm}\frac{\theta_\infty^-\left(\pm \frac{1}{\sqrt{x-1}}\right)}{(x-1)|x|^{1/2}}\rmd x.
\]
For the first term, we make change of variable $1/\sqrt{1-x}\mapsto t$ so that $x=1-1/t^2$ and $\rmd x=2\rmd t/t^3$. Hence
\[
\int_{0}^{1}\sum_{\pm}\frac{\theta_\infty^+\left(\pm \frac{1}{\sqrt{1-x}}\right)}{(1-x)|x|^{1/2}}\rmd x =\sum_{\pm}\int_{1}^{+\infty}\theta_\infty^+(\pm t)t^2\frac{2\rmd t}{t^3|1-1/t^2|^{1/2}}=2\int_{x^2-1>0}\frac{\theta_\infty^+(x)}{\sqrt{x^2-1}}\rmd x.
\]
Similarly,
\[
\int_{1}^{+\infty}\sum_{\pm}\frac{\theta_\infty^-\left(\pm \frac{1}{\sqrt{x-1}}\right)}{(x-1)|x|^{1/2}}\rmd x= 2\int_{\RR}\frac{\theta_\infty^-(x)}{\sqrt{x^2+1}}\rmd x.
\]
Hence the conclusion follows by Proposition 8.3 of \cite{cheng2025}.
\end{proof}

\begin{proposition}\label{prop:ramifiedtraceeisenstein}
We have
\[
\Tr(\xi_0(f_p))=2(1-p^{-1})^{-1}\int_{Y_1} \frac{\widehat{\Theta}_p(y)}{|1-y|_p|y|_p'^{1/2}}\rmd y,
\]
where
\[
Y_1=\left\{y\in \QQ_p\,\middle|\,\omega(y)=\legendresymbol{y|y|_p'}{p}=1\right\}.
\]
\end{proposition}
\begin{proof}
In the proof we omit the subscript $p$ in the norms. By the definition of $\widehat{\Theta}_p(y)$ we have
\[
\int_{Y_1}\frac{\widehat{\Theta}_p(y)}{|1-y||y|'^{1/2}}\rmd y=\int_{Y_1\times\QQ_p}\theta_p\left(z,\frac{z^2(1-y)}{4}\right)\frac{\rmd y\rmd z}{|1-y||z||y|'^{1/2}}.
\]
Now we make change of variable $z\mapsto T$ and $z^2(1-y)/4\mapsto N$ so that
\[
y\mapsto 1-\frac{4N}{T^2}\quad\text{and}\quad z\mapsto T.
\]
We have
\[
\rmd y\wedge\rmd z=\left(\frac{8N}{T^3}\rmd T-\frac{4}{T^2}\rmd N\right)\wedge \rmd T=\frac{4}{T^2}\rmd T\wedge \rmd N.
\]
For $T,N\in \QQ_p^2$ with $T^2-4N\neq 0$, we denote $\gamma_{T,N}$ to be the regular element in $\G(\QQ_p)$ with trace $T$ and determinant $N$, up to conjugacy. Then $\gamma_{T,N}$ is split if and only if $\omega(T^2-4N)=1$, if and only if
\[
\omega\left(1-\frac{4N}{T^2}\right)=\omega(y)=1.
\]
Hence we obtain
\begin{align*}
\int_{Y_1}\frac{\widehat{\Theta}_p(y)}{|1-y||y|'^{1/2}}\rmd y&=\int_{\gamma_{T,N}\in \A(\QQ_p)}\theta_p(T,N) \frac{\rmd T\rmd N}{|4N/T^2||T||1-4N/T^2|'^{1/2}}\left|\frac{4}{T^2}\right|\\
&=\int_{\gamma_{T,N}\in \A(\QQ_p)} \frac{1}{|N||T^2-4N|'^{1/2}}\theta_{p}(T,N)\rmd T\rmd N.
\end{align*}
By Proposition 8.5 of \cite{cheng2025} we obtain the desired result.
\end{proof}

Now we give the asymptotic formula for the discrete part in the continuous spectrum.
\begin{theorem}\label{thm:contributioneisenstein}
The contribution of the discrete part in the continuous spectrum is
\begin{align*}
\sum_{\substack{n<X\\\gcd(n,S)=1}}\frac{1}{4}\sum_{\mu}\Tr(M(0,\mu)(\xi_0\otimes\mu)(f^n))
=\frac12\mf{B}(X\log X+(2\upgamma_S-1)X)+O(X^{\frac12}).
\end{align*}
where $\mf{B}$ is defined in Theorem 5.1 of \cite{cheng2025c}. The implied constant only depends on $f_\infty$ and $f_{q_i}$.
\end{theorem}
\begin{proof}
By \autoref{prop:archimedeantraceeisenstein} and \autoref{prop:ramifiedtraceeisenstein} we have
\begin{equation}\label{eq:ramifiedproducteisenstein}
\begin{split}
\prod_{v\in S}\Tr(\xi_0(f_v))=&2^{r+1}\prod_{i=1}^{r}(1-q_i^{-1})^{-1}\int_{X_0}\int_{Y_\mathbf{1}} |1-x|_\infty^{-1}|1-y|_q^{-1}\widehat{\Theta}_\infty(x)\widehat{\Theta}_{q}(y) |x|_\infty^{-\frac{1}{2}}|y|_q'^{-\frac{1}{2}}\rmd x\rmd y\\=&-2\prod_{i=1}^{r}(1-q_i^{-1})^{-2}\mf{B}.
\end{split}
\end{equation}
Hence by \autoref{cor:eisensteinasymptotic}
we obtain our result.
\end{proof}

\subsection{Contribution of the continuous part} In this section we will give an estimate of
\[
S_\cont(X)=\sum_{\substack{n<X\\ \gcd(n,S)=1}}J_\cont(f^n).
\]

The continuous term is the sum of $J_\cont^1(f^n)=\eqref{eq:continuouspart1}$ and $J_\cont^2(f^n)=\eqref{eq:continuouspart2}$. Hence it suffices to consider
\[
S_\cont^1(X)=\sum_{\substack{n<X\\ \gcd(n,S)=1}}J_\cont^1(f^n)\qquad\text{and}\qquad S_\cont^2(X)=\sum_{\substack{n<X\\ \gcd(n,S)=1}}J_\cont^2(f^n).
\]

By splitting the trace of $f^n$ into local traces, we obtain
\[
J_\cont^1(f^n)=-\frac{1}{4\uppi\rmi}\sum_{\mu}\int_{(0)}\frac{m'(s,\mu)}{m(s,\mu)}  \prod_{v\in S}\Tr\left(\Ind_{\B(\QQ_v)}^{\G(\QQ_v)}(s,\mu)(f_v)\right)\cdot\prod_{p\notin S}\Tr\left(\Ind_{\B(\QQ_p)}^{\G(\QQ_p)}(s,\mu_p)(f_p^n)\right)\rmd s.
\]

Note that $R_v(s,\mu_v)=\id$ on spherical vectors if $\mu_v$ is unramified. For $v\notin S$, $f_v^n$ is spherical and $R_v'(s,\mu_v)$ kills all spherical vectors, so we have
\[
R_v(s,\mu_v)^{-1}R_v'(s,\mu_v)\Ind_{\B(\QQ_v)}^{\G(\QQ_v)}(s,\mu_v)(f_v)=0.
\]
Therefore $J_\cont^2(f^n)$ equals
\[
-\frac{1}{4\uppi\rmi}\sum_{v\in S}\sum_{\mu}\int_{(0)}\Tr \left(R_v(s,\mu_v)^{-1}R_v'(s,\mu_v)\Ind_{\B(\QQ_v)}^{\G(\QQ_v)}(s,\mu_v)(f_v)\right)\prod_{w\neq v}\Tr\left(\Ind_{\B(\QQ_w)}^{\G(\QQ_w)}(s,\mu_w)(f^n_w)\right)\rmd s.
\]

\begin{proposition}
Suppose that $p\in S$. There exists $M>0$ such that 
\[
\Ind_{\B(\QQ_p)}^{\G(\QQ_p)}(s,\mu_p)(f_p)
\]
is the zero operator if $\mu_{1,p}$ or $\mu_{2,p}$ is not trivial on $1+p^M\ZZ_p$.
\end{proposition}
\begin{proof}
The kernel function of this operator is (cf. (3.13) of \cite{cheng2025b})
\[
A(g,h)=\int_{ \A(\QQ_p)}\frac{|x-y|_p}{|xy|_p^{1/2}}\left|\frac xy\right|_p^{s}\mu_{1,p}(x)\mu_{2,p}(y)f_p(g^{-1}th) \rmd t\quad g,h\in K.
\]
Since $f_p$ is a smooth function, the kernel function is smooth in $g$ and $h$. Hence there exists $M>0$ such that
\[
A\left(g,\begin{pmatrix}
           a & 0 \\
           0 & b 
         \end{pmatrix}h\right)=A(g,h)
\]
if $a,b\in 1+p^M\ZZ_p$. Then for any $a,b\in 1+p^M\ZZ_p$ and $\phi\in \Ind_{\B(\QQ_p)}^{\G(\QQ_p)}(s,\mu_p)$ we have
\begin{align*}
\left(\Ind_{\B(\QQ_p)}^{\G(\QQ_p)}(s,\mu_p)(f_p)\phi\right)(g)=&\int_{K}A(g,h)\phi(h)\rmd h=\int_{K}A\left(g,\begin{pmatrix}
           a & 0 \\
           0 & b 
         \end{pmatrix}h\right)\phi\left(\begin{pmatrix}
           a & 0 \\
           0 & b 
         \end{pmatrix}h\right)\rmd h\\
         =&\mu_{1,p}(a)\mu_{2,p}(b)\int_{K}A(g,h)\phi(h)\rmd h.
\end{align*}
Hence 
\[
\left(\Ind_{\B(\QQ_p)}^{\G(\QQ_p)}(s,\mu_p)(f_p)\phi\right)(g)=\int_{K}A(g,h)\phi(h)\rmd h=0
\]
for any $\phi$ if $\mu_{1,p}$ or $\mu_{2,p}$ is not trivial on $1+p^M\ZZ_p$.
\end{proof}

\begin{corollary}\label{cor:characterfinite}
The sum over $\mu$ in $J_\cont^2(f^n)$ is finite.
\end{corollary}
\begin{proof}
For any $q_i\in S$, there exists $M_i>0$ such that $\Ind_{\B(\QQ_{q_i})}^{\G(\QQ_{q_i})}(s,\mu_{q_i})(f_{q_i})$ is zero if $\mu_{1,q_i}$ or $\mu_{2,q_i}$ is not trivial on $1+q_i^{M_i}\ZZ_{q_i}$. Hence
\[
\Tr \left(R_{q_i}(s,\mu_{q_i})^{-1}R_{q_i}'(s,\mu_{q_i})\Ind_{\B(\QQ_{q_i})}^{\G(\QQ_{q_i})}(s,\mu_{q_i})(f_{q_i}) \right)=0\quad\text{and}\quad
\Tr \left(\Ind_{\B(\QQ_{q_i})}^{\G(\QQ_{q_i})}(s,\mu_{q_i})(f_{q_i})\right)=0.
\]

Suppose that $\mu=\mu_1\boxtimes \mu_2$ is a character on $\A(\QQ)\bs \A(\AA)^1=\QQ^\times\bs (\AA^\times)^1\times\QQ^\times\bs (\AA^\times)^1$ contributing the sum. Hence $\mu_{i,p}$ is unramified for each $p\notin S$ and $\mu_{q_i}$ is trivial on $1+q_i^{M_i}\ZZ_{q_i}$ for each $i$. Hence $\mu_1$ and $\mu_2$ can be considered as characters on
\[
\left(\RR_{>0}\QQ^\times\prod_{p\notin S}\ZZ_p^\times \prod_{i=1}^{r}(1+q_i^{M_i}\ZZ_{q_i})\right)\bs \AA^\times,
\]
which is a finite group by class field theory. Thus, the sum in \eqref{eq:continuouspart1} over $\mu=\mu_1\boxtimes \mu_2$ is finite.
\end{proof}

\begin{theorem}\label{thm:rapiddecay}
The ramified part
\[
\sum_{v\in S}\Tr \left(R_v(s,\mu_v)^{-1}R_v'(s,\mu_v)\Ind_{\B(\QQ_v)}^{\G(\QQ_v)}(s,\mu_v)(f_v)\right)\prod_{\substack{w\neq v\\w\in S}}\Tr\left(\Ind_{\B(\QQ_w)}^{\G(\QQ_w)}(s,\mu_w)(f_w)\right)
\]
has rapid decay vertically on the imaginary axis. 
\end{theorem}
\begin{proof}
See the proofs of \cite{muller2011,muller2004,muller2002}, especially the last reference.
\end{proof}

\begin{lemma}\label{lem:unramifiedtracecontinuous}
If $\mu_1$ and $\mu_2$ are unramified at all places outside $S$, then
\[
\prod_{p\notin S}\Tr\left(\Ind_{\B(\QQ_p)}^{\G(\QQ_p)}(s,\mu_p)(f_p^n)\right)=\frac{1}{n^{s}}\sum_{d\mid n}d^{2s}\mu_1(d)\mu_2\legendresymbol{n}{d},
\]
where, by abuse of notation, $\mu_1$ and $\mu_2$ are considered as Dirichlet characters corresponding to the characters on id\`eles. Otherwise, the above equals $0$.
\end{lemma}
\begin{proof}
By Proposition 3.9 of \cite{cheng2025b} we have
\[
\prod_{p\notin S}\Tr\left(\Ind_{\B(\QQ_p)}^{\G(\QQ_p)}(s,\mu_p)(f_p^n)\right)=0
\]
if $\mu_1$ or $\mu_2$ is ramified at a place $p\notin S$ and
\[
\prod_{p\notin S}\Tr\left(\Ind_{\B(\QQ_p)}^{\G(\QQ_p)}(s,\mu_p)(f_p^n)\right)=\prod_{p\notin S}\left(\sum_{u=0}^{n_p}\mu_{1,p}(p^u)\mu_{2,p}(p^{n_p-u})p^{(2u-n_p)s}\right)
\]
otherwise, where $n_p=v_p(n)$. Clearly we have
\[
\prod_{p\notin S}\left(\sum_{u=0}^{n_p}\mu_{1,p}(p^u)\mu_{2,p}(p^{n_p-u})p^{(2u-n_p)s}\right)= \prod_{p\notin S}\left(\sum_{u=0}^{n_p}\mu_{1}(p^u)\mu_{2}(p^{n_p-u})\frac{p^{2us}}{p^{n_ps}}\right)=\frac{1}{n^{s}} \sum_{d\mid n}d^{2s}\mu_1(d)\mu_2\legendresymbol{n}{d}.
\]
Hence we obtain the desired result.
\end{proof}

\begin{lemma}\label{lem:estimateconvolution}
Let $s\in \rmi\RR$ such that $s\neq 0$. Let $\chi_1$ and $\chi_2$ be Dirichlet characters such that $\chi_i(m)\neq 0$ if and only if $\gcd(m,S)=1$ for $i=1,2$. Then for any $\varepsilon>0$ we have
\begin{align*}
   & \sum_{n<X}\frac{1}{n^{s}}\sum_{d\mid n}d^{2s}\chi_1(d)\chi_2\legendresymbol{n}{d} \\
  = & \prod_{i=1}^{r}(1-q_i^{-1})\left[\delta(\chi_1)L(2s+1,\chi_2)\frac{X^{s+1}}{s+1}+ \delta(\chi_2)L(-2s+1,\chi_1)\frac{X^{-s+1}}{-s+1}\right]\\
  +&O((1+|s|)^{1/3+\varepsilon}C(\chi_1)^{1/6+\varepsilon}C(\chi_2)^{1/6+\varepsilon}X^{2/3+\varepsilon}),
\end{align*}
where $C(\chi)$ denotes the conductor of $\chi$. The implied constant only depends on $\varepsilon$. 
\end{lemma}
\begin{proof}
Clearly we may assume that $X$ is not an integer. For the single term we have
\[
\frac{1}{n^{s}}\sum_{d\mid n}d^{2s}\chi_1(d)\chi_2\legendresymbol{n}{d}\ll \bm{d}(n)\ll_\varepsilon n^{\varepsilon}.
\]
Let $c>1$ and $T>0$. By Perron's formula we have
\[
\sum_{n<X}\frac{1}{n^{s}}\sum_{d\mid n}d^{2s}\chi_1(d)\chi_2\legendresymbol{n}{d}=\frac{1}{\dpii}\int_{c-\rmi T}^{c+\rmi T}\sum_{n=1}^{+\infty} \frac{1}{n^{s+u}}\sum_{d\mid n}d^{2s}\chi_1(d)\chi_2\legendresymbol{n}{d}\frac{X^u}{u}\rmd u+ O\left(\frac{X^c}{T}\right).
\]
The integrand is
\begin{equation}\label{eq:integrandhyperbola}
\sum_{n=1}^{+\infty} \frac{1}{n^{s+u}}\sum_{d\mid n}d^{2s}\chi_1(d)\chi_2\legendresymbol{n}{d}\frac{X^u}{u} =\sum_{a,b=1}^{+\infty}\frac{a^{2s}}{(ab)^{s+u}}\chi_1(a)\chi_2(b)=   L(-s+u,\chi_1)L(s+u,\chi_2)\frac{X^u}{u}.
\end{equation}

For $\Re u>0$, \eqref{eq:integrandhyperbola} is holomorphic except at $u=\pm s+1$.
If $\chi_1$ is not trivial, then \eqref{eq:integrandhyperbola} is regular at $u=s+1$. If $\chi_1$ is trivial, then \eqref{eq:integrandhyperbola} has a pole at $u=s+1$ with residue
\[
\res_{u=s+1}L(-s+u,\chi_1)L(s+u,\chi_2)\frac{X^u}{u}=\prod_{i=1}^{r}(1-q_i^{-1})L(2s+1,\chi_2) \frac{X^{s+1}}{s+1}.
\]
Similarly, if $\chi_2$ is not trivial, then  \eqref{eq:integrandhyperbola} is regular at $u=-s+1$. If $\chi_2$ is trivial, then \eqref{eq:integrandhyperbola} has a pole at $u=-s+1$ with residue
\[
\prod_{i=1}^{r}(1-q_i^{-1})L(-2s+1,\chi_1)\frac{X^{-s+1}}{-s+1}.
\]
Hence by residue formula we have
\begin{align*}
&\frac{1}{\dpii}\int_{c-\rmi T}^{c+\rmi T}\sum_{n=1}^{+\infty} \frac{1}{n^{s+u}}\sum_{d\mid n}d^{2s}\chi_1(d)\chi_2\legendresymbol{n}{d}\frac{X^u}{u}\rmd u=\frac{1}{\dpii}\left(\int_{c-\rmi T}^{\frac12-\rmi T}+\int_{\frac12-\rmi T}^{\frac12+\rmi T}+\int_{\frac12+\rmi T}^{c+\rmi T}\right)\cdots\rmd u\\
+&\prod_{i=1}^{r}(1-q_i^{-1})\left[\delta(\chi_1)L(2s+1,\chi_2)\frac{X^{s+1}}{s+1}+ \delta(\chi_2)L(-2s+1,\chi_1)\frac{X^{-s+1}}{-s+1}\right],
\end{align*}
where $\cdots$ denotes the integrand \eqref{eq:integrandhyperbola}.

We have the following \emph{Weyl bound} for Dirichlet $L$-functions 
\[
L(\sigma+\rmi t,\chi)\ll_\varepsilon ((1+|t|)C(\chi))^{1/6+\varepsilon}
\]
for $\sigma\in [1/2,c]$ (cf. \cite{petrow2023}). From this, it is not hard to see that
\[
\frac{1}{\dpii}\int_{c-\rmi T}^{\frac12-\rmi T}\cdots\rmd u\ll (1+|s|)^{\frac13+\varepsilon}C(\chi_1)^{\frac16+\varepsilon}C(\chi_2)^{\frac16+\varepsilon}X^cT^{-\frac23+\varepsilon},
\]
\[
\frac{1}{\dpii}\int_{\frac12-\rmi T}^{\frac12+\rmi T}\cdots\rmd u\ll (1+|s|)^{\frac13+\varepsilon}C(\chi_1)^{\frac16+\varepsilon}C(\chi_2)^{\frac16+\varepsilon}X^{\frac12}T^{\frac13+ \varepsilon},
\]
and
\[
\frac{1}{\dpii}\int_{c+\rmi T}^{\frac12+\rmi T}\cdots\rmd u\ll (1+|s|)^{\frac13+\varepsilon}C(\chi_1)^{\frac16+\varepsilon}C(\chi_2)^{\frac16+\varepsilon}X^cT^ {-\frac23+\varepsilon}.
\]
Taking $c=1+\varepsilon$ and $T=X^{1/2}$ we know that the three contour integrals are bounded by
\[
(1+|s|)^{\frac13+\varepsilon}C(\chi_1)^{\frac16+\varepsilon}C(\chi_2)^{\frac16+\varepsilon}X^{\frac23+\varepsilon}.
\] 
for any $\varepsilon>0$.
Also, the truncated term $X^c/T\ll X^{1/2+\varepsilon}$. Hence we obtain the asymptotic formula.
\end{proof}

\begin{remark}
If we only assume the \emph{convexity bound}, that is
\[
L(1/2+\rmi t,\chi)\ll_\varepsilon ((1+|t|)C(\chi))^{\frac14+\varepsilon},
\]
the error term is bounded by
\[
(1+|s|)^{\frac12+\varepsilon}C(\chi_1)^{\frac14+\varepsilon}C(\chi_2)^{\frac14+\varepsilon}X^{\frac34+\varepsilon},
\]
and we can also derive the main result using this bound. If we assume the \emph{Lindel\"of hypothesis} for Dirichlet $L$-functions, that is
\[
L(1/2+\rmi t,\chi)\ll_\varepsilon ((1+|t|)C(\chi))^{\varepsilon},
\]
then the error term is bounded by
\[
(1+|s|)^{\varepsilon}C(\chi_1)^{\varepsilon}C(\chi_2)^{\varepsilon}X^{\frac12+\varepsilon}.
\]
\end{remark}

Now we prove technical propositions that will be used in estimating the spectral side.
\begin{proposition}\label{prop:estimatecontinuouscharacter}
Suppose that for any prime $p\notin S$ we are given a positive number $\alpha_p$ and for any character $\mu=\mu_1\boxtimes\mu_2$ on $\A(\QQ)\bs \A(\AA)^1$, we are given a holomorphic function
$\Phi(s,\mu)$  on the vertical strip $\{s\in \CC\,|\, \Re s\in \lopen -1,1\ropen\}$, such that 
\begin{enumerate}[itemsep=0pt,parsep=0pt,topsep=0pt,leftmargin=0pt,labelsep=2.5pt,itemindent=15pt,label=\upshape{(\roman*)}]
  \item There exists a finite set of id\'elic characters such that $\Phi(s,\mu)=0$ if $\mu_1\mu_2$ is not in this finite set. 
  \item There exists at most one prime $p$ such that, if $\Phi(s,\mu)\neq0$, then for any prime $\ell\neq p$, $C(\mu_{1,\ell})$ and $C(\mu_{2,\ell})$ both lie in a certain finite set. For almost all primes $\ell$, $\Phi(s,\mu)\neq 0$ implies $C(\mu_{1,\ell})=C(\mu_{2,\ell})=1$.
  \item We have the following estimate
  \[
  \Phi(s,\mu)\ll \alpha_p(1+|s|)^{-2}C(\mu_1)^{-1}C(\mu_2)^{-1},
  \]
  where the implied constant only depends on $\sigma=\Re s$.
\end{enumerate} 
Then for any $\varepsilon>0$,
\[
\sum_{n<X}\frac{1}{\dpii}\sum_{\mu}\int_{(0)}\Phi(s,\mu)\frac{1}{n^{s}}\sum_{d\mid n}d^{2s}\mu_1(d)\mu_2\legendresymbol{n}{d}\rmd s=\frac12\prod_{i=1}^{r}(1-q_i^{-1})^2 \Phi(0,\triv)X+O(\alpha_pX^{\frac23+\varepsilon}),
\]
where the implied constant only depends on the finite set and $\varepsilon$.
\end{proposition}
\begin{proposition}\label{prop:estimatecontinuouscharacter2}
Suppose that $\mu=\mu_1\boxtimes\mu_2$ is a character on $\A(\QQ)\bs \A(\AA)^1$.
$\Phi(s)$ is a smooth function on the imaginary axis with rapid decay.
Then 
\[
\sum_{n<X}\frac{1}{\dpii}\int_{(0)}\Phi(s)\frac{1}{n^{s}}\sum_{d\mid n}d^{2s}\mu_1(d)\mu_2\legendresymbol{n}{d}\rmd s=\frac12\delta(\mu_1)\delta(\mu_2)\prod_{i=1}^{r}(1-q_i^{-1})^2 \Phi(0)X+o(X).
\]
\end{proposition}
\autoref{prop:estimatecontinuouscharacter} gives a better bound of the error term, while it has more restrictions on the function $\Phi$. We will mainly use \autoref{prop:estimatecontinuouscharacter2} in this section and \autoref{prop:estimatecontinuouscharacter} for the estimate of the geometric side. The proofs of the two propositions use different methods. For \autoref{prop:estimatecontinuouscharacter}, we use the contour shift method. For \autoref{prop:estimatecontinuouscharacter2}, we use the Riemann-Lebesgue lemma and the trick in \cite{gelbart1979}.

\begin{proof}[Proof of \autoref{prop:estimatecontinuouscharacter}]
By \autoref{lem:estimateconvolution}, the left hand side equals
\begin{equation}\label{eq:sumresidue}
\begin{split}
&\frac{1}{\dpii}\sum_{\mu}\int_{(0)}\Phi(s,\mu)\sum_{n<X}\frac{1}{n^{s}}\sum_{d\mid n}d^{2s}\mu_1(d)\mu_2\legendresymbol{n}{d}\rmd s\\
=&\frac{1}{\dpii}\sum_{\mu}\int_{(0)}\Phi(s,\mu) \prod_{i=1}^{r}(1-q_i^{-1})\left[\delta(\mu_1)L(2s+1,\mu_2)\frac{X^{s+1}}{s+1}+ \delta(\mu_2)L(-2s+1,\mu_1)\frac{X^{-s+1}}{-s+1}\right]\rmd s\\
+&\frac{1}{\dpii}\sum_{\mu}\int_{(0)}\Phi(s,\mu)O((1+|s|)^{\frac 13+\varepsilon}C(\mu_1)^{\frac16}C(\mu_2)^{\frac16}X^{\frac23+\varepsilon})\rmd s.
\end{split}
\end{equation}

We claim that for any $\sigma\in \lopen-1,1\ropen$,
\[
\sum_{\mu}\int_{(\sigma)}\Phi(s,\mu)(1+|s|)^{\frac 13+\varepsilon}C(\mu_1)^{\frac16}C(\mu_2)^{\frac16}\rmd s
\]
converges absolutely. Suppose that $\mu=\mu_1\boxtimes\mu_2$ contributes the sum. By assumption, the choice of $\mu_1$ is finite once $\mu_2$ is fixed. Moreover, $r=C(\mu_2)^{(p)}$ has finitely many choices. Now we fix $r$ and $\chi=\mu_1\mu_2$. Clearly we have $C(\mu_{1})\asymp C(\mu_{2})$ for $\mu$ contributing to the sum. Therefore,
\begin{align*}
   & \sum_{\mu}\int_{(0)}\Phi(s,\mu)(1+|s|)^{\frac 13+\varepsilon}C(\mu_1)^{\frac16}C(\mu_2)^{\frac16}\rmd s \\
  \ll & \alpha_p\sum_{\chi,r}\sum_{j=0}^{+\infty}\sum_{\mu_2\colon C(\mu_2)=rp^j}\int_{(\sigma)}\Phi(s,\mu)(1+|s|)^{\frac 13+\varepsilon}C(\mu_2)^{\frac13}\rmd s\\
  \ll &\alpha_p\sum_{\chi,r}\sum_{j=0}^{+\infty}\sum_{\mu_2\colon C(\mu_2)=rp^j}C(\mu_2)^{-\frac53}\int_{(\sigma)}(1+|s|)^{-\frac53}\rmd s.
\end{align*}
Since
\[
\int_{(\sigma)}(1+|s|)^{-\frac53}\rmd s\ll 1
\]
and
\[
\sum_{j=0}^{+\infty}\sum_{\mu_2\colon C(\mu_2)=rp^j}C(\mu_2)^{-\frac53}\ll \sum_{j=0}^{+\infty}(rp^j)^{1-\frac 53}\ll 1,
\]
we obtain the desired result.

Hence the error term of \eqref{eq:sumresidue} is bounded by
\[
\frac{1}{\dpii}\sum_{\mu}\int_{(0)}\Phi(s,\mu)(1+|s|)^{\frac 13+\varepsilon}C(\mu_1)^{\frac16}C(\mu_2)^{\frac16}X^{\frac23+\varepsilon}\rmd s\ll \alpha_p X^{\frac23+\varepsilon}.
\]

For the main term we use the contour shift method. Consider the following three cases:

\emph{\underline{Case 1:}}\ \ Neither $\mu_1$ nor $\mu_2$ is trivial.

In this case, the main term vanishes.

\emph{\underline{Case 2:}}\ \ Exactly one of $\mu_1$ and $\mu_2$ is trivial.

We may assume that $\mu_1$ is trivial while $\mu_2$ is not. In this case, the first term becomes
\begin{equation}\label{eq:continuousmainterm}
\frac{1}{\dpii}\int_{(0)}\Phi(s) \prod_{i=1}^{r}(1-q_i^{-1})L(2s+1,\mu_2)\frac{X^{s+1}}{s+1}\rmd s.
\end{equation}
Since $\mu_2$ is not trivial, $L(s+1,\mu_2)$ is an entire function.

We move the contour from $(0)$ to $(-1+\varepsilon)$ and obtain
\[
\eqref{eq:continuousmainterm}=\frac{1}{\dpii}\int_{(-1+\varepsilon)}\Phi(s) \prod_{i=1}^{r}(1-q_i^{-1})L(2s+1,\mu_2)\frac{X^{s+1}}{s+1}\rmd s.
\] 
We have
\[
L(2s+1,\mu_2)\ll ((1+|s|)C(\mu_2))^{\frac12}
\] 
by using the functional equation and the Stirling formula for $\Re s=-1+\varepsilon$. Hence
\begin{align*}
\eqref{eq:continuousmainterm}&\ll\int_{(-1+\varepsilon)}|\Phi(s)| \prod_{i=1}^{r}(1-q_i^{-1})|L(2s+1,\mu_2)|\left|\frac{X^{s+1}}{s+1}\right|\rmd |s|\\
&\ll \alpha_p\int_{(-1+\varepsilon)}(1+|s|)^{-2}C(\mu_2)^{-2}((1+|s|)C(\mu_2))^{\frac12}X^{\frac12+\varepsilon}\rmd |s|\ll\alpha_p C(\mu_2)^{-\frac32}X^{\frac12+\varepsilon}.
\end{align*}
Since $\mu_1\mu_2$ has finitely many choices, $\mu_2=\triv\mu_2=\mu_1\mu_2$ has finitely many choices as well.

Thus 
\[
\sum_{\mu\colon \mu_1=\triv\ \text{or}\ \mu_2=\triv}\frac{1}{\dpii}\int_{(-1+\varepsilon)}\Phi(s) \prod_{i=1}^{r}(1-q_i^{-1})L(2s+1,\mu_2)\frac{X^{s+1}}{s+1}\rmd s\ll X^{\frac12+\varepsilon}.
\]

\emph{\underline{Case 3:}}\ \ Both $\mu_1$ and $\mu_2$ are trivial.

In this case, the first term becomes
\begin{equation}\label{eq:firsttermbothtrivial}
\frac{1}{\dpii}\int_{(0)}\Phi(s,\mu) \prod_{i=1}^{r}(1-q_i^{-1})\left[\zeta^S(2s+1)\frac{X^{s+1}}{s+1}+ \zeta^S(-2s+1)\frac{X^{-s+1}}{-s+1}\right]\rmd s.
\end{equation}

The integrand is holomorphic for $\Re s\in \lopen -1,1\ropen$ (for $s=0$, by considering the Laurent expansion). We first move the contour from $(0)$ to $(1-\varepsilon)$. By residue theorem we have
\begin{align*}
\eqref{eq:firsttermbothtrivial}=&\frac{1}{\dpii}\int_{(1-\varepsilon)}\Phi(s) \prod_{i=1}^{r}(1-q_i^{-1})\zeta^S(2s+1)\frac{X^{s+1}}{s+1}\\
+&\frac{1}{\dpii}\int_{(1-\varepsilon)}\Phi(s) \prod_{i=1}^{r}(1-q_i^{-1}) \zeta^S(-2s+1)\frac{X^{-s+1}}{-s+1}\rmd s.
\end{align*}

The second term is bounded by
\begin{align*}
\int_{(1-\varepsilon)}|\Phi(s,\mu)||\zeta^S(-2s+1)|\left|\frac{X^{-s+1}}{-s+1}\right|\rmd |s|
&\ll\alpha_p\int_{(1-\varepsilon)}(1+|s|)^{-2}(1+|s|)^{1/2}\left|\frac{X^{-s+1}}{-s+1}\right|\rmd |s|\\
&\ll \alpha_pX^{1/2+\varepsilon}.
\end{align*}
For the first term, we move the contour to $(-1+\varepsilon)$.
The function $\zeta^S(2s+1)/(s+1)$ has a simple pole at $s=0$ with residue
\[
\res_{s=0}\frac{\zeta^S(2s+1)}{s+1}=\frac12\prod_{i=1}^{r}(1-q_i^{-1}).
\]
Hence by residue formula, the first term equals
\[
\frac{1}{\dpii}\int_{(-1+\varepsilon)}\Phi(s,\mu)\zeta^S(2s+1)\frac{X^{s+1}}{s+1}\rmd s+\frac12\Phi(0,\mu)\prod_{i=1}^{r}(1-q_i^{-1})^2 X.
\]
Similarly we have
\[
\frac{1}{\dpii}\int_{(-1+\varepsilon)}\Phi(s,\mu)\zeta^S(-2s+1)\frac{X^{-s+1}}{-s+1}\ll \alpha_p X^{\frac12+\varepsilon}.
\]
Hence 
\[
\eqref{eq:firsttermbothtrivial}=\frac12\Phi(0,\mu)\prod_{i=1}^{r}(1-q_i^{-1})^2 X+O(\alpha_p X^{\frac12+\varepsilon}).
\]
The conclusion follows by adding the three cases together.
\end{proof}

\begin{proof}[Proof of \autoref{prop:estimatecontinuouscharacter2}]
By \autoref{lem:estimateconvolution}, the left hand side equals
\begin{align*}
&\frac{1}{\dpii}\int_{(0)}\Phi(s)\sum_{n<X}\frac{1}{n^{s}}\sum_{d\mid n}d^{2s}\mu_1(d)\mu_2\legendresymbol{n}{d}\rmd s\\
=&\frac{1}{\dpii}\int_{(0)}\Phi(s) \prod_{i=1}^{r}(1-q_i^{-1})\left[\delta(\mu_1)L(2s+1,\mu_2)\frac{X^{s+1}}{s+1}+ \delta(\mu_2)L(-2s+1,\mu_1)\frac{X^{-s+1}}{-s+1}\right]\rmd s\\
+&\frac{1}{\dpii}\int_{(0)}\Phi(s)O((1+|s|)^{\frac 13+\varepsilon}X^{\frac23+\varepsilon})\rmd s.
\end{align*}
The second term is bounded by
\[
\int_{(0)}\Phi(s)(1+|s|)^{\frac 13+\varepsilon}\rmd s X^{\frac23+\varepsilon}\ll X^{\frac23+\varepsilon}
\]
since $\Phi(s)$ has rapid decay.

For the first term, we consider the following three cases:

\emph{\underline{Case 1:}}\ \ Neither $\mu_1$ nor $\mu_2$ is trivial.

In this case, the first term vanishes.

\emph{\underline{Case 2:}}\ \ Exactly one of $\mu_1$ and $\mu_2$ is trivial.

We may assume that $\mu_1$ is trivial while $\mu_2$ is not. In this case, the first term becomes
\[
\frac{1}{\dpii}\int_{(0)}\Phi(s) \prod_{i=1}^{r}(1-q_i^{-1})L(2s+1,\mu_2)\frac{X^{s+1}}{s+1}\rmd s.
\]
Since $\mu_2$ is not trivial, $L(s+1,\mu_2)$ is holomorphic on the imaginary axis. Hence by Riemann-Lebesgue lemma we have
\[
\lim_{X\to+\infty}\frac{1}{\dpii}\int_{(0)}\Phi(s) \prod_{i=1}^{r}(1-q_i^{-1})L(2s+1,\mu_2)\frac{X^{s}}{s+1}\rmd s=0.
\]
In other words,
\[
\frac{1}{\dpii}\int_{(0)}\Phi(s) \prod_{i=1}^{r}(1-q_i^{-1})L(2s+1,\mu_2)\frac{X^{s+1}}{s+1}\rmd s=o(X).
\]

\emph{\underline{Case 3:}}\ \ Both $\mu_1$ and $\mu_2$ are trivial.

In this case, the first term becomes \eqref{eq:firsttermbothtrivial}. The function $\zeta^S(2s+1)/(s+1)$ has a simple pole at $s=0$ with residue
\[
\res_{s=0}\frac{\zeta^S(2s+1)}{s+1}=\frac12\prod_{i=1}^{r}(1-q_i^{-1}).
\]

We may then write \eqref{eq:firsttermbothtrivial} as the sum of
\[
\frac12\prod_{i=1}^{r}(1-q_i^{-1})^2\frac{1}{\dpii}\int_{(0)}\Phi(s)\frac{X^{s}-X^{-s}}{s}\rmd sX
+
\frac{1}{\dpii}\int_{(0)}\Phi(s) \prod_{i=1}^{r}(1-q_i^{-1})(r(s)X^{s} +r(-s)X^{-s})\rmd s X,
\]
where 
\[
r(s)=\frac{\zeta^S(2s+1)}{s+1}-\prod_{i=1}^{r}(1-q_i^{-1})\frac{1}{2s}
\]
is holomorphic on the imaginary axis, with polynomial growth. Hence by Riemann-Lebesgue lemma we know that the second term above is $o(X)$. Now it suffices to consider the limit
\[
\lim_{X\to +\infty}\frac{1}{\dpii}\int_{(0)}\Phi(s) \frac{X^{s}-X^{-s}}{s}\rmd s.
\]
By \cite[Lemma 6.31]{gelbart1979}, the limit is $\Phi(0)$.
Hence
\[
\eqref{eq:firsttermbothtrivial}=\frac12\prod_{i=1}^{r}(1-q_i^{-1})^2\Phi(0)X+o(X).
\]

Combining all the cases together we obtain our result.
\end{proof}

\begin{proposition}\label{prop:contributioncontinuous1}
We have
\[
\sum_{\substack{n<X\\\gcd(n,S)=1}}J_\cont^1(f^n)=-\left(\upgamma-2\log 2-\log\uppi\right) \mf{B} X+o(X).
\]
\end{proposition}
\begin{proof}
Recall that
\[
J_\cont^1(f^n)=-\frac{1}{4\uppi\rmi}\sum_{\mu}\int_{(0)}\frac{m'(s,\mu)}{m(s,\mu)}  \prod_{v\in S}\Tr\left(\Ind_{\B(\QQ_v)}^{\G(\QQ_v)}(s,\mu)(f_v)\right)\cdot\prod_{p\notin S}\Tr\left(\Ind_{\B(\QQ_p)}^{\G(\QQ_p)}(s,\mu_p)(f_p^n)\right)\rmd s.
\]
The sum over $\mu$ in $J_\cont^1(f^n)$ is finite by Corollary 3.11 of \cite{cheng2025b}. By \autoref{lem:unramifiedtracecontinuous} we obtain
\[
J_\cont^1(f^n)=-\frac{1}{4\uppi\rmi}\sum_{\mu}\int_{(0)} \frac{m'(s,\mu)}{m(s,\mu)}\prod_{v\in S}\Tr\left(\Ind_{\B(\QQ_v)}^{\G(\QQ_v)}(s,\mu_v)(f_v)\right)\frac{1}{n^{s}}\sum_{d\mid n}d^{2s}\mu_1(d)\mu_2\legendresymbol{n}{d}\rmd s.
\]
Let
\[
\Phi(s,\mu)= \frac{m'(s,\mu)}{m(s,\mu)}\prod_{v\in S}\Tr\left(\Ind_{\B(\QQ_v)}^{\G(\QQ_v)}(s,\mu_v)(f_v)\right).
\]
It is holomorphic on the imaginary axis.
Note that 
\[
\Tr\left(\Ind_{\B(\RR)}^{\G(\RR)}(s,\mu_\infty)(f_\infty)\right)
\]
has rapid decay vertically since it is a Mellin transform of a compactly supported function. Moreover, the nonarchimedean parts are bounded holomorphic functions and $m'(s,\mu)/m(s,\mu)$ has polynomial growth (cf. the proof of \cite[Proposition 3.7]{cheng2025b}). Hence $\Phi(s,\mu)$ has rapid decay vertically on $s$. By \autoref{prop:estimatecontinuouscharacter2} and since the sum over $\mu$ is finite, we get
\[
\sum_{\substack{n<X\\ \gcd(n,S)=1}} J_\cont^1(f^n)=-\frac14\prod_{i=1}^{r}(1-q_i^{-1})^2 \Phi(0,\triv)X+o(X).
\]
By \eqref{eq:ramifiedproducteisenstein},
\[
\Phi(0,\triv)= \left.\frac{m'(s,\triv)}{m(s,\triv)}\right|_{s=0}\prod_{v\in S}\Tr\left(\xi_0(f_v)\right)=-2\left.\frac{m'(s,\triv)}{m(s,\triv)}\right|_{s=0}\prod_{i=1}^{r}(1-q_i^{-1})^{-2}\mf{B}
\]
Recall that
\[
m(s,\triv)=\frac{\Lambda(1-2s)}{\Lambda(1+2s)}.
\]
Hence
\[
\left.\frac{m'(s,\triv)}{m(s,\triv)}\right|_{s=0} =\left.-2\frac{\Lambda'(1-2s)}{\Lambda(1-2s)}-2\frac{\Lambda'(1+2s)}{\Lambda(1+2s)}\right|_{s=0}
=-4\fp_{s=0}\frac{\Lambda'(1+2s)}{\Lambda(1+2s)}.
\]
It is easy to see that
\[
\fp_{s=0}\frac{\Lambda'(1+2s)}{\Lambda(1+2s)}=\upgamma-2\log 2-\log\uppi.
\]
Hence
\[
\Phi(0,\triv)=4\left(\upgamma-2\log 2-\log\uppi\right) \prod_{i=1}^{r}(1-q_i^{-1})^{-2}\mf{B}
\]
and thus we obtain the result.
\end{proof}

\begin{proposition}\label{prop:contributioncontinuous2}
We have
\[
\sum_{\substack{n<X\\\gcd(n,S)=1}}J_\cont^2(f^n)=\mf{F}X+o(X),
\]
where
\[
\mf{F}=-\frac14\prod_{i=1}^{r}(1-q_i^{-1})^2 \Tr\left(R(0,\triv)^{-1}R'(0,\triv)\xi_0(f)\right).
\]
\end{proposition}
\begin{proof}
The sum over $\mu$ in $J_\cont^2(f^n)$ is finite by \autoref{cor:characterfinite}. Let
\[
\Phi(s,\mu)= \sum_{v\in S}\Tr \left(R_v(s,\mu_v)^{-1}R_v'(s,\mu_v)\Ind_{\B(\QQ_v)}^{\G(\QQ_v)}(s,\mu_v)(f_v)\right)\prod_{\substack{w\neq v\\w\in S}}\Tr\left(\Ind_{\B(\QQ_w)}^{\G(\QQ_w)}(s,\mu_w)(f_w)\right).
\]
It is holomorphic on the imaginary axis with rapid decay by \autoref{thm:rapiddecay}. By \autoref{prop:estimatecontinuouscharacter2} and since the sum over $\mu$ is finite, we get
\[
\sum_{\substack{n<X\\ \gcd(n,S)=1}} J_\cont^2(f^n)=-\frac14\prod_{i=1}^{r}(1-q_i^{-1})^2 \Phi(0,\triv)X+o(X).
\]
We have
\[
\Phi(0,\triv)=\sum_{v\in S}\Tr \left(R_v(0,\triv)^{-1}R_v'(0,\triv)\xi_0(f_v)\right)\prod_{\substack{w\neq v\\w\in S}}\Tr\left(\xi_0(f_w)\right)=\Tr\left(R(0,\triv)^{-1}R'(0,\triv)\xi_0(f)\right).
\]
Hence we obtain the desired result.
\end{proof}

\section{Contribution of non-hyperbolic parts in the geometric side}

\subsection{Contribution of the identity part and the unipotent part} The contributions of the identity part and the unipotent part can be easily computed and they give small contributions.
\begin{proposition}\label{prop:identitycontribution}
We have
\[
\sum_{n<X}I_{\id}(f^n)\ll X^{\frac12},
\]
where the implied constant only depends on $f_\infty$ and $f_{q_i}$.
\end{proposition}
\begin{proof}
By the proof of Proposition 3.3 of \cite{cheng2025b}, $I_\id(f^n)=0$ if $n$ is not a square, and $I_\id(f^n)\ll_{f_\infty,f_{q_i}} 1$ if $n$ is a square. Hence 
\[
\sum_{\substack{n<X\\ \gcd(n,S)=1}} I_\id(f^n)\ll \sum_{\substack{n<X\\ \gcd(n,S)=1\\ n=\square}} 1\ll X^{\frac12}.\qedhere
\]
\end{proof}

\begin{proposition}\label{prop:unipotentcontribution}
For any $\varepsilon>0$, we have
\[
\sum_{\substack{n<X\\ \gcd(n,S)=1}} J_\unip(f^n)\ll X^{\frac12+\varepsilon},
\]
where the implied constant only depends on $m$, $f_{q_i}$ and $\varepsilon$.
\end{proposition}
\begin{proof}
By the proof of Proposition 3.4 of \cite{cheng2025b}, $J_\unip(f^n)=0$ if $n$ is not a square, and $J_\unip(f^n)\ll_{f_\infty,f_{q_i},\varepsilon} X^{\varepsilon}$ if $n$ is a square. Hence 
\[
\sum_{\substack{n<X\\ \gcd(n,S)=1}} J_\unip(f^n)\ll \sum_{\substack{n<X\\ \gcd(n,S)=1\\ n=\square}} X^\varepsilon\ll X^{\frac12+\varepsilon}.\qedhere
\]
\end{proof}

\subsection{Contribution of the elliptic part} In this subsection we will give an asymptotic formula for the elliptic part. Recall the main theorem in \cite{cheng2025c} is
\begin{theorem}\label{thm:contributeellipticfinal}
For any $\varepsilon>0$, we have
\begin{align*}
\sum_{\substack{n<X\\\gcd(n,S)=1}}I_\el(f^n) =&\mf{A}X^{\frac32}+\mf{B}(X\log X-X)+(\mf{C}+\mf{D})X\\
 -&\upgamma_S\sum_{\substack{n<X\\\gcd(n,S)=1}}\prod_{i=1}^{r}(1-q_i^{-1})\sum_{\pm}\sum_{\nu\in \ZZ^r}\sum_{\substack{a\neq b\in \ZZ^S\\ ab=\pm nq^\nu}}\theta_\infty\begin{pmatrix} a & 0 \\ 0 & b \end{pmatrix}\theta_q\begin{pmatrix}
 a & 0 \\ 0 & b\end{pmatrix} \\
 +&\sum_{\substack{n<X\\\gcd(n,S)=1}}\prod_{i=1}^{r}(1-q_i^{-1})\sum_{\pm}\sum_{\nu\in \ZZ^r}\sum_{\substack{a\neq b\in \ZZ^S\\ ab=\pm nq^\nu}}\theta_\infty\begin{pmatrix} a & 0 \\ 0 & b \end{pmatrix}\theta_q\begin{pmatrix}
 a & 0 \\ 0 & b\end{pmatrix}\sum_{d\mid (a-b)^{(q)}}\frac{\bm{\Lambda}(d)}{d}\\
 +&O(X^{\frac78+\varepsilon}),
\end{align*}
where $\mf{A}$ is defined in \cite[Theorem 7.10]{cheng2025c}, $\mf{B},\mf{C},\mf{D}$ are defined in \cite[Theorem 5.1]{cheng2025c}, $\upgamma_S$ is defined in \eqref{eq:defgammas} (also in \textnormal{(8.7)} of \cite{cheng2025c}).
The implied constant only depends on $\varepsilon$, $f_\infty$ and $f_q$.
\end{theorem}

The terms on the second line and the third line come from the square term $\Sigma^n(\square)$. They were expressed in \cite{cheng2025c} as
\[
-\upgamma_SI_{\hyp}^{\deg=1}(f^n)\quad\text{and}\quad-\widehat{J}_{\hyp}^{S}(f^n)
\]
respectively. In this paper, we will define $I_{\hyp}^{\deg=1}(f^n)$ and $\widehat{J}_{\hyp}^{S}(f^n)$ for more general functions and show that such definition coincides with this notation.

\section{Contribution of the hyperbolic part}
Recall that the hyperbolic part of the trace formula is given by
\[
J_\hyp(f)=-\frac{1}{2}\sum_{\gamma\in \A(\QQ)_{\reg}}\int_{\A(\AA)\bs \G(\AA)}f(g^{-1}\gamma g)\alpha(H_\B(wg)+H_\B(g))\rmd g,
\]
where $\A$ is the diagonal torus, $w$ is the nontrivial element in the Weyl group of $(\G,\A)$, $\alpha$ denotes the positive root in $\mathfrak{sl}_2$, $\B$ denotes the subgroup of upper triangular matrices, and $H_\B$ denotes the Harish-Chandra map. It can be split into local weighted orbital integrals for $f=\bigotimes_{v\in \mf{S}}'f_v$. We have
\begin{align*}
&\int_{\A(\AA)\bs \G(\AA)}f(g^{-1}\gamma g)\alpha(H_\B(wg)+H_\B(g))\rmd g \\
=&\sum_{v\in \mf{S}}\int_{\A(\QQ_v)\bs \G(\QQ_v)}f_v(g_v^{-1}\gamma g_v)\alpha(H_\B(wg_v)+H_\B(g_v))\rmd g_v\cdot\prod_{w\neq v} \int_{\A(\QQ_w)\bs \G(\QQ_w)}f_w(g_w^{-1}\gamma g_w)\rmd g_w.
\end{align*}
Hence we may write
\[
J_\hyp(f)=\sum_{v\in \mf{S}}J_{\hyp,v}(f),
\]
where the \emph{local hyperbolic part} $J_{\hyp,v}(f)$ is defined to be
\[
-\frac{1}{2}\sum_{\gamma\in \A(\QQ)_{\reg}}\int_{\A(\QQ_v)\bs \G(\QQ_v)}f_v(g_v^{-1}\gamma g_v)\alpha(H_\B(wg_v)+H_\B(g_v))\rmd g_v\cdot\prod_{w\neq v} \int_{\A(\QQ_w)\bs \G(\QQ_w)}f_w(g_w^{-1}\gamma g_w)\rmd g_w.
\]
Also, we define the \emph{local weighted orbital integral} as
\begin{equation}\label{eq:modifiedlocalweightedorbital}
\worb(f_v;\gamma)=\int_{\A(\QQ_v)\bs \G(\QQ_v)}f_v(g_v^{-1}\gamma g_v)\alpha(H_\B(wg_v)+H_\B(g_v))\rmd g_v
\end{equation}
for $\gamma=(\begin{smallmatrix}\gamma_1 &  \\   & \gamma_2 \end{smallmatrix})\in \A(\QQ_v)_\reg$. The test function $f_v$ satisfies $f_v\in C_c^\infty(\G(\QQ_v))$ if $v$ is nonarchimedean, and $f_v\in C_c^\infty(Z_+\bs\G(\RR))$ if $v=\infty$.

We also recall the notations from the \autoref{sec:hyperbolicpoisson}. The \emph{modified local weighted orbital integral (of the first kind)} is defined by
\begin{align*}
  \worb\sptilde(f;t) & =\int_{\A(\QQ_v)\bs \G(\QQ_v)}f(g^{-1}t g)\left(\alpha(H_\B(wg)+H_\B(g))-2\log\frac{|a-b|_v}{|ab|_v^{1/2}}\right)\rmd g \\
   & =\worb(f;t)-2\log\frac{|a-b|_v}{|ab|_v^{1/2}}\orb(f;t)
\end{align*}
for $t=(\begin{smallmatrix} a & 0 \\  0 & b \end{smallmatrix})$. It is $Z_+$-invariant with respect to $t$ when $v=\infty$.
Additionally, for nonarchimedean $v$, we define the \emph{modified local weighted orbital integral of the second kind} to be 
\begin{align*}
  \worb\sphat(f;t) & =\int_{\A(\QQ_v)\bs \G(\QQ_v)}f(g^{-1}t g)\left(\alpha(H_\B(wg)+H_\B(g))-2\log|a-b|_v\right)\rmd g \\
   & =\worb(f;t)-2\log|a-b|_v\orb(f;t)
\end{align*}
for $t=(\begin{smallmatrix}  a & 0 \\ 0  & b  \end{smallmatrix})$.

We define the \emph{degree $1$ term of the hyperbolic part} to be
\begin{equation}\label{eq:hyperbolicdeg1}
I_\hyp^{\deg=1}(f)=\sum_{\gamma\in \A(\QQ)_{\reg}}\int_{\A(\AA)\bs \G(\AA)}f(g^{-1}\gamma g)\rmd g.
\end{equation}

For any $v\in \mf{S}$ and $f=\bigotimes_{w\in \mf{S}}'f_w$, we define the \emph{modified local hyperbolic part (of the first kind)} as
\begin{align*}
  \widetilde{J}_{\hyp,v}(f)& =-\frac{1}{2}\sum_{\gamma\in \A(\QQ)_{\reg}}\worb\sptilde(f_v;\gamma)\prod_{w\neq v}\orb(f_w;\gamma)\\
   &=J_{\hyp,v}(f)+\sum_{\gamma\in \A(\QQ)_{\reg}}\log\frac{|\gamma_1-\gamma_2|_v}{|\gamma_1\gamma_2|_v^{1/2}}\int_{\A(\AA)\bs \G(\AA)}f(g^{-1}\gamma g)\rmd g,
\end{align*}
where $\gamma=(\begin{smallmatrix}\gamma_1 &  \\   & \gamma_2 \end{smallmatrix})$. Also, for nonarchimedean $v$, we define the \emph{modified local hyperbolic part of the second kind} as
\begin{align*}
  \widehat{J}_{\hyp,v}(f)& =-\frac{1}{2}\sum_{\gamma\in \A(\QQ)_{\reg}}\worb\sphat(f_v;\gamma)\prod_{w\neq v}\orb(f_w;\gamma)\\
   &=J_{\hyp,v}(f)+\sum_{\gamma\in \A(\QQ)_{\reg}}\log|\gamma_1-\gamma_2|_v\int_{\A(\AA)\bs \G(\AA)}f(g^{-1}\gamma g)\rmd g\\
   &=\widetilde{J}_{\hyp,v}(f)+\sum_{\gamma\in \A(\QQ)_{\reg}}\log|\gamma_1\gamma_2|_v^{1/2}\int_{\A(\AA)\bs \G(\AA)}f(g^{-1}\gamma g)\rmd g.
\end{align*}

Finally we define the \emph{unramified modified hyperbolic part of the first (resp. second) kind} to be
\[
\widetilde{J}_{\hyp}^S(f)=\sum_{v\notin S}\widetilde{J}_{\hyp,v}(f)\qquad\text{and}\qquad\widehat{J}_{\hyp}^S(f)=\sum_{v\notin S}\widehat{J}_{\hyp,v}(f),
\]
respectively.

\subsection{The relation between the square term and the hyperbolic parts}
In this subsection we will relate the summation occurring in $\Sigma^n(\square)$ to the degree $1$ term of the hyperbolic part and the unramified modified hyperbolic part of the second kind.

\begin{theorem}\label{thm:hyperbolicunweighted}
The degree $1$ term of the hyperbolic part for $f^n$ can be expressed as
\[
I_\hyp^{\deg=1}(f^n)=\prod_{i=1}^{r}(1-q_i^{-1})\sum_{\pm}\sum_{\nu\in \ZZ^r}\sum_{\substack{a\neq b\in \ZZ^S\\ ab=\pm nq^\nu}}\theta_\infty\begin{pmatrix} a & 0 \\ 0 & b \end{pmatrix}\theta_q\begin{pmatrix}
 a & 0 \\ 0 & b\end{pmatrix}.
\]
\end{theorem}
\begin{proof}
We have
\[
I_\hyp^{\deg=1}(f^n)=\sum_{\gamma\in \A(\QQ)_{\reg}}\orb(f^n;\gamma)=\sum_{\gamma\in \A(\QQ)_{\reg}}\prod_{v\in \mf{S}}\orb(f_v^n;\gamma).
\]
If $\orb(f^n;\gamma)$ is nonzero, then by Theorem 2.7 of \cite{cheng2025}, $\gamma\in \cO_{E_p}$ and $\mathopen{|}\det\gamma\mathclose{|}_p=p^{-n_p}$ for all $p\notin S$. Let $\gamma=(\begin{smallmatrix} a & 0 \\ 0 & b \end{smallmatrix})$. Then we have $a,b\in \ZZ^S$ and $ab=\pm nq^\nu$ with $\nu\in \ZZ^r$.

Now we assume that $a,b\in \ZZ^S$ and  $ab=\pm nq^\nu$ with $\nu\in \ZZ^r$. By Theorem 2.7 of \cite{cheng2025} we have
\begin{equation}\label{eq:unramifiedorbitalsplit}
\orb(f_p^n;\gamma)=p^{-n_p/2}\left(1+\sum_{j=1}^{k_p}p^j(1-p^{-1})\right)= p^{-n_p/2}p^{k_p}=\frac{p^{-n_p/2}}{|a-b|_p},
\end{equation}
where $k_p=v_p(a-b)$ for any prime $p$ (cf. \cite[Proposition 2.6]{cheng2025}).

Hence by the definition of $\theta_\infty(\gamma)$ and $\theta_{q_i}(\gamma)$ we have
\begin{align*}
\orb(f^n;\gamma)&=n^{-\frac12}\frac{|ab|_\infty^{1/2}}{|a-b|_\infty}\theta_\infty\begin{pmatrix} a & 0 \\ 0 & b \end{pmatrix}\prod_{i=1}^{r}\left(q_i^{-\nu_i/2}\theta_{q_i}\begin{pmatrix} a & 0 \\ 0 & b \end{pmatrix}
(1-q_i^{-1})\frac{1}{|a-b|_{q_i}}\right)\prod_{p\notin S}\frac{1}{|a-b|_p}\\
&=\prod_{i=1}^{r}(1-q_i^{-1})\theta_\infty\begin{pmatrix} a & 0 \\ 0 & b \end{pmatrix}\prod_{i=1}^{r}\theta_{q_i}\begin{pmatrix} a & 0 \\ 0 & b \end{pmatrix}.
\end{align*}
By summing over all possible $\gamma\in \A(\QQ)_\reg$ we obtain the result.
\end{proof}

Next we analyze the unramified modified hyperbolic part of the second kind. First we need to compute the unramified local weighted orbital integrals.
\begin{theorem}\label{thm:unramifiedweightedlocalorbital}
Let $p$ be a prime. Suppose that $\gamma=(\begin{smallmatrix}  a & 0 \\ 0 & b \end{smallmatrix})\in \A(\QQ_p)_\reg$. Let $k_p=v_p(a-b)$. Then
\[
\worb(\triv_{\cX_p^m};\gamma)=\frac{2}{|a-b|_p}\left(\log|a-b|_p+ \sum_{j=1}^{k_p}\frac{\log p}{p^j}\right)
\]
if $a,b\in \ZZ_p$ and $|ab|_p=p^{-m}$, and  $\worb(\triv_{\cX_p^m};\gamma)=0$ otherwise.
\end{theorem}
\begin{proof}
By Iwasawa decomposition we have
\[
\worb(\triv_{\cX_p^m};\gamma)=\int_{u\in\N(\QQ_p)}\int_{k\in\G(\ZZ_p)}\triv_{\cX_p^m}(k^{-1}u^{-1}\gamma uk)\rmd k\alpha(H_\B(wu))\rmd u,
\]
where $\N$ denotes the unipotent radical of $\B$. Suppose that $u=(\begin{smallmatrix} 1 & x \\  0 & 1 \end{smallmatrix})$. Then
\[
u^{-1}\gamma u=\begin{pmatrix}   a & x(b-a) \\  0 & b \end{pmatrix}.
\]
Hence $k^{-1}u^{-1}\gamma uk\in \cX_p^m$ if and only if $u^{-1}\gamma u\in \cX_p^m$, if and only if
\[
a,b,x(b-a)\in \ZZ_p\qquad\text{and}\qquad |ab|_p=p^{-m}.
\]

If $a\notin \ZZ_p$ or $b\notin \ZZ_p$ or $|ab|_p\neq p^{-m}$, then $k^{-1}u^{-1}\gamma uk\notin \cX_p^m$ for all $u\in \N(\QQ_p)$ and $k\in \G(\ZZ_p)$. Hence $\worb(\triv_{\cX_p^m};\gamma)=0$.

Now we assume that $a,b\in \ZZ_p$ and $|ab|_p=p^{-m}$. In this case, $k^{-1}u^{-1}\gamma uk\in \cX_p^m$ if and only if $x(a-b)\in \ZZ_p$, if and only if $v_p(x)\geq -k_p$. Since $\G(\ZZ_p)$ has volume $1$, we obtain
\[
\worb(\triv_{\cX_p^m};\gamma)=\int_{x\in p^{-k_p}\ZZ_p}\alpha(H_\B(wu))\rmd x,
\]
where $u=(\begin{smallmatrix}  1 & x \\  0 & 1 \end{smallmatrix})$. Using Iwasawa decomposition we find that
\begin{equation}\label{eq:nonarchimedeanweight}
\alpha(H_\B(wu))=\begin{cases}
                   2v_p(x)\log p, & v_p(x)<0, \\
                   0, & v_p(x)\geq 0.
                 \end{cases}
\end{equation}
Hence $\worb(\triv_{\cX_p^m};\gamma)$ equals
\[
-\sum_{j=1}^{k_p}\int_{x\in p^{-j}\ZZ_p^\times}2j\log p\rmd x=-2\log p(1-p^{-1})\sum_{j=1}^{k_p}jp^j =-2 p^{k_p}\left(k_p \log p-\sum_{j=0}^{k_p-1}\frac{\log p}{p^{k_p-j}}\right).
\]
Since $|a-b|_p=p^{-k_p}$, we obtain the result.
\end{proof}
By the explicit computation of the orbital integral (see Theorem 2.7 of \cite{cheng2025}), we obtain
\begin{corollary}\label{cor:unramifiedweightedlocalorbital}
Let $p$ be a prime. Suppose that $\gamma=(\begin{smallmatrix}  a & 0 \\ 0 & b \end{smallmatrix})\in \A(\QQ_p)_\reg$. Let $k_p=v_p(a-b)$. Then
\[
\worb\sphat(\triv_{\cX_p^m};\gamma)=\frac{2}{|a-b|_p} \sum_{j=1}^{k_p}\frac{\log p}{p^j}
\]
if $a,b\in \ZZ_p$ and $|ab|_p=p^{-m}$, and  $\worb\sphat(\triv_{\cX_p^m};\gamma)=0$ otherwise.
\end{corollary}

\begin{theorem}\label{thm:hyperbolicunramifiedmodify}
The unramified modified hyperbolic part for $f^n$ of the second kind can be expressed as
\[
\widehat{J}_{\hyp}^S(f^n)=-\prod_{i=1}^{r}(1-q_i^{-1})\sum_{\pm}\sum_{\nu\in \ZZ^r}\sum_{\substack{a\neq b\in \ZZ^S\\ ab=\pm nq^\nu}}\theta_\infty\begin{pmatrix} a & 0 \\ 0 & b \end{pmatrix}\theta_q\begin{pmatrix}
 a & 0 \\ 0 & b\end{pmatrix}\sum_{d\mid (a-b)^{(q)}}\frac{\bm{\Lambda}(d)}{d}.
\]
\end{theorem}
\begin{proof}
For any prime $p\notin S$, we have
\[
\widehat{J}_{\hyp,p}(f^n)=-\frac{1}{2}\sum_{\gamma\in \A(\QQ)_\reg}\prod_{v\in S}\orb(f_v;\gamma)\worb\sphat(f_p^n;\gamma)\cdot\prod_{\substack{\ell\notin S\\ \ell\neq p}}\orb(f_\ell^n;\gamma).
\]
Consider a single term
\begin{equation}\label{eq:singleunramifiedhyperbolic}
\prod_{v\in S}\orb(f_v;\gamma)\worb\sphat(f_p^n;\gamma)\cdot\prod_{\substack{\ell\notin S\\ \ell\neq p}}\orb(f_\ell^n;\gamma).
\end{equation}
By \eqref{eq:unramifiedorbitalsplit}, \autoref{thm:unramifiedweightedlocalorbital} and the definition of $\theta_\infty(\gamma)$ and $\theta_{q_i}(\gamma)$, we know that $\eqref{eq:singleunramifiedhyperbolic}=0$ unless $a,b\in \ZZ^S$ and $ab=\pm nq^\nu$ with $\nu\in \ZZ^r$. In this case, we have
\begin{align*}
  \eqref{eq:singleunramifiedhyperbolic} & =n^{-\frac12}\frac{|ab|_\infty^{1/2}}{|a-b|_\infty}\theta_\infty\begin{pmatrix} a & 0 \\ 0 & b \end{pmatrix}\prod_{i=1}^{r}\left(q_i^{-\nu_i/2}\theta_{q_i}\begin{pmatrix} a & 0 \\ 0 & b \end{pmatrix}
(1-q_i^{-1})\frac{1}{|a-b|_{q_i}}\right)\frac{2}{|a-b|_p}\sum_{j=1}^{k_p}\frac{\log p}{p^j}\prod_{\substack{\ell\notin S\\\ell\neq p}}\frac{1}{|a-b|_\ell} \\
   & =2\prod_{i=1}^{r}(1-q_i^{-1})\theta_\infty\begin{pmatrix} a & 0 \\ 0 & b \end{pmatrix}\prod_{i=1}^{r}\theta_{q_i}\begin{pmatrix} a & 0 \\ 0 & b \end{pmatrix}\sum_{j=1}^{v_p(a-b)}\frac{\log p}{p^j}.
\end{align*}
Hence we obtain
\begin{equation}\label{eq:modifyhyperboliccompute}
\widehat{J}_{\hyp,p}(f^n)=-\prod_{i=1}^{r}(1-q_i^{-1})\sum_{\pm}\sum_{\nu\in \ZZ^r}\sum_{\substack{a\neq b\in \ZZ^S\\ ab=\pm nq^\nu}}\theta_\infty\begin{pmatrix} a & 0 \\ 0 & b \end{pmatrix}\theta_q\begin{pmatrix}
 a & 0 \\ 0 & b\end{pmatrix}\sum_{j=1}^{v_p(a-b)}\frac{\log p}{p^j}.
\end{equation}

Since $\theta_{q}$ is compactly supported and $\theta_\infty$ is compactly supported modulo center, the sums over $a,b$ and $\nu$ are finite. Hence we may interchange the summation over $p\notin S$ and that over $a,b$ and $\nu$ and obtain
\begin{align*}
  \widehat{J}_{\hyp}^S(f^n)= & \sum_{p\notin S} \widehat{J}_{\hyp,p}(f^n)
   = -\prod_{i=1}^{r}(1-q_i^{-1})\sum_{\pm}\sum_{\nu\in \ZZ^r}\sum_{\substack{a\neq b\in \ZZ^S\\ ab=\pm nq^\nu}}\theta_\infty\begin{pmatrix} a & 0 \\ 0 & b \end{pmatrix}\theta_q\begin{pmatrix}
 a & 0 \\ 0 & b\end{pmatrix}\sum_{p\notin S}\sum_{j=1}^{v_p(a-b)}\frac{\log p}{p^j}\\
  =&-\prod_{i=1}^{r}(1-q_i^{-1})\sum_{\pm}\sum_{\nu\in \ZZ^r}\sum_{\substack{a\neq b\in \ZZ^S\\ ab=\pm nq^\nu}}\theta_\infty\begin{pmatrix} a & 0 \\ 0 & b \end{pmatrix}\theta_q\begin{pmatrix}
 a & 0 \\ 0 & b\end{pmatrix}\sum_{d\mid (a-b)^{(q)}}\frac{\bm{\Lambda}(d)}{d}.\qedhere
\end{align*}
\end{proof}

By \autoref{thm:contributeellipticfinal}, \autoref{thm:hyperbolicunweighted}, and \autoref{thm:hyperbolicunramifiedmodify} we obtain
\begin{corollary}\label{cor:contributionelliptic}
For any $\varepsilon>0$, we have
\begin{align*}
\sum_{\substack{n<X\\\gcd(n,S)=1}}I_\el(f^n) =&\mf{A}X^{\frac32}+\mf{B}(X\log X-X)+(\mf{C}+\mf{D})X\\ -&\upgamma_S\sum_{\substack{n<X\\\gcd(n,S)=1}}I_{\hyp}^{\deg=1}(f^n) -\sum_{\substack{n<X\\\gcd(n,S)=1}}\widehat{J}_{\hyp}^{S}(f^n) +O(X^{\frac{5}{6}+\frac{2}{3}\varrho+\varepsilon}).
\end{align*}
The implied constant only depends on $\varepsilon$, $f_\infty$ and $f_{q_i}$.
\end{corollary}

\subsection{The degree $1$ term of the trace formula} Recall that we also have the following degree $1$ term of the Arthur-Selberg trace formula for $\GL_2$:
\begin{equation}\label{eq:traceformuladeg1}
I_\hyp^{\deg=1}(f)+I_\unip^{\deg=1}(f)=I_\cont^{\deg=1}(f).
\end{equation}
The hyperbolic part is defined in \eqref{eq:hyperbolicdeg1}. The unipotent part is
\begin{equation}\label{eq:unipotentdeg1}
I_\unip^{\deg=1}(f)=\sum_{z\in \Z(\QQ)}\int_{\AA}F_z(x)\rmd x,
\end{equation}
where
\[
F_z(x)=\int_{K}f\left(k^{-1}z\begin{pmatrix}  1 & x \\ 0 & 1 \end{pmatrix}k\right)\rmd k
\]
and $K$ denotes the maximal compact subgroup of $\G(\AA)$. The continuous part is
\begin{equation}\label{eq:continuousdeg1}
I_\cont^{\deg=1}(f)=\frac{1}{2\uppi\rmi}\sum_{\mu}\int_{(0)}\Tr\left(\Ind_{\B(\AA)}^{\G(\AA)}(s,\mu)(f)\right)\rmd s,
\end{equation}
where $\mu$ runs over all characters on $\A(\QQ)\bs\A(\AA)^1$.

The degree $1$ term of the trace formula can be also derived by the hyperbolic Poisson summation formula (see \autoref{sec:hyperbolicpoisson}).

We now give the asymptotic formulas for the unipotent part and the continuous part.
\begin{proposition}\label{prop:contributionunipotentdeg1}
We have
\[
\sum_{\substack{n<X\\ \gcd(n,S)=1}}  I_\unip^{\deg=1}(f^n)\ll X^{\frac12},
\]
where the implied constant only depends on $f_\infty$ and $f_{q_i}$.
\end{proposition}
\begin{proof}
By Proposition 3.4 of \cite{cheng2025b} we know the following:
If $n$ is not a square, then $F_z$ is identically $0$ for any $z\in \Z(\QQ)$ so that $I_\unip^{\deg=1}(f^n)=0$ if $n$ is not a square. If $n$ is a square, $F_z\neq 0$ only if $z=(\begin{smallmatrix} a & 0 \\  0 & a \end{smallmatrix})$ with $a=\pm n^{1/2}q^{\nu/2}$ for some $\nu/2\in \ZZ^r$. Moreover, the number of $\nu$ is finite and the number is independent of $n$.

Let
\[
F_{z,v}(x_v)=\int_{K_v}f_v^n\left(k_v^{-1}z\begin{pmatrix}
                                 1 & x_v \\
                                 0 & 1 
                               \end{pmatrix}k_v\right)\rmd k_v.
\]
Then we have
\[
\int_{\AA}F_z(x)\rmd x=\prod_{v\in \mf{S}}\int_{\QQ_v}F_{z,v}(x_v)\rmd x_v.
\]
Now we compute the local integrals.

\underline{\emph{Case 1:}}\ \ $v=p\notin S$.

In this case, $f_v^n$ is bi-$K_v$-invariant. Hence 
\[
F_{z,v}(x_v)=f_v^n\left(n^{1/2}q^{\nu/2}\begin{pmatrix}
                                 1 & x_v \\
                                 0 & 1 
                               \end{pmatrix}\right).
\]
Thus $F_{z,v}(x_v)=p^{-n_p/2}$ if $|x_p|_p\leq p^{n_p/2}$, and is $0$ otherwise. Hence
\[
\int_{\QQ_v}F_{z,v}(x_v)\rmd x_v=p^{-n_p/2}\vol(p^{-n_p/2}\ZZ_p)=1.
\]

\underline{\emph{Case 2:}}\ \ $v\in S$.

Since $f_v$ is compactly supported for $v$ nonarchimedean and $f_\infty$ is compactly supported modulo center, 
\[
F_{z,v}(x_v)=\int_{K_v}f_v\left(k_v^{-1}z\begin{pmatrix}
                                 1 & x_v \\
                                 0 & 1 
                               \end{pmatrix}k_v\right)\rmd k_v
\]
is compactly supported on $\QQ_v$. Hence 
\[
\int_{\QQ_v}F_{z,v}(x_v)\rmd x_v\ll_{f_v} 1.
\]

By the above two cases
\[
\int_{\AA}F_z(x)\rmd x=\prod_{v\in \mf{S}}\int_{\QQ_v}F_{z,v}(x_v)\rmd x_v\ll_{f_\infty,f_{q_i}} 1.
\]

Assume that $n$ is a square. In this case, for $z\in \Z(\QQ)$ contributing the sum, we must have $z=(\begin{smallmatrix} a & 0 \\  0 & a \end{smallmatrix})$ with $a=\pm n^{1/2}q^{\nu/2}$ for some $\nu/2\in \ZZ^r$, and the choice of $\nu$ is finite. Hence we conclude that
\[
I_\unip^{\deg=1}(f^n)=\sum_{z\in \Z(\QQ)}\int_{\AA}F_z(x)\rmd x\ll_{f_\infty, f_{q_i}} 1.
\]
Since $I_\unip^{\deg=1}(f^n)=0$ if $n$ is not a square, we obtain
\[
\sum_{\substack{n<X\\ \gcd(n,S)=1}} I_\unip^{\deg=1}(f^n)\ll \sum_{\substack{n<X\\ \gcd(n,S)=1\\ n=\square}} 1\ll X^{\frac12}.\qedhere
\]
\end{proof}

\begin{proposition}\label{prop:contributioncontinuousdeg1}
For any $\varepsilon>0$ we have
\[
\sum_{\substack{n<X\\ \gcd(n,S)=1}} I_\cont^{\deg=1}(f^n)=-\mf{B}X+O(X^{\frac23+\varepsilon}),
\]
where the implied constant only depends on $f_\infty$, $f_{q_i}$ and $\varepsilon$.
\end{proposition}
\begin{proof}
We have
\[
I_\cont^{\deg=1}(f^n)=\frac{1}{2\uppi\rmi}\sum_{\mu}\int_{(0)} \prod_{v\in S}\Tr\left(\Ind_{\B(\QQ_v)}^{\G(\QQ_v)}(s,\mu_v)(f_v)\right)\cdot\prod_{p\notin S}\Tr\left(\Ind_{\B(\QQ_p)}^{\G(\QQ_p)}(s,\mu_p)(f_p)\right)\rmd s.
\]
Also, the sum over $\mu$ is finite as shown in \cite[Corollary 3.11]{cheng2025b}. By \autoref{lem:unramifiedtracecontinuous} we obtain
\[
I_\cont^{\deg=1}(f^n)=\frac{1}{2\uppi\rmi}\sum_{\mu}\int_{(0)} \prod_{v\in S}\Tr\left(\Ind_{\B(\QQ_v)}^{\G(\QQ_v)}(s,\mu_v)(f_v)\right)\frac{1}{n^{s}}\sum_{d\mid n}d^{2s}\mu_1(d)\mu_2\legendresymbol{n}{d}\rmd s.
\]
Let
\[
\Phi(s,\mu)= \prod_{v\in S}\Tr\left(\Ind_{\B(\QQ_v)}^{\G(\QQ_v)}(s,\mu_v)(f_v)\right).
\]
Note that 
\[
\Tr\left(\Ind_{\B(\RR)}^{\G(\RR)}(s,\mu_\infty)(f_\infty)\right)
\]
has rapid decay vertically since it is a Mellin transform of a compactly supported function. Moreover, the nonarchimedean parts are bounded holomorphic functions (cf. the proof of \cite[Proposition 3.7]{cheng2025b}). Hence $\Phi(s,\mu)$ has rapid decay vertically on $s$. By \autoref{prop:estimatecontinuouscharacter} and the fact that the sum over $\mu$ is finite, we get
\[
\sum_{\substack{n<X\\ \gcd(n,S)=1}} I_\cont^{\deg=1}(f^n)=\frac12\prod_{i=1}^{r}(1-q_i^{-1})^2 \Phi(0,\triv)X+O(X^{\frac23+\varepsilon}).
\]
Since
\[
\Phi(0,\triv)= \prod_{v\in S}\Tr\left(\xi_0(f_v)\right)=-2\prod_{i=1}^{r}(1-q_i^{-1})^{-2}\mf{B}
\]
by \eqref{eq:ramifiedproducteisenstein}, we obtain the result.
\end{proof}
Combining \autoref{prop:contributionunipotentdeg1} and \autoref{prop:contributioncontinuousdeg1} we obtain the asymptotic formula for the degree $1$ terms of the hyperbolic part.

\begin{corollary}\label{cor:contributionhyperbolicdeg1}
For any $\varepsilon>0$ we have
\[
\sum_{\substack{n<X\\ \gcd(n,S)=1}} I_\hyp^{\deg=1}(f^n)=-\mf{B}X+O(X^{\frac23+\varepsilon}),
\]
where the implied constant only depends on $f_\infty$, $f_{q_i}$ and $\varepsilon$.
\end{corollary}

\subsection{Relation between the unramified modified local hyperbolic parts}
In this subsection we will give a relation between the two kinds of modified local hyperbolic parts.
\begin{proposition}
We have
\[
\widehat{J}_{\hyp}^S(f^n)-\widetilde{J}_{\hyp}^S(f^n)=-\frac12\log nI_{\hyp}^{\deg=1}(f^n).
\]
\end{proposition}
\begin{proof}
By definition we have
\[
\widehat{J}_{\hyp,p}(f^n)-\widetilde{J}_{\hyp,p}(f^n)=\sum_{\gamma\in \A(\QQ)_\reg}\log|\gamma_1\gamma_2|_p^{1/2}\int_{\A(\AA)\bs \G(\AA)}f^n(g^{-1}\gamma g)\rmd g.
\]
Recall that $\orb(f;\gamma)\neq 0$ only if $\det\gamma=\gamma_1\gamma_2=\pm nq^\nu$. Hence
\[
\widehat{J}_{\hyp,p}(f^n)-\widetilde{J}_{\hyp,p}(f^n)=\frac12\sum_{\gamma\in \A(\QQ)_\reg}\log p^{-n_p}\int_{\A(\AA)\bs \G(\AA)}f^n(g^{-1}\gamma g)\rmd g=-\frac12\log p^{n_p}I_{\hyp}^{\deg=1}(f^n).
\]
Summing over all $p\notin S$ yields the result.
\end{proof}

\begin{proposition}
For any $\varepsilon>0$ we have
\[
\sum_{\substack{n<X\\ \gcd(n,S)=1}}\log n I_\hyp^{\deg=1}(f^n)=-\mf{B}(X\log X-X)+O(X^{\frac23+\varepsilon}),
\]
where the implied constant only depends on $f_\infty$, $f_{q_i}$ and $\varepsilon$.
\end{proposition}
\begin{proof}
Let
\[
H(X)=\sum_{\substack{n<X\\ \gcd(n,S)=1}}I_\hyp^{\deg=1}(f^n)=-\mf{B}X+O(X^{\frac23+\varepsilon})
\]
by \autoref{cor:contributionhyperbolicdeg1}. By partial summation formula we have
\begin{align*}
\sum_{\substack{n<X\\ \gcd(n,S)=1}}\log n I_\hyp^{\deg=1}(f^n)&=H(X)\log X-\int_{1}^{X}\frac{H(x)}{x}\rmd x\\
&=-\mf{B}X\log X-\int_{1}^{X}(-\mf{B})\rmd x+O(X^{\frac23+\varepsilon}\log X)-\int_{1}^{X}O(x^{-\frac13+\varepsilon})\rmd x\\
&=-\mf{B}(X\log X-X)+O(X^{\frac23+\varepsilon'}).\qedhere
\end{align*}
\end{proof}
\begin{corollary}\label{cor:relationhyperbolic}
For any $\varepsilon>0$ we have
\[
\sum_{\substack{n<X\\ \gcd(n,S)=1}}\widetilde{J}_{\hyp}^S(f^n)=\sum_{\substack{n<X\\ \gcd(n,S)=1}}\widehat{J}_{\hyp}^S(f^n)-\frac12\mf{B}(X\log X-X)+O(X^{\frac23+\varepsilon}).
\]
\end{corollary}

\subsection{Estimate of the unramified modified local hyperbolic part}
For any $v\in \mf{S}$ and $f=\bigotimes_{w\in \mf{S}}'f_w$, recall that
\begin{align*}
  \widetilde{J}_{\hyp,v}(f)
   &=J_{\hyp,v}(f)+\sum_{\gamma\in \A(\QQ)_{\reg}}\log\frac{|\gamma_1-\gamma_2|_v}{|\gamma_1\gamma_2|_v^{1/2}}\int_{\A(\AA)\bs \G(\AA)}f(g^{-1}\gamma g)\rmd g,
\end{align*}
where $\gamma=(\begin{smallmatrix}\gamma_1 &  \\   & \gamma_2 \end{smallmatrix})$. Since for $\alpha\in \QQ$ we have
\[
\sum_{v\in \mf{S}}\log |\alpha|_v=0,
\]
we conclude that
\begin{equation}\label{eq:hyperboliclocal}
J_\hyp(f)=\sum_{v\in \mf{S}}J_{\hyp,v}(f)=\sum_{v\in \mf{S}}\widetilde{J}_{\hyp,v}(f).
\end{equation}
In this subsection we will give asymptotic formulas for
\[
\sum_{\substack{n<X\\ \gcd(n,S)=1}}\widetilde{J}^S_{\hyp}(f^n)=\sum_{\substack{n<X\\ \gcd(n,S)=1}}\sum_{p\notin S}\widetilde{J}_{\hyp,p}(f^n)
\]
and
\[
\sum_{\substack{n<X\\ \gcd(n,S)=1}}\widehat{J}^S_{\hyp}(f^n)=\sum_{\substack{n<X\\ \gcd(n,S)=1}}\sum_{p\notin S}\widehat{J}_{\hyp,p}(f^n).
\]
By \autoref{cor:relationhyperbolic}, it suffices to consider the latter. We will use \autoref{thm:modifieddegree1trace2} to reduce the computation to
\[
\sum_{\substack{n<X\\ \gcd(n,S)=1}}\widehat{J}_{\unip,p}(f^n)
\qquad\text{and}\qquad \sum_{\substack{n<X\\ \gcd(n,S)=1}}\widehat{J}_{\cont,p}(f^n),
\]
and then sum over all $p\notin S$. Unfortunately, the series
\[
\sum_{p\notin S}\widehat{J}_{\unip,p}(f^n)
\]
\emph{diverges} if $n$ is a square. However, we can consider the truncation
\[
\sum_{\substack{p<CX\\p\notin S}}\widehat{J}_{\unip,p}(f^n).
\]

\begin{lemma}\label{lem:modifiedhyperbolicsum}
There exists a constant $C$ (depending only on $f_{\infty}$ and $f_{q_i}$) such that $\widehat{J}_{\hyp,p}(f^n)=0$ if $p>Cn$.
\end{lemma}
\begin{proof}
Recall that $\widehat{J}_{\hyp,p}(f^n)$ is given by \eqref{eq:modifyhyperboliccompute}. Since $\theta_{q_i}(\gamma)$ is compactly supported, we know that $a_{(q)}$ and $b_{(q)}$ have finitely many choices. Hence there exist $\alpha,\beta\in \ZZ_{>0}^r$ such that
\[
-\alpha_i\leq v_{q_i}(a),v_{q_i}(b)\leq \beta_i
\]
for $a$ and $b$ contributing the sum and all $i\in \{1,\dots,r\}$.

Since $ab=\pm nq^\nu$, we have $(ab)^{(q)}=n$. Then
\[
|a-b|\leq |nq^\beta-q^{-\alpha}|\leq nq^\beta.
\]
Moreover, $aq^\alpha-bq^\alpha\in \ZZ$.

Let $C=q^{\alpha+\beta}$, then
\[
|aq^\alpha-bq^\alpha|=q^\alpha|a-b|\leq Cn.
\]
Hence for any prime $p\notin S$ and $p>Cn$, we have $q^\alpha |a-b|<p$. Thus $v_p(a-b)=0$ for such $p$. By the expression of  \eqref{eq:modifyhyperboliccompute} we know that $\widehat{J}_{\hyp,p}(f^n)$ vanishes.
\end{proof}

Hence we obtain
\[
\sum_{\substack{n<X\\ \gcd(n,S)=1}}\widehat{J}_{\hyp}^S(f^n)=\sum_{\substack{n<X\\ \gcd(n,S)=1}}\sum_{\substack{p<CX\\ p\notin S}}\widehat{J}_{\hyp,p}(f^n).
\]

\begin{proposition}\label{prop:modifyunipotentestimate}
For any $\varepsilon>0$ we have
\[
\sum_{\substack{n<X\\ \gcd(n,S)=1}}\sum_{\substack{p<CX\\ p\notin S}}\widehat{J}_{\unip,p}(f^n)\ll X^{\frac12+\varepsilon},
\]
where the implied constant only depends on $f_\infty$, $f_{q_i}$, $C$ and $\varepsilon$.
\end{proposition}
\begin{proof}
For $p\notin S$ we have
\begin{align*}
  \widehat{J}_{\unip,p}(f^n)=&\sum_{z\in \Z(\QQ)}\frac{1}{|a|_p}\int_{\QQ_p}\int_{K_p}f_p^n\left(k_p^{-1}\begin{pmatrix}
                                                a & x_p \\
                                                0 & a 
                                              \end{pmatrix} k_p\right)\rmd k_p\log\mathopen{|}x_p\mathclose{|}_p\rmd x_p\prod_{w\neq p}\int_{\QQ_w}F_z(x_w)\rmd x_w,
\end{align*}
where $z=(\begin{smallmatrix}
            a & 0 \\
            0 & a 
          \end{smallmatrix})$. $z\in \Z(\QQ)$ contributing the sum only if 
\[
\left|\det\left(k_p^{-1}\begin{pmatrix} a & x_p \\   0 & a \end{pmatrix} k_p\right)\right|=p^{-n_p}
\]
and $\mathopen{|}\det z\mathclose{|}_\ell=\ell^{-n_\ell}$ for any $\ell\notin S\cup\{p\}$ by the first paragraph of the proof of \autoref{prop:contributionunipotentdeg1}. Hence $n$ must be a square. 

Now we assume that $n$ is a square. $F_z\neq 0$ only if $z=(\begin{smallmatrix} a & 0 \\  0 & a \end{smallmatrix})$ with $a=\pm n^{1/2}q^{\nu/2}$ for some $\nu/2\in \ZZ^r$. Moreover, the number of $\nu$ is finite and the number is independent of $n$. Now we fix such $z$. By the proof of \autoref{prop:contributionunipotentdeg1} we have
\[
\int_{\QQ_w}F_z(x_w)\rmd x_w=1
\]
for $w\notin S$ with $w\neq p$. Moreover,
\[
\int_{\QQ_w}F_z(x_w)\rmd x_w\ll_{f_w} 1
\]
for $w\in S$. The key step is to compute
\begin{equation}\label{eq:unipotentmodifydeg1}
\frac{1}{|a|_p}\int_{\QQ_p}\int_{K_p}f_p^n\left(k_p^{-1}\begin{pmatrix}
                                                a & x_p \\
                                                0 & a 
                                              \end{pmatrix} k_p\right)\rmd k_p\log|x_p|_p\rmd x_p.
\end{equation}
By definition of $f_p^n$, we have
\[
\int_{K_p}f_p^n\left(k_p^{-1}\begin{pmatrix}   a & x_p \\ 0 & a \end{pmatrix} k_p\right)\rmd k_p=\begin{cases}
      p^{-n_p/2}, & x_p\in \ZZ_p, \\
      0, & x_p\notin \ZZ_p.
    \end{cases}
\]
Hence
\[
\eqref{eq:unipotentmodifydeg1}=p^{n_p/2}\int_{\ZZ_p}p^{-n_p/2}\log|x_p|\rmd x_p
=-\sum_{j=0}^{+\infty} p^{-j}(1-1/p) j\log p=-\frac{\log p}{p-1}.
\]
Combining all things above, we obtain
\[
\sum_{\substack{p<CX\\p\notin S}}\widehat{J}_{\unip,p}(f^n)\ll_{f_\infty,f_q} \sum_{p<CX}\frac{\log p}{p-1}\ll \log X,
\]
where the last step is the Merten's estimate. 

Since $\widehat{J}_{\unip,p}(f^n)=0$ if $n$ is not a square, we obtain
\[
\sum_{\substack{n<X\\ \gcd(n,S)=1}}\sum_{\substack{p<CX\\p\notin S}}\widehat{J}_{\unip,p}(f^n)\ll \sum_{\substack{n<X\\ \gcd(n,S)=1\\ n=\square}} \log X\ll X^{\frac12+\varepsilon}.\qedhere
\]
\end{proof}

\begin{theorem}\label{thm:modifycontinuousestimate}
For any $\varepsilon>0$ we have
\[
\sum_{\substack{n<X\\ \gcd(n,S)=1}}\sum_{\substack{p<CX\\ p\notin S}}\widehat{J}_{\cont,p}(f^n) =\sum_{p\notin S}\frac{\log p}{p^2-1}\mf{B}X+O(X^{\frac23+\varepsilon}),
\]
where the implied constant only depends on $f_\infty$, $f_{q_i}$, $C$ and $\varepsilon$.
\end{theorem}
\begin{proof}
We have
\[ 
\sum_{\substack{n<X\\ \gcd(n,S)=1}}\widehat{J}_{\cont,p}(f^n)=\sum_{m=0}^{+\infty}\sum_{\substack{a<X/p^m\\ \gcd(a,S\cup\{p\})=1}}\widehat{J}_{\cont,p}(f^{p^ma})
\]
 By definition of the test function $f^n$, we have $f_\ell^{p^ma}=f_\ell^a$ for $\ell\notin S\cup\{p\}$, and $f_p^{p^ma}=f_p^{p^m}$. 

\underline{\emph{Step 1:}}\ \ Contribution of $\widehat{J}_{\cont,p}(f^{p^ma})$.

By definition we have
\begin{align*}
\widehat{J}_{\cont,p}(f^{p^ma})=&-\frac{1}{4\uppi\rmi}\sum_{\mu}\int_{(0)}\widehat{\Tr} \left(\Ind_{\B(\QQ_p)}^{\G(\QQ_p)}(s,\mu_p)(f_p^{p^m})\right)\prod_{w\in S}\Tr\left(\Ind_{\B(\QQ_w)}^{\G(\QQ_w)}(s,\mu_w)(f_w)\right)\\
\times&\prod_{\ell\notin S\cup\{p\}}\Tr\left(\Ind_{\B(\QQ_\ell)}^{\G(\QQ_\ell)}(s,\mu_\ell)(f_\ell^a)\right)\rmd s.
\end{align*}
By \autoref{lem:unramifiedtracecontinuous}, we know that
if $\mu_1$ and $\mu_2$ are unramified at all places outside $S\cup\{p\}$, then
\[
\prod_{\ell\notin S\cup\{p\}}\Tr\left(\Ind_{\B(\QQ_\ell)}^{\G(\QQ_\ell)}(s,\mu_\ell)(f_\ell^{a})\right)=\frac{1} {a^{s}}\sum_{d\mid a}d^{2s}\mu_1(d)\mu_2\legendresymbol{a}{d},
\]
and the above equals $0$ otherwise.
 
Let
\[
\Phi(s,\mu)=\widehat{\Tr} \left(\Ind_{\B(\QQ_p)}^{\G(\QQ_p)}(s,\mu_p)(f_p^{p^m})\right)\prod_{w\in S}\Tr\left(\Ind_{\B(\QQ_w)}^{\G(\QQ_w)}(s,\mu_w)(f_w)\right).
\]
$\Phi(s,\mu)$ satisfies the conditions of  \autoref{prop:estimatecontinuouscharacter}: By \eqref{eq:defwidehattrace} and the proof of \autoref{thm:modifieddegree1trace}, the choice of $\mu_1\mu_2$ contributing the sum is finite. For $\ell\neq p$, $C(\mu_{1,\ell})$ and $C(\mu_{2,\ell})$ both have finitely many choices, and for $\ell\notin S\cup\{p\}$, $C(\mu_{1,\ell})=C(\mu_{2,\ell})=1$. Hence (i) and (ii) for \autoref{prop:estimatecontinuouscharacter} hold. For (iii), since the choice of $\mu_1\mu_2$ is finite, we have $C(\mu_1)\asymp C(\mu_2)$. The conclusion now follows from \autoref{prop:estimategermnonarchimedean2}.

Hence by \autoref{prop:estimatecontinuouscharacter} and \autoref{prop:explicitmodifiedtrace2} we have
\begin{equation}\label{eq:contributionunramifiedhyperbolic}
\begin{split}
   &\sum_{\substack{a<X/p^m\\ \gcd(a,S\cup\{p\})=1}}\widehat{J}_{\cont,p}(f^{p^ma})\\
  =&-\frac12 \prod_{i=1}^{r}(1-q_i^{-1})^2(1-p^{-1})^2\widehat{\Tr} \left(\Ind_{\B(\QQ_p)}^{\G(\QQ_p)}(0,\triv)(f_p^{p^m})\right)\prod_{w\in S}\Tr(\xi_0(f_w))\frac{X}{p^m} \\
  +&O\left(\frac{\log p}{p-1}\legendresymbol{X}{p^m}^{\frac23+\varepsilon}\right).
\end{split}
\end{equation}

\underline{\emph{Step 2:}}\ \ Simplification of $\widehat{\Tr} \left(\Ind_{\B(\QQ_p)}^{\G(\QQ_p)}(0,\triv)(f_p^{p^m})\right)$.

We have proved in \autoref{subsec:modifieddeg1} that
\[
\widehat{\Tr}\left(\Ind_{\B(\QQ_v)}^{\G(\QQ_v)}(0,\triv)(f_p^{p^m})\right)=\int_{ \A(\QQ_v)}\frac{|a-b|_v}{|ab|_v^{1/2}}\worb\sphat(f_p^n;t)\rmd t.
\]
By \autoref{cor:unramifiedweightedlocalorbital}, the above can be written as
\begin{align*}
\int_{a,b\in\ZZ_p,\,v_p(ab)=m}\frac{|a-b|}{|ab|^{1/2}}\frac{2}{|a-b|}p^{-m/2} \sum_{j=1}^{v_p(a-b)}\frac{\log p}{p^j}\rmd t=&2\sum_{k=0}^{m}\int_{p^k\ZZ_p^\times}\int_{p^{m-k}\ZZ_p^\times}\sum_{j=1}^{v_p(a-b)}\frac{\log p}{p^j}\rmd ^\times a\rmd^\times b\\
=&2\sum_{k=0}^{m}\int_{\ZZ_p^\times\times\ZZ_p^\times}\sum_{j=1}^{v_p(p^ka-p^{m-k}b)}\frac{\log p}{p^j}\rmd ^\times a\rmd^\times b.
\end{align*}

\underline{\emph{Case 1:}}\ \ $m$ is odd.

In this case we have $k\neq m-k$ for any $k$. Hence $v_p(p^ka-p^{m-k}b)=\min\{k,m-k\}$. Thus we have
\[
\widehat{\Tr}\left(\Ind_{\B(\QQ_v)}^{\G(\QQ_v)}(0,\triv)(f_p^{p^m})\right)=2\sum_{k=0}^{m} \sum_{j=1}^{\min\{k,m-k\}}\frac{\log p}{p^j}.
\]

\underline{\emph{Case 2:}}\ \ $m$ is even.

In this case, for $k\neq m/2$ we have $v_p(p^ka-p^{m-k}b)=\min\{k,m-k\}$.
For $k=m/2$, the summand becomes
\begin{align*}
\int_{\ZZ_p^\times\times\ZZ_p^\times}\sum_{j=1}^{m/2+v_p(a-b)}\frac{\log p}{p^j}\rmd^\times a \rmd^\times b=&\int_{\ZZ_p^\times\times\ZZ_p^\times}\sum_{j=1}^{m/2+v_p(1-b)}\frac{\log p}{p^j}\rmd^\times a\rmd^\times b\\
=&(1-p^{-1})^{-1}\int_{\ZZ_p^\times}\sum_{j=1}^{m/2+v_p(1-b)}\frac{\log p}{p^j}\rmd b,
\end{align*}
where in the first step we made change of variable $b\mapsto ab$. Now we split the integral as
\begin{align*}
\int_{\ZZ_p^\times}\sum_{j=1}^{m/2+v_p(1-b)}\frac{\log p}{p^j}\rmd b&=\int_{\ZZ_p^\times -(1+p\ZZ_p)}\sum_{j=1}^{m/2}\frac{\log p}{p^j}\rmd b+\sum_{i=1}^{+\infty}\int_{1+p^i\ZZ_p^\times}\sum_{j=1}^{m/2+i}\frac{\log p}{p^j}\rmd b\\
&=\frac{p-2}{p}\sum_{j=1}^{m/2}\frac{\log p}{p^j}+\sum_{i=1}^{+\infty}p^{-i}(1-p^{-1})\sum_{j=1}^{m/2+i}\frac{\log p}{p^j}\\
&=\frac{p-1}{p}\sum_{j=1}^{m/2}\frac{\log p}{p^j}-\frac{1}{p}\frac{\log p}{p}\frac{1-p^{-m/2}}{1-p^{-1}}+\sum_{i=1}^{+\infty}p^{-i-1}(1-p^{-m/2-i})\log p\\
&=\frac{p-1}{p}\sum_{j=1}^{m/2}\frac{\log p}{p^j}+\left[-\frac{1-p^{-m/2}}{p(p-1)}+\frac{1}{p(p-1)}-\frac{p^{-m/2}}{p(p^2-1)}\right]\log p.
\end{align*}
Hence we obtain
\begin{align*}
\int_{\ZZ_p^\times\times\ZZ_p^\times}\sum_{j=1}^{\frac m2+v_p(a-b)}\frac{\log p}{p^j}\rmd^\times a \rmd^\times b&=\sum_{j=1}^{\frac m2}\frac{\log p}{p^j}+
(1-p^{-1})^{-1}\!\left[-\frac{1-p^{-\frac m2}}{p(p-1)}+\frac{1}{p(p-1)}-\frac{p^{-\frac m2}}{p(p^2-1)}\right]\log p\\
&=\sum_{j=1}^{m/2}\frac{\log p}{p^j}+\frac{p^{-\frac m2+1}}{p^2-1}\frac{\log p}{p-1}.
\end{align*}
Therefore
\[
\widehat{\Tr}\left(\Ind_{\B(\QQ_v)}^{\G(\QQ_v)}(0,\triv)(f_p^{p^m})\right)=2\sum_{k=0}^{m} \sum_{j=1}^{\min\{k,m-k\}}\frac{\log p}{p^j}+2\frac{p^{-m/2+1}}{p^2-1}\frac{\log p}{p-1}.
\]

\underline{\emph{Step 3:}}\ \ Summing over $m$.

By \eqref{eq:contributionunramifiedhyperbolic} and \eqref{eq:ramifiedproducteisenstein} we know that
\begin{align*}
\sum_{m=0}^{+\infty}
   \sum_{\substack{a<X/p^m\\ \gcd(a,S\cup\{p\})=1}}\widehat{J}_{\cont,p}(f^{p^ma}) =&\frac12(1-p^{-1})^2\sum_{m=0}^{+\infty}\widehat{\Tr} \left(\Ind_{\B(\QQ_p)}^{\G(\QQ_p)}(0,\triv)(f_p^{p^m})\right)\frac{X}{p^m}\mf{B}\\  +&\sum_{m=0}^{+\infty}O\left(\frac{\log p}{p-1}(X/p^m)^{\frac23+\varepsilon}\right).
\end{align*}
The error term is bounded by
\[
\frac{\log p}{p-1}\sum_{m=0}^{+\infty}\frac{X^{\frac23+\varepsilon}}{p^{m(\frac23+\varepsilon)}}\ll X^{\frac23+\varepsilon}\frac{\log p}{p-1}.
\]
The main term dividing $\mf{B}X$ becomes
\[
(1-p^{-1})^2\sum_{m=0}^{+\infty}\sum_{k=0}^{m} \sum_{j=1}^{\min\{k,m-k\}}\frac{\log p}{p^{m+j}}+(1-p^{-1})^2\sum_{\substack{m=0\\ 2\mid m}}^{+\infty}\frac{1}{p^m}\frac{p^{-m/2+1}}{p^2-1}\frac{\log p}{p-1}.
\]
The first term equals
\begin{align*}
&(1-p^{-1})^2\sum_{a=0}^{+\infty}\sum_{b=0}^{+\infty}\sum_{j=1}^{\min\{a,b\}}\frac{\log p}{p^{a+b+j}}=2\sum_{a=0}^{+\infty}\sum_{b=0}^{a-1}\sum_{j=1}^{b}\frac{\log p}{p^{a+b+j}}+ \sum_{a=0}^{+\infty}\sum_{j=1}^{a}\frac{\log p}{p^{2a+j}}\\
=&2(1-p^{-1})^2\sum_{a=0}^{+\infty}\sum_{b=0}^{a-1}\frac{\log p}{p^{a+b}}\frac{1-p^{-b}}{p-1} +\sum_{a=0}^{+\infty}\frac{1}{p^{2a}}\frac{1-p^{-a}}{p-1} \log p\\
=&2(1-p^{-1})^2\frac{\log p}{p-1}\sum_{a=0}^{+\infty}\frac{1}{p^a}\left(\frac{1-p^{-a}}{1-p^{-1}} -\frac{1-p^{-2a}}{1-p^{-2}}\right) +\frac{\log p}{p-1}\sum_{a=0}^{+\infty}\frac{1}{p^{2a}}(1-p^{-a})\\
=&2(1-p^{-1})^2\frac{\log p}{p-1}\left(\frac{1}{(1-p^{-1})^2}-\frac{2}{(1-p^{-1})(1-p^{-2})}+\frac{1}{(1-p^{-2})(1-p^{-3})}\right)\\
+&\frac{\log p}{p-1}\left(\frac{1}{1-p^{-2}}-\frac{1}{1-p^{-3}}\right)
\end{align*}
and the second term equals
\[
(1-p^{-1})^2\sum_{a=0}^{+\infty}\frac{1}{p^{2a}}\frac{p^{-a+1}}{p^2-1}\frac{\log p}{p-1}= (1-p^{-1})^2\frac{\log p}{p-1}\frac{p}{(p^2-1)(1-p^{-3})}.
\]
Hence the main term becomes
\begin{align*}
  &(1-p^{-1})^2\frac{\log p}{p-1}\left(\frac{2}{(1-p^{-1})^2}-\frac{4}{(1-p^{-1})(1-p^{-2})}\right. \\
  +& \left.\frac{2}{(1-p^{-2})(1-p^{-3})} +\frac{1}{1-p^{-2}}-\frac{1}{1-p^{-3}}+\frac{p}{(p^2-1)(1-p^{-3})}\right)\mf{B}X,
\end{align*}
which is precisely $\log p\mf{B}X/(p^2-1)$. Thus we obtain

\[
   \sum_{\substack{n<X\\ \gcd(n,S)=1}}\widehat{J}_{\cont,p}(f^{n})\\
  =2\frac{\log p}{p^2-1}\mf{B}X +O\left(\frac{\log p}{p-1}X^{\frac23+\varepsilon}\right).
\]

\underline{\emph{Step 4:}}\ \ Summing over $p<CX$.

The main term is thus
\[
\sum_{\substack{p<CX\\ p\notin S}}\frac{\log p}{p^2-1}\mf{B}X.
\]
Since
\[
\sum_{\substack{p<CX\\ p\notin S}}\frac{\log p}{p^2-1}\leq \sum_{ p\notin S}\frac{\log p}{p^2-1} -\sum_{p\geq CX}\frac{\log p}{p^2-1}
\qquad\text{and}\qquad
\sum_{p\geq CX}\frac{\log p}{p^2-1}\ll X^{-1+\varepsilon},
\]
we obtain the desired main term.

The error term is bounded by
\[
\sum_{\substack{p<CX\\ p\notin S}}\frac{\log p}{p-1}X^{\frac23+\varepsilon}\ll X^{\frac23+\varepsilon}\log X\ll X^{\frac23+\varepsilon'}
\]
by Merten's estimate. Thus the theorem is proved.
\end{proof}

By \autoref{lem:modifiedhyperbolicsum}, \autoref{prop:modifyunipotentestimate}, \autoref{thm:modifycontinuousestimate} and \autoref{thm:modifieddegree1trace2} we obtain
\begin{theorem}\label{thm:contributionhyperbolic2}
For any $\varepsilon>0$, we have
\[
\sum_{\substack{n<X\\ \gcd(n,S)=1}}\widehat{J}_{\hyp}^S(f^{n})\\
=\sum_{p\notin S}\frac{\log p}{p^2-1}\mf{B}X +O(X^{\frac23+\varepsilon}),
\]
where the implied constant only depends on $f_\infty$, $f_q$ and $\varepsilon$.
\end{theorem}

Combining the above theorem with \autoref{cor:relationhyperbolic}, we obtain
\begin{theorem}\label{thm:contributionhyperbolic1}
For any $\varepsilon>0$, we have
\[
\sum_{\substack{n<X\\ \gcd(n,S)=1}}\widetilde{J}_{\hyp}^S(f^{n})\\
=-\frac12\mf{B}(X\log X-X)+\sum_{p\notin S}\frac{\log p}{p^2-1}\mf{B}X +O(X^{\frac23+\varepsilon}),
\]
where the implied constant only depends on $f_\infty$, $f_q$ and $\varepsilon$.
\end{theorem}

\subsection{Estimate of the ramified modified local hyperbolic part}
Suppose that $v\in S$. We will give an asymptotic formula for
\[
\sum_{\substack{n<X\\ \gcd(n,S)=1}}\widetilde{J}_{\hyp,v}(f^n)
\]
using \autoref{thm:modifieddegree1trace}.
\begin{proposition}\label{prop:unipotentmodifycontribution}
We have
\[
\sum_{\substack{n<X\\ \gcd(n,S)=1}}\widetilde{J}_{\unip,v}(f^n)\ll X^{\frac12},
\]
where the implied constant only depends on $f_\infty$ and $f_{q_i}$.
\end{proposition}
\begin{proof}
We have
\[
\sum_{z\in \Z(\QQ)}\int_{\QQ_v}F_{z,v}(x_v)\log|x_v|\rmd x_v\cdot\prod_{w\neq v} \int_{\QQ_w}F_{z,w}(x_w)\rmd x_w.
\]
Suppose that $v=p\notin S$. Then
\[
F_{z,p}(x_p)=\int_{K_p}f_p^n\left(k_p^{-1}z\begin{pmatrix}
                                 1 & x_p \\
                                 0 & 1 
                               \end{pmatrix}k_p\right)\rmd k_p=0
\]
unless $\mathopen{|}\det z\mathclose{|}=p^{-n_p}$, and in this case 
\[
\int_{\QQ_p}F_{z,p}(x_p)\rmd x_p=1
\]
by the proof of \autoref{prop:contributionunipotentdeg1}. On the other hand, for $v\in S$, we clearly have
\[
\int_{\QQ_v}F_{z,v}(x_v)\rmd x_v\ll_{f_v} 1\qquad\text{and}\qquad\int_{\QQ_v}F_{z,v}(x_v)\log|x_v|\rmd x_v\ll_{f_v} 1.
\]
Now by the first paragraph of \autoref{prop:contributionunipotentdeg1} we obtain
\[
\sum_{\substack{n<X\\ \gcd(n,S)=1}}\widetilde{J}_{\unip,v}(f^n)\ll \sum_{\substack{n<X\\ \gcd(n,S)=1\\ n=\square}} 1\ll X^{\frac12}.\qedhere
\]
\end{proof}

\begin{proposition}\label{prop:continuousmodifycontribution}
For any $\varepsilon>0$ we have
\[
\sum_{\substack{n<X\\ \gcd(n,S)=1}}\widetilde{J}_{\cont,v}(f^n)=-\frac14\prod_{i=1}^{r}(1-q_i^{-1})^2\widetilde{\Tr} \left(\Ind_{\B(\QQ_v)}^{\G(\QQ_v)}(0,\triv)(f_v)\right)\prod_{\substack{w\in S\\w\neq v}}\Tr\left(\xi_0(f_w)\right) X+ O(X^{\frac23+\varepsilon}),
\]
where the implied constant only depends on $f_\infty$, $f_{q_i}$ and $\varepsilon$.
\end{proposition}
\begin{proof}
Recall that (see \autoref{sec:hyperbolicpoisson})
\[
\widetilde{J}_{\cont,v}(f)=-\frac{1}{4\uppi\rmi}\sum_{\mu}\int_{(0)}\widetilde{\Tr} \left(\Ind_{\B(\QQ_v)}^{\G(\QQ_v)}(s,\mu_v)(f_v)\right)\cdot\prod_{w\neq v}\Tr\left(\Ind_{\B(\QQ_w)}^{\G(\QQ_w)}(s,\mu_w)(f_w)\right)\rmd s.
\]
By \autoref{lem:unramifiedtracecontinuous} we know that $\widetilde{J}_{\cont,v}(f^n)$ equals
\[
-\frac{1}{4\uppi\rmi}\sum_{\mu}\int_{(0)}\widetilde{\Tr} \left(\Ind_{\B(\QQ_v)}^{\G(\QQ_v)}(s,\mu_v)(f_v)\right)\prod_{\substack{w\in S\\w\neq v}}\Tr\left(\Ind_{\B(\QQ_w)}^{\G(\QQ_w)}(s,\mu_w)(f_w)\right)\frac{1}{n^{s}}\sum_{d\mid n}d^{2s}\mu_1(d)\mu_2\legendresymbol{n}{d}\rmd s.
\]
Let
\[
\Phi(s,\mu)= \widetilde{\Tr}\left(\Ind_{\B(\QQ_v)}^{\G(\QQ_v)}(s,\mu_v)(f_v)\right)\prod_{\substack{w\in S\\w\neq v}}\Tr\left(\Ind_{\B(\QQ_w)}^{\G(\QQ_w)}(s,\mu_w)(f_w)\right).
\]
We must verify that the functions $\Phi(s,\mu)$ satisfies the conditions of  \autoref{prop:estimatecontinuouscharacter}. By the proof of \autoref{thm:modifieddegree1trace}, the choice of $\mu_1\mu_2$ contributing the sum is finite, and for $p\neq v$, $C(\mu_{1,p})$ and $C(\mu_{2,p})$ both have finitely many choices, and for $p\notin S$, $C(\mu_{1,p})=C(\mu_{2,p})=1$. Hence (i) and (ii) holds. Now it suffices to verify
\[
\Phi(s,\mu)\ll (1+|s|)^{-2}C(\mu_1)^{-1}C(\mu_2)^{-1}
\]
for $\mu=\mu_1\boxtimes\mu_2$ contributing the sum. Since the choice of $\mu_1\mu_2$ is finite, we have $C(\mu_1)\asymp C(\mu_2)$. The conclusion now follows from the two cases of the proof of \autoref{thm:modifieddegree1trace} for $v$ nonarchimedean and archimedean, respectively.

\underline{\emph{Case 1:}}\ \ $v$ is nonarchimedean.

For $\mu$ contributing to the sum, by \autoref{prop:estimategermnonarchimedean} we have 
\[
\widetilde{\Tr}\left(\Ind_{\B(\QQ_v)}^{\G(\QQ_v)}(s,\mu_v)(f_v)\right)\ll C(\mu_{2,p})^{-2}\asymp C(\mu_{1,p})^{-1}C(\mu_{2,p})^{-1}\asymp C(\mu_{1})^{-1}C(\mu_{2})^{-1}
\]
since $v_p(\mu_{i})$ and $\mu_1\mu_2$ only have finitely many choices.

For $w=\infty$,
\begin{equation}\label{eq:traceind}
\Tr\left(\Ind_{\B(\QQ_w)}^{\G(\QQ_w)}(s,\mu_w)(f_w)\right)
\end{equation}
has rapid decay vertically in $s$. Moreover, for $w\neq v$, \eqref{eq:traceind} is uniformly bounded (see the proof of Proposition 3.7 of \cite{cheng2025b}). Hence
\[
\Phi(s,\mu)\ll (1+|s|)^{-2} C(\mu_{1})^{-1}C(\mu_{2})^{-1}
\]

\underline{\emph{Case 2:}}\ \ $v$ is archimedean.
For $\mu$ contributing to the sum, by \autoref{prop:estimategermarchimedean} we have
\[
\widetilde{\Tr}\left(\Ind_{\B(\QQ_v)}^{\G(\QQ_v)}(s,\mu_v)(f_v)\right)\ll (1+|s|)^{-2}.
\]
Also, the sum over $\mu$ in this case is a finite set. Hence $C(\mu_1)^{-1}C(\mu_2)^{-1}\asymp 1$. Since \eqref{eq:traceind} is uniformly bounded we obtain
\[
\Phi(s,\mu)\ll (1+|s|)^{-2} \ll (1+|s|)^{-2}  C(\mu_{1})^{-1}C(\mu_{2})^{-1}.
\]

Now by using \autoref{prop:estimatecontinuouscharacter}, we obtain

\begin{align*}
\sum_{\substack{n<X\\ \gcd(n,S)=1}}\widetilde{J}_{\cont,v}(f^n)&=-\frac14\prod_{i=1}^{r}(1-q_i^{-1})^2\Phi(0,\triv)X+ O(X^{\frac23+\varepsilon})\\
&=-\frac14\prod_{i=1}^{r}(1-q_i^{-1})^2\widetilde{\Tr} \left(\Ind_{\B(\QQ_v)}^{\G(\QQ_v)}(0,\triv)(f_v)\right)\prod_{\substack{w\in S\\w\neq v}}\Tr\left(\xi_0(f_w)\right) X+ O(X^{\frac23+\varepsilon}).\qedhere
\end{align*}
\end{proof}
By \autoref{thm:modifieddegree1trace} and \autoref{prop:unipotentmodifycontribution}, \autoref{prop:continuousmodifycontribution} we obtain

\begin{corollary}
For any $\varepsilon>0$ we have
\[
\sum_{\substack{n<X\\ \gcd(n,S)=1}}\widetilde{J}_{\hyp,v}(f^n)=-\frac14\prod_{i=1}^{r}(1-q_i^{-1})^2\widetilde{\Tr} \left(\Ind_{\B(\QQ_v)}^{\G(\QQ_v)}(0,\triv)(f_v)\right)\prod_{\substack{w\in S\\w\neq v}}\Tr\left(\xi_0(f_w)\right) X+ O(X^{\frac23+\varepsilon}),
\]
where the implied constant only depends on $f_\infty$, $f_{q_i}$ and $\varepsilon$.
\end{corollary}
Summing over all $v\in S$ yields
\begin{theorem}\label{thm:hyperbolicramifiedcontribution}
For any $\varepsilon>0$ we have
\[
\sum_{\substack{n<X\\ \gcd(n,S)=1}}\widetilde{J}_{\hyp,S}(f^n)=\mf{E}X+O(X^{\frac23+\varepsilon}),
\]
where
\[
\mf{E}=-\frac14\prod_{i=1}^{r}(1-q_i^{-1})^2\sum_{v\in S}\widetilde{\Tr} \left(\Ind_{\B(\QQ_v)}^{\G(\QQ_v)}(0,\triv)(f_v)\right)\prod_{\substack{w\in S\\w\neq v}}\Tr\left(\xi_0(f_w)\right).
\]
The implied constant only depends on $f_\infty$, $f_{q_i}$ and $\varepsilon$.
\end{theorem}

\section{Final result}\label{sec:finalresult} 
We summarize all the contributions. 

\smallskip

\noindent \textbf{Geometric side:}
\begin{enumerate}[itemsep=0pt,parsep=0pt,topsep=0pt,leftmargin=0pt,labelsep=3pt,itemindent=9pt,label=\textbullet]
\item \textbf{Elliptic part:} By \autoref{cor:contributionelliptic}, \autoref{cor:contributionhyperbolicdeg1} and \autoref{thm:contributionhyperbolic2}, we have
\[
S_{\el}(X)=\sum_{\substack{n<X\\\gcd(n,S)=1}}I_\el(f^n) =\mf{A}X^{\frac32}+\mf{B}(X\log X-X)+(\mf{C}+\mf{D})X+\upgamma_S\mf{B}X -\sum_{p\notin S}\frac{\log p}{p^2-1}\mf{B}X +o(X).
\] 
\item \textbf{Hyperbolic part:} By \eqref{eq:hyperboliclocal}, \autoref{thm:contributionhyperbolic1} and \autoref{thm:hyperbolicramifiedcontribution}, we have
\[
S_{\hyp}(X)=\sum_{\substack{n<X\\\gcd(n,S)=1}}J_\hyp(f^n) =-\frac12\mf{B}(X\log X-X)+\sum_{p\notin S}\frac{\log p}{p^2-1}\mf{B}X+\mf{E}X +o(X).
\]
\item \textbf{Identity part:} By \autoref{prop:identitycontribution} we have
\[
S_{\id}(X)=\sum_{\substack{n<X\\\gcd(n,S)=1}}I_\id(f^n)=o(X).
\]
\item \textbf{Unipotent part:} By \autoref{prop:unipotentcontribution} we have
\[
S_{\unip}(X)=\sum_{\substack{n<X\\\gcd(n,S)=1}}I_\unip(f^n)=o(X).
\]
\end{enumerate}

\noindent \textbf{Spectral side:}
\begin{enumerate}[itemsep=0pt,parsep=0pt,topsep=0pt,leftmargin=0pt,labelsep=3pt,itemindent=9pt,label=\textbullet]
\item\textbf{Discrete parts:} By \autoref{thm:1dimestimate} we have
\[
S_{1-\mathrm{dim}}(X)=\sum_{\substack{n<X\\\gcd(n,S)=1}}\sum_{\mu}\Tr(\mu(f^n)) =\mf{A}X^{\frac32}+o(X)
\]
and by \autoref{thm:contributioneisenstein} we have
\[
S_{\mathrm{res}}(X)=\sum_{\substack{n<X\\\gcd(n,S)=1}}\frac{1}{4}\sum_{\mu}\Tr(M(0,\mu)(\xi_0\otimes\mu) (f^n))=\frac12\mf{B}(X\log X+\left(2\upgamma_S-1\right)X)+o(X).
\]
\item\textbf{Continuous part:} By \autoref{prop:contributioncontinuous1} and \autoref{prop:contributioncontinuous2} we have
\[
S_{\cont}(X)=\sum_{\substack{n<X\\\gcd(n,S)=1}}J_\cont(f^n)=-\left(\upgamma-2\log 2-\log\uppi\right) \mf{B} X +\mf{F}X+o(X).
\]
\end{enumerate}
In the above formulas,
\begin{enumerate}[itemsep=0pt,parsep=0pt,topsep=0pt,leftmargin=0pt,labelsep=3pt,itemindent=9pt,label=\textbullet]
\item $\mf{A}$ is defined in \cite[Theorem 7.10]{cheng2025c}; 
\item $\mf{B},\mf{C},\mf{D}$ are defined in \cite[Theorem 5.1]{cheng2025c}; 
\item $\mf{E}$ is defined in \autoref{thm:hyperbolicramifiedcontribution};
\item $\mf{F}$ is defined in \autoref{prop:contributioncontinuous2};
\item $\upgamma$ is the Euler-Mascheroni constant and $\upgamma_S$ is defined in \eqref{eq:defgammas}.
\end{enumerate}

Assuming that \eqref{eq:standardrepresentation} holds, we can compare the coefficients of $X^{3/2}$, $X\log X$ and $X$ on both sides of the trace formula.
The comparison of the $X^{3/2}$ and $X\log X$ terms are trivial. However, for the $X$ term it is not immediate to see that two sides of the trace formula are equal. Hence we have the following theorem:

\begin{theorem}[Limit form of the trace formula, global version]\label{thm:limittraceformula}
Assume that \eqref{eq:standardrepresentation} holds. Then for any test function $f=\bigotimes_{v\in S}f_v\in C_c^\infty(Z_+\bs\G(\QQ_S))$, we have
\[
\mf{C}+\mf{D}+\mf{E}=-\left(\upgamma-2\log 2-\log\uppi\right) \mf{B} +\mf{F}.
\]
\end{theorem}
In particular, for any $f=\bigotimes_{v\in \mf{S}}f_v\in C_c^\infty(\G(\AA)^1)$, there exists a finite set $S$ of places of $\QQ$ such that $\infty,2\in S$ and $f_v$ is the characteristic function on the standard maximal compact group $\G(\ZZ_v)$ for $v\notin S$. Hence we also obtain the limit form of the trace formula for $f$. Hence it is called the global version of the limit form of the trace formula.

In \autoref{sec:limittracelocal} we will derive local versions of the form of the trace formula by the global version.

Finally, we state a conjecture that the error terms can be improved.
\begin{conjecture}
In all terms above, the error term $o(X)$ can be replaced by $O(X^{1/2+\varepsilon})$ for any $\varepsilon>0$, where the implied constant only depends on $f_\infty$, $f_q$ and $\varepsilon$.
\end{conjecture}

\appendix
\section{Hyperbolic Poisson summation}\label{sec:hyperbolicpoisson}
\subsection{The hyperbolic Poisson summation formula}
In this section we will give a Poisson summation formula for nice functions $f$ on $Z_+\bs \A(\AA)$. Recall that the id\`eles can be decomposed as 
\[
\begin{split}
\AA^\times&\cong (\AA^\times)^1\times \RR_{>0}\\
x=(x_v)&\mapsto \left(\left(\frac{x_\infty}{|x|},(x_p)_{p}\right),|x|\right).
\end{split}
\]
Hence we have a decomposition
\[
Z_+\bs\A(\AA)\cong \A(\AA)^1\times Z_+\bs(\RR_{>0}\times \RR_{>0})
\]
such that the splitting only comes from the archimedean place. 

\begin{lemma}\label{lem:quotientmeasure}
We have a measure preserving isomorphism
\[
\begin{split}
   Z_+\bs (\RR_{>0}\times \RR_{>0}) &\to \RR_{>0} \\
   (a,b)  &\mapsto \frac{a}{b}.
\end{split}
\]
\end{lemma}
\begin{proof}
Clearly the map is an isomorphism. Now we prove that it is measure preserving.

Recall that the action of $Z_+\cong \RR_{>0}$ on $\RR_{>0}\times \RR_{>0}$ is given by $c\cdot (a,b)=(ca,cb)$. Since
\[
\int_{1}^{\rme}\frac{\rmd c}{c}=1,
\]
it suffices to show that
\[
E=\left\{(ca,cb)\,\middle|\,c\in [1,\rme],\frac ab\in [1,\rme], ab=1\right\}\subseteq \RR_{>0}\times \RR_{>0}
\]
has volume $1$ with respect to the measure $\frac{\rmd a\rmd b}{ab}$. By making the change of variable $\Phi\colon a\mapsto \rme^a, b\mapsto \rme^b$, it suffices to compute the standard Lebesgue measure of $\Phi(E)\subseteq \RR^2$. Clearly $\Phi(E)$ is the rectangle with vertices $(0,0)$, $(1,1)$, $(1/2,-1/2)$ and $(3/2,1/2)$, which has volume $1$ with respect to the Lebesgue measure. Hence the conclusion holds.
\end{proof}

\begin{theorem}[Hyperbolic Poisson summation formula]\label{thm:hyperbolicpoisson}
Suppose that $f$ is a function on $Z_+\bs\A(\AA)$ such that
\begin{enumerate}[itemsep=0pt,parsep=0pt,topsep=0pt,leftmargin=0pt,labelsep=2.5pt,itemindent=15pt,label=\upshape{(\roman*)}]
  \item For any $t\in Z_+\bs\A(\AA)$, the sum
  \[
  \sum_{\gamma\in \A(\QQ)}f(\gamma t)
  \]
  converges absolutely and it defines a continuous integrable function on $Z_+\A(\QQ)\bs\A(\AA)$.
  \item The series
  \[
  \sum_{\mu}\int_{(0)}\left|\int_{Z_+\bs\A(\AA)}f(t)\mu(t)\left|\frac{a}{b}\right|^{s}\rmd t\right|\rmd s
  \]
  converges, where $\mu$ runs over all characters on $\A(\QQ)\bs\A(\AA)^1$.
\end{enumerate} 
Then
\begin{equation}\label{eq:hyperbolicpoisson}
\sum_{\gamma\in \A(\QQ)}f(\gamma)=\frac{1}{2\uppi\rmi}\sum_{\mu}\int_{(0)}\int_{Z_+\bs\A(\AA)}f(t)\mu(t) \left|\frac{a}{b}\right|^{s}\rmd t\rmd s.
\end{equation}
\end{theorem}
\begin{proof}
Let
\[
F(g)=\sum_{\gamma\in \A(\QQ)}f(\gamma g).
\]
Then by (i), $F$ is a continuous integrable function on $Z_+\A(\QQ)\bs \A(\AA)\cong \A(\QQ)\bs \A(\AA)^1\times Z_+\bs (\RR_{>0}\times \RR_{>0})$.

By \autoref{lem:quotientmeasure}, the map 
\[
\begin{split}
   Z_+\A(\QQ)\bs \A(\AA) & \to \A(\QQ)\bs \A(\AA)^1\times \RR_{>0} \\
    \begin{pmatrix}
      a & 0 \\
      0 & b 
    \end{pmatrix} & \mapsto \left(\left(\begin{pmatrix}
      a_\infty/|a| & 0 \\
      0 & b_\infty/|b| 
    \end{pmatrix},\begin{pmatrix}
      a_p & 0 \\
      0 & b_p 
    \end{pmatrix}_p\right),\left|\frac{a}{b}\right|\right)
\end{split}
\]
is a measure preserving isomorphism. By the Mellin inversion formula we have a measure preserving isomorphism
\[
\begin{split}
   \left(\rmi \RR,\frac{\rmd (\rmi c)}{\dpii}\right) & \xrightarrow{\cong} \widehat{\RR_{>0}}\\
     \rmi c &\mapsto (x\mapsto x^{\rmi c}).
\end{split}
\]
Since $\A(\QQ)\bs \A(\AA)^1$ has volume $1$, the dual measure on $(\A(\QQ)\bs \A(\AA)^1)\sphat\ $ is the counting measure $\#$.
Hence by (ii) and the abstract Fourier inversion formula \cite[Chapitre 2, \SSec 1, N°4]{bourbaki2019spectral}, we conclude that
\[
F(g)=\int_{(Z_+\A(\QQ)\bs \A(\AA))\sphat}\left(\int_{Z_+\A(\QQ)\bs \A(\AA)}F(t)\overline{\chi(t)}\rmd t\right)\chi(g)\rmd \chi.
\]
Hence
\begin{align*}
\sum_{\gamma\in \A(\QQ)}f(\gamma)&=F(I)=\int_{(Z_+\A(\QQ)\bs \A(\AA))\sphat}\left(\int_{Z_+\A(\QQ)\bs \A(\AA)}F(t)\chi(t)\rmd t\right)\rmd \chi\\
&=\frac{1}{\dpii}\sum_{\mu}\int_{(0)}\int_{Z_+\A(\QQ)\bs \A(\AA)}F(t)\mu(t)\left|\frac{a}{b}\right|^{s}\rmd t\rmd s\\\
&=\frac{1}{\dpii}\sum_{\mu}\int_{(0)}\int_{Z_+\A(\QQ)\bs \A(\AA)}\sum_{\gamma\in \A(\QQ)}f(\gamma t)\mu(t)\left|\frac{a}{b}\right|^{s}\rmd t\rmd s\\
&=\frac{1}{2\uppi\rmi}\sum_{\mu}\int_{(0)}\int_{Z_+\bs\A(\AA)}f(t)\mu(t)\left|\frac{a}{b}\right|^{s}\rmd t\rmd s.\qedhere
\end{align*}
\end{proof}

\subsection{The degree $1$ term of the Arthur-Selberg trace formula for $\GL_2$} 
In this subsection we will derive the formula \eqref{eq:traceformuladeg1} for $f=\bigotimes_{v\in \mf{S}}'f_v\in C_c^\infty(\G(\AA)^1)$.

We first give an explicit formula of $\orb(f_v;t)$ for $t\in \A(\QQ_v)$.
By Iwasawa decomposition we have
\[
\orb(f_v;t)=
\int_{u\in\N(\QQ_v)}\int_{k\in K_v}f(k^{-1}u^{-1}\gamma uk)\rmd k\rmd u,
\]
where $K_v$ denotes the standard maximal compact subgroup of $\G(\QQ_v)$.
Suppose that $u=(\begin{smallmatrix} 1 & x \\  0 & 1 \end{smallmatrix})$. Then
\[
u^{-1}\gamma u=\begin{pmatrix}   a & x(b-a) \\  0 & b \end{pmatrix}.
\]
Hence
\[
\orb(f_v;t)=
\int_{\QQ_v}F_v\begin{pmatrix}   a & x(b-a) \\  0 & b \end{pmatrix}\rmd x,
\]
where
\[
F_v(g)=\int_{K_v}f_v(k^{-1}gk)\rmd k.
\]
By making the change of variable $x\mapsto (b-a)^{-1}x$ in $\orb(f_v;t)$ we obtain
\[
\frac{|a-b|_v}{|ab|_v^{1/2}}\orb(f_v;t)=
\frac{1}{|ab|_v^{1/2}}\int_{\QQ_v}F\begin{pmatrix}   a & x \\  0 & b \end{pmatrix}\rmd x.
\]
As $b\to a$, the function then becomes
\[
\frac{1}{|a|_v}\int_{\QQ_v}F_v\begin{pmatrix}   a & x \\  0 & b \end{pmatrix}\rmd x=\int_{\QQ_v}\int_{K_v}f_v\left(k^{-1}z\begin{pmatrix}
                                      1 & x \\
                                      0 & 1 
                                    \end{pmatrix}k\right)\rmd k\rmd x 
\]
where $z=(\begin{smallmatrix} a & 0 \\ 0 & a \end{smallmatrix})$. In the last step we made change of variable $x\mapsto ax$.

Hence globally for $t\in \A(\AA)$ we have
\begin{equation}\label{eq:limitunipotent}
\lim_{b\to a}\frac{|a-b|}{|ab|^{1/2}}\orb(f;t)=\int_{\AA}\int_{K}f\left(k^{-1}z\begin{pmatrix}
                                      1 & x \\
                                      0 & 1 
                                    \end{pmatrix}k\right)\rmd k\rmd x.
\end{equation}
Note that the right hand side is precisely the terms in the unipotent part.

We want to apply the hyperbolic Poisson summation formula to the function
\[
\frac{|a-b|}{|ab|^{1/2}}\orb(f;t).
\]
This function is continuous and compactly supported. Hence the condition (i) in \autoref{thm:hyperbolicpoisson} holds. Now we verify (ii). It suffices to show the absolute convergence of
\begin{equation}\label{eq:continuousdeg1abs}
\sum_{\mu}\int_{(0)}\left|\int_{Z_+\bs \A(\AA)}\frac{|a-b|}{|ab|^{1/2}}\left|\frac ab\right|^{s}\mu_1(a)\mu_2(b)\int_{\A(\AA)\bs \G(\AA)}f(g^{-1}tg)\rmd g \rmd t\right|\rmd |s|.
\end{equation}
By Proposition 3.8 of \cite{cheng2025b}, the formula in the absolute value is precisely
\begin{equation}\label{eq:continuousdeg1absinner}
\Tr\left(\Ind_{\B(\AA)}^{\G(\AA)}(s,\mu)(f)\right)=\prod_{v\in \mf{S}}\Tr\left(\Ind_{\B(\QQ_v)}^{\G(\QQ_v)}(s,\mu_v)(f_v)\right).
\end{equation}
By the proof of Corollary 3.11 of \cite{cheng2025b}, the sum over $\mu$ in \eqref{eq:continuousdeg1abs} is finite. Also, \eqref{eq:continuousdeg1absinner} is of rapid decay vertically in $s$ by the proof of Proposition 3.7 in \cite{cheng2025b}. Hence condition (ii) in \autoref{thm:hyperbolicpoisson} holds. 

Now we can use the hyperbolic Poisson summation formula  \eqref{eq:hyperbolicpoisson}. The right hand side of  \eqref{eq:hyperbolicpoisson} is
\begin{align*}
&\frac{1}{\dpii}\sum_{\mu}\int_{(0)}\int_{Z_+\bs \A(\AA)}\frac{|a-b|}{|ab|^{1/2}}\left|\frac ab\right|^{s}\mu_1(a)\mu_2(b)\int_{\A(\AA)\bs \G(\AA)}f(g^{-1}tg)\rmd g \rmd t\\
=&\frac{1}{\dpii}\sum_{\mu}\int_{(0)}\Tr\left(\Ind_{\B(\AA)}^{\G(\AA)}(s,\mu)(f)\right)\rmd s,
\end{align*}
which is precisely the degree $1$ term of the continuous part. The left hand side is
\[
\sum_{\gamma\in \A(\QQ)_\reg}\frac{|\gamma_1-\gamma_2|}{|\gamma_1\gamma_2|^{1/2}}\orb(f;\gamma)+ \sum_{z\in \Z(\QQ)}\int_{\AA}\int_{K}f\left(k^{-1}z\begin{pmatrix}
                                      1 & x \\
                                      0 & 1 
                                    \end{pmatrix}k\right)\rmd k\rmd x
\]
by using \eqref{eq:limitunipotent}. The second term is precisely $I_{\unip}^{\deg=1}(f)$. Since $|\alpha|_{\AA}=1$ for $\alpha\in \QQ^\times$, we know that the first term is precisely $I_{\hyp}^{\deg=1}(f)$. Hence we obtain
\[
I_{\hyp}^{\deg=1}(f)+I_{\unip}^{\deg=1}(f)=I_{\cont}^{\deg=1}(f).
\]

\subsection{The modified weighted hyperbolic orbital integral} 
Let $v\in \mf{S}$. For $t\in \A(\QQ_v)_\reg$ if $v$ is nonarchimedean, and $t\in Z_+\bs \A(\RR)_\reg$ if $v=\infty$, we define the \emph{modified weighted hyperbolic orbital integral (of the first kind)} to be
\begin{align*}
  \worb\sptilde(f;t) & =\int_{\A(\QQ_v)\bs \G(\QQ_v)}f(g^{-1}t g)\left(\alpha(H_\B(wg)+H_\B(g))-2\log\frac{|a-b|_v}{|ab|_v^{1/2}}\right)\rmd g \\
   & =\worb(f;t)-2\log\frac{|a-b|_v}{|ab|_v^{1/2}}\orb(f;t),
\end{align*}
where $t=(\begin{smallmatrix}  a & 0 \\   0 & b \end{smallmatrix})$.

\begin{proposition}\label{prop:modifiedorbitalcentral}
Let $v\in \mf{S}$ and $a\in \QQ_v^\times$. Then
\begin{equation}\label{eq:modifiedorbitalcentral}
\lim_{b\to a}\frac{|a-b|_v}{|ab|_v^{1/2}}\worb\sptilde(f;t)=\frac{2}{|a|_v}\int_{\QQ_v}F\begin{pmatrix}   a & x \\  0 & a \end{pmatrix}\log\left|\frac{a}{x}\right|_v\rmd x
\end{equation}
and the convergence is locally uniform in $a$, where
\[
F(g)=\int_{K}f(k^{-1}gk)\rmd k
\]
and $K$ denotes the standard maximal compact subgroup of $\G(\QQ_v)$.
\end{proposition}
\begin{proof}
We first assume that $v=p$ is a prime. By Iwasawa decomposition we have
\[
\worb\sptilde(f;t)=
\int_{u\in\N(\QQ_p)}\int_{k\in K}f(k^{-1}u^{-1}t uk)\rmd k\left(\alpha(H_\B(wu))-2\log\frac{|a-b|_v}{|ab|_v^{1/2}}\right)\rmd u.
\]
Suppose that $u=(\begin{smallmatrix} 1 & x \\  0 & 1 \end{smallmatrix})$. Then
\[
u^{-1}\gamma u=\begin{pmatrix}   a & x(b-a) \\  0 & b \end{pmatrix}.
\]
Hence
\[
\worb\sptilde(f;t)=
\int_{\QQ_p}F\begin{pmatrix}   a & x(b-a) \\  0 & b \end{pmatrix}\left(H_p(x)-2\log\frac{|a-b|_v}{|ab|_v^{1/2}}\right)\rmd x,
\]
where 
\[
H_p(x)=
\alpha(H_\B(wu))=\begin{cases}
                   2v_p(x)\log p, & v_p(x)<0, \\
                   0, & v_p(x)\geq 0
                 \end{cases}
\]
by \eqref{eq:nonarchimedeanweight}. By making the change of variable $x\mapsto (b-a)^{-1}x$ in $\worb\sptilde(f;t)$ we obtain
\begin{align*}
&\frac{|a-b|_v}{|ab|_v^{1/2}}\worb\sptilde(f;t)=
\frac{1}{|ab|_v^{1/2}}\int_{\QQ_p}F\begin{pmatrix}   a & x \\  0 & b \end{pmatrix}\left(H_v((a-b)^{-1}x)-2\log\frac{|a-b|_v}{|ab|_v^{1/2}}\right)\rmd x\\
=&\frac{2}{|ab|_v^{1/2}}\left[\sum_{j=-\infty}^{v_p(a-b)}\int_{p^{j}\ZZ_p^\times}F\begin{pmatrix}   a & x \\  0 & b \end{pmatrix}v_p((a-b)^{-1}x)\log p\rmd x-\log\frac{|a-b|_v}{|ab|_v^{1/2}}\int_{\QQ_p}F\begin{pmatrix}   a & x \\  0 & b \end{pmatrix}\rmd x\right]\\
=&\frac{2}{|ab|_v^{1/2}}\left[\sum_{j=-\infty}^{v_p(a-b)}\int_{p^{j}\ZZ_p^\times}F\begin{pmatrix}   a & x \\  0 & b \end{pmatrix}\log\frac{|a-b|_v}{|x|_v}\rmd x-\log\frac{|a-b|_v}{|ab|_v^{1/2}}\int_{\QQ_p}F\begin{pmatrix}   a & x \\  0 & b \end{pmatrix}\rmd x\right].
\end{align*}
Since $F(g)$ is smooth and compactly supported, we know that when $b\to a$, the above converges to the limit \eqref{eq:modifiedorbitalcentral}, which is locally uniform in $a$. 

Now we assume that $v=\infty$. Using Iwasawa decomposition again, we find that
\begin{align*}
  \worb\sptilde(f;t) & =
\int_{u\in\N(\RR)}\int_{k\in K}f(k^{-1}u^{-1}\gamma uk)\rmd k\left(\alpha(H_\B(wu))-2\log\frac{|a-b|_\infty}{|ab|_\infty^{1/2}}\right)\rmd u\\
   & =
\int_{\RR}F\begin{pmatrix}   a & x(b-a) \\  0 & b \end{pmatrix}\left(H_\infty(x)-2\log\frac{|a-b|_\infty}{|ab|_\infty^{1/2}}\right)\rmd x,
\end{align*}
where
\[
H_\infty(x)=-\log(1+x^2)
\]
by using archimedean Iwasawa decomposition. By making the change of variable $x\mapsto (b-a)^{-1}x$, we obtain
\begin{align*}
\frac{|a-b|_\infty}{|ab|_\infty^{1/2}}\worb\sptilde(f;t)&=
\frac{1}{|ab|_\infty^{1/2}}\int_{\RR}F\begin{pmatrix}   a & x \\  0 & b \end{pmatrix}\left(H_\infty((a-b)^{-1}x)-2\log\frac{|a-b|_\infty}{|ab|_\infty^{1/2}}\right)\rmd x\\
&=
-\frac{1}{|ab|_\infty^{1/2}}\int_{\RR}F\begin{pmatrix}   a & x \\  0 & b \end{pmatrix}\left(\log\left(1+\frac{x^2}{(a-b)^2}\right)+2\log\frac{|a-b|_\infty}{|ab|_\infty^{1/2}}\right)\rmd x\\
&=-\frac{1}{|ab|_\infty^{1/2}}\int_{\RR}F\begin{pmatrix}   a & x \\  0 & b \end{pmatrix}\log\left(\frac{|a-b|_\infty^2}{|ab|_\infty}+\frac{x^2}{|ab|_\infty}\right)\rmd x,
\end{align*}
which yields the desired limit, locally uniform in $a$, since $F(g)$ is smooth and compactly supported. 
\end{proof}
Thus we obtain
\begin{corollary}\label{cor:weightedcontinuous}
The function
\[
\varphi_v(t)=\frac{|a-b|_v}{|ab|_v^{1/2}}\worb\sptilde(f;t)
\]
can be extended to be a continuous function on $t\in \A(\QQ_v)$. 
Moreover, for $z\in \Z(\QQ_v)$, we have
\[
\varphi(z)=-2\int_{\QQ_v}\int_{K}f\left(k^{-1}z\begin{pmatrix}   1 & x \\  0 & 1\end{pmatrix}k\right)\log |x|_v\rmd k\rmd x,
\]
where $K$ denotes the standard maximal compact subgroup of $\G(\QQ_v)$.
\end{corollary}
\begin{proof}
Suppose that $z=(\begin{smallmatrix}  a &  \\    & a \end{smallmatrix})$. By the above proposition it suffices to show that
\[
\frac{2}{|a|_v}\int_{\QQ_v}F\begin{pmatrix}   a & x \\  0 & a \end{pmatrix}\log\left|\frac{a}{x}\right|_v\rmd x=-2\int_{\QQ_v}F\left(z\begin{pmatrix}   1 & x \\  0 & 1\end{pmatrix}\right)\log |x|_v\rmd x,
\]
which holds by making the change of variable $x\mapsto ax$.
\end{proof}

The function $\varphi$ is not smooth even for the archimedean place. 
However, we can compute the germ expansion near the center. First we consider the nonarchimedean case.
\begin{proposition}\label{prop:nonarchimedeangerm}
Let $v=p$ be a prime and $f\in C_c^\infty(\G(\QQ_p))$. Let $z=aI$ be an element in the center of $\G(\QQ_p)$. Then there exist a neighborhood $N\subseteq \A(\QQ_p)$ of $z$ and constants $\lambda_1$ and $\lambda_2$ such that for any $t=(\begin{smallmatrix} a & 0 \\  0 & b \end{smallmatrix})\in N$,
\[
\worb\sptilde(f;t)=\lambda_1+\frac{|ab|_v^{1/2}}{|a-b|_v}\lambda_2.
\]
\end{proposition}
\begin{proof}
For $t$ near $z$ we have $|b|_p=|a|_p$ and 
\[
F\begin{pmatrix}
   a & x \\
   0 & b 
 \end{pmatrix}=F\begin{pmatrix}
   a & x \\
   0 & a 
 \end{pmatrix}.
\]
Hence
\begin{align*}
\varphi_v(t)&=\frac{2}{|a|_v}\left[\sum_{j=-\infty}^{v_p(a-b)}\int_{p^{j}\ZZ_p^\times}F\begin{pmatrix}   a & x \\  0 & a \end{pmatrix}\log\frac{|a-b|_v}{|x|_v}\rmd x-\log\frac{|a-b|_v}{|a|_v}\int_{\QQ_p}F\begin{pmatrix}   a & x \\  0 & a \end{pmatrix}\rmd x\right]\\
&=\frac{2}{|a|_v}\left[-\sum_{j=v_p(a-b)+1}^{+\infty}\int_{p^{j}\ZZ_p^\times}F\begin{pmatrix}   a & x \\  0 & a \end{pmatrix}\log\frac{|a-b|_v}{|x|_v}\rmd x+\log\frac{|a|_v}{|x|_v}\int_{\QQ_p}F\begin{pmatrix}   a & x \\  0 & a \end{pmatrix}\rmd x\right]\\
&=\varphi_v(z)-\frac{2}{|a|_v}\sum_{j=n+1}^{+\infty}\int_{p^{j}\ZZ_p^\times}F\begin{pmatrix}   a & x \\  0 & a \end{pmatrix}\log\frac{|a-b|_v}{|x|_v}\rmd x,
\end{align*}
where $n=v_p(a-b)$.

If $b$ is sufficiently close to $a$, we have 
\[
F\begin{pmatrix}
   a & x \\
   0 & a
 \end{pmatrix}=F\begin{pmatrix}
   a & 0 \\
   0 & a 
 \end{pmatrix}
\]
if $x\in p^j\ZZ_p^\times$ for $j\geq n+1$. Hence
\begin{align*}
&\sum_{j=n+1}^{+\infty}\int_{p^{j}\ZZ_p^\times}F\begin{pmatrix}   a & x \\  0 & a \end{pmatrix}\log\frac{|a-b|_v}{|x|_v}\rmd x\\
=& F\begin{pmatrix}   a & 0 \\  0 & a \end{pmatrix} \log|a-b|_vp^{-n-1}+ F\begin{pmatrix}   a & 0 \\  0 & a \end{pmatrix} \sum_{j=n+1}^{+\infty}p^{-j}(1-p^{-1})j\log p\\
=&-F\begin{pmatrix}   a & 0 \\  0 & a \end{pmatrix} (n+1)p^{-n-1}\log p+ F\begin{pmatrix}   a & 0 \\  0 & a \end{pmatrix} \log p\left[\frac{p^{-n-2}}{1-p^{-1}}+(n+1)p^{-n-1}\right]\\
=&F\begin{pmatrix}   a & 0 \\  0 & a \end{pmatrix}\frac{p^{-2}\log p}{1-p^{-1}}|a-b|_v.
\end{align*}
Hence 
\[
\worb\sptilde(f;t)=\frac{|ab|_v^{1/2}}{|a-b|_v}\varphi_v(t)=\frac{|ab|_v^{1/2}}{|a-b|_v}\varphi_v(z) -2F\begin{pmatrix}   a & 0 \\  0 & a \end{pmatrix}\frac{p^{-2}\log p}{1-p^{-1}}
\]
for $b$ near $a$.
\end{proof}

\begin{proposition}\label{prop:estimategermnonarchimedean}
There exists $M>0$ such that
\[
\widetilde{\Tr} \left(\Ind_{\B(\QQ_v)}^{\G(\QQ_v)}(s,\mu_v)(f_v)\right):=\int_{ \A(\QQ_v)}\frac{|x-y|_v}{|xy|_v^{1/2}}\left|\frac xy\right|_v^{s}\mu_{1,v}(x)\mu_{2,v}(y)\worb\sptilde(f_v;t)\rmd t=0
\]
if $\mu_{1,v}\mu_{2,v}$ is not trivial on $1+p^M\ZZ_p$. Moreover, we have
\[
\widetilde{\Tr} \left(\Ind_{\B(\QQ_v)}^{\G(\QQ_v)}(s,\mu_v)(f_v)\right)\ll C(\mu_{2,v})^{-2},
\]
where $C(\mu_{i,v})$ denotes the conductor of $\mu_{i,v}$. The implied constant only depends on $f_v$.
\end{proposition}
\begin{proof}
We first prove that there exists $M>0$ such that
\[
\varphi_v(t)=\frac{|x-y|_v}{|xy|_v^{1/2}}\worb\sptilde(f_v;t)
\]
is invariant under the map $t\mapsto ct$ for any $c\in 1+p^M\ZZ_p$. Recall that
\[
\varphi_v(t)=\frac{2}{|ab|_v^{1/2}}\left[\sum_{j=-\infty}^{v_p(a-b)}\int_{p^{j}\ZZ_p^\times}F\begin{pmatrix}   a & x \\  0 & b \end{pmatrix}\log\frac{|a-b|_v}{|x|_v}\rmd x-\log\frac{|a-b|_v}{|ab|_v^{1/2}}\int_{\QQ_p}F\begin{pmatrix}   a & x \\  0 & b \end{pmatrix}\rmd x\right].
\]
Since $F$ is smooth and compactly supported, there exists $M>0$ such that the above expression is invariant under the map $a\mapsto ca$ and $b\mapsto cb$ for any $c\in 1+p^M\ZZ_p$. Hence the claim holds.

By the claim above we have
\begin{align*}
&\widetilde{\Tr} \left(\Ind_{\B(\QQ_v)}^{\G(\QQ_v)}(s,\mu_v)(f_v)\right)=\int_{\A(\QQ_v)}\varphi_v(ct)\left|\frac {cx}{cy}\right|_v^{s}\mu_{1,v}(cx)\mu_{2,v}(cy)\rmd t\\
=&(\mu_{1,v}\mu_{2,v})(c)\int_{\A(\QQ_v)}\varphi_v(t)\left|\frac {x}{y}\right|_v^{s}\mu_{1,v}(x)\mu_{2,v}(y)\rmd t=(\mu_{1,v}\mu_{2,v})(c)\widetilde{\Tr} \left(\Ind_{\B(\QQ_v)}^{\G(\QQ_v)}(s,\mu_v)(f_v)\right)
\end{align*}
for any $c\in 1+p^M\ZZ_p$. Hence
\[
\widetilde{\Tr} \left(\Ind_{\B(\QQ_v)}^{\G(\QQ_v)}(s,\mu_v)(f_v)\right)=0
\]
if $\mu_{1,v}\mu_{2,v}$ is not trivial on $1+p^M\ZZ_p$. 

Next we prove the second assertion. By \autoref{prop:nonarchimedeangerm}, there exist smooth and compactly supported functions $g_1$ and $g_2$ on $\A(\QQ_v)$ such that
\[
\varphi_v(t)=g_1(t)+|a-b|g_2(t)
\]
for all $t=(\begin{smallmatrix}  a & 0 \\ 0 & b \end{smallmatrix})\in \A(\QQ_v)$. Hence $\widetilde{\Tr} \left(\Ind_{\B(\QQ_v)}^{\G(\QQ_v)}(s,\mu_v)(f_v)\right)$ can be split into the following two terms
\begin{equation}\label{eq:germfirstterm}
\int_{\A(\QQ_v)}g_1(t)\left|\frac {a}{b}\right|_v^{s}\mu_{1,v}(a)\mu_{2,v}(b)\rmd t
\end{equation}
and
\begin{equation}\label{eq:germsecondterm}
\int_{\A(\QQ_v)}|a-b|_vg_2(t)\left|\frac {a}{b}\right|_v^{s}\mu_{1,v}(a)\mu_{2,v}(b)\rmd t
\end{equation}

Since $g_1(t)$ is smooth and compactly supported, there exists $M>0$ such that $\eqref{eq:germfirstterm}=0$ if $\mu_{1,v}$ or $\mu_{2,v}$ is not trivial on $1+p^M\ZZ_p$. Hence we clearly have (uniform in $s$)
\[
\eqref{eq:germfirstterm}\ll C(\mu_2)^{-2}.
\]

Next we consider \eqref{eq:germsecondterm}. $g_2$ is a linear combination of characteristic functions of the form
\[
\triv_{z(1+p^M\ZZ_p)\times (1+p^M\ZZ_p)}.
\]
 Hence we may assume that $g_2$ is of such form.
We may assume that $z\in \Z(\QQ_v)$, for otherwise we may add this characteristic function to $g_1$ since $|a-b|$ is smooth away from the center.
Moreover, by a change of variable we may assume that $z$ is the identity matrix. Now it suffices to consider the integral
\[
\int_{(1+p^M\ZZ_p)\times (1+p^M\ZZ_p)}|a-b|_v\mu_{1,v}(a)\mu_{2,v}(b)\rmd^\times a\rmd^\times b.
\]
By making the change of variable $b\mapsto ab$, the above integral reduces to
\begin{equation}\label{eq:germsecondtermreduced}
\int_{1+p^M\ZZ_p}(\mu_{1,v}\mu_{2,v})(a)\rmd^\times a \int_{1+p^M\ZZ_p}|1-b|\mu_{2,v}(b)\rmd^\times b.
\end{equation}
Clearly we have
\[
\int_{1+p^M\ZZ_p}(\mu_{1,v}\mu_{2,v})(a)\rmd^\times a \ll 1.
\]
For the integral over $b$, we proceed as follows: We have
\[
\int_{1+p^M\ZZ_p}|1-b|\mu_{2,v}(b)\rmd^\times b=\int_{p^M\ZZ_p}|x|\mu_{2,v}(1+x)\rmd x=\sum_{j=M}^{+\infty}p^{-j}\int_{p^j\ZZ_p^\times}\mu_{2,v}(1+x)\rmd x.
\]

Recall that the conductor of $\mu$ is $p^N$, where $N\in \ZZ_{\geq 0}$ is the smallest nonnegative integer such that $\mu$ is trivial on $1+p^N\ZZ_p$. Hence we have
\begin{equation}\label{eq:conductorcase}
\int_{p^j\ZZ_p^\times}\mu(1+x)\rmd x=\begin{cases}
                                                   p^{-j}(1-p^{-1}), & j\geq \log_p C(\mu), \\
                                                   -p^{-j-1}, & j=\log_p C(\mu)-1, \\
                                                   0, & j<\log_p C(\mu)-1.
                                                 \end{cases}
\end{equation}
Set $N=\log_p C(\mu_{2,v})$. Then for $N$ sufficiently large we have
\[
\sum_{j=M}^{+\infty}p^{-j}\int_{p^j\ZZ_p^\times}\mu_{2,v}(1+x)\rmd x=-p^{-2N+1}+\sum_{j=N}^{+\infty}p^{-2j}(1-p^{-1})\ll p^{-2N}=C(\mu_{2,v})^{-2}.
\]
Hence we obtain $\eqref{eq:germsecondtermreduced}\ll C(\mu_{2,v})^{-2}$
and thus $\eqref{eq:germsecondterm}\ll C(\mu_{2,v})^{-2}$. 

Combining the bounds of \eqref{eq:germfirstterm} and \eqref{eq:germsecondterm} we obtain the result.
\end{proof}

Next, we consider the archimedean case.
\begin{lemma}
The limits
\[
\lim_{\varepsilon\to 0^+}\frac{\diff}{\diff \varepsilon}\log(\varepsilon^2+x^2)\qquad\text{and}\qquad \lim_{\varepsilon\to 0^-}\frac{\diff}{\diff \varepsilon}\log(\varepsilon^2+x^2)
\]
are well defined distributions on $\RR$.
\end{lemma}
\begin{proof}
We have
\[
\lim_{\varepsilon\to 0^+}\frac{\diff}{\diff \varepsilon}\log(\varepsilon^2+x^2)=\lim_{\varepsilon\to 0^+}\frac{2\varepsilon}{\varepsilon^2+x^2}=\lim_{\varepsilon\to 0^+}\frac{\rmi}{2}\left(\frac{1}{x+\rmi\varepsilon}-\frac{1}{x-\rmi\varepsilon}\right),
\]
which is the distribution
\[
\frac{\rmi}{2}\left(\frac{1}{x+\rmi 0}-\frac{1}{x-\rmi 0}\right).
\]
Similarly,
\[
\lim_{\varepsilon\to 0^-}\frac{\diff}{\diff \varepsilon}\log(\varepsilon^2+x^2)=\frac{\rmi}{2}\left(-\frac{1}{x+\rmi 0}+\frac{1}{x-\rmi 0}\right)
\]
is well defined.
\end{proof}

\begin{lemma}
The limit
\[
\lim_{\varepsilon\to 0}\frac{\diff^2}{\diff \varepsilon^2}\log(\varepsilon^2+x^2)
\]
is a well defined distribution on $\RR$.
\end{lemma}
\begin{proof}
We have
\[
\frac{\diff^2}{\diff \varepsilon^2}\log(\varepsilon^2+x^2)=\frac{\diff}{\diff \varepsilon}\frac{2\varepsilon}{\varepsilon^2+x^2}=\frac{\diff}{\diff x}\frac{2x}{\varepsilon^2+x^2}.
\]
Hence for any $\phi\in C_c^\infty(\RR)$ we have
\[
\int_{\RR}\phi(x)\frac{\diff^2}{\diff \varepsilon^2}\log(\varepsilon^2+x^2)\rmd x=\int_{\RR}\phi(x)\frac{\diff}{\diff x}\frac{2x}{\varepsilon^2+x^2}\rmd x=-\int_{\RR}\phi'(x)\frac{2x}{\varepsilon^2+x^2}\rmd x.
\]
Since
\[
\lim_{\varepsilon\to 0}\frac{2x}{\varepsilon^2+x^2}=2\pv \frac{1}{x},
\]
we obtain the result.
\end{proof}

\begin{proposition}\label{prop:estimategermarchimedean}
We have
\[
\widetilde{\Tr} \left(\Ind_{\B(\RR)}^{\G(\RR)}(s,\mu_\infty)(f_\infty)\right):=\int_{Z_+\bs\A(\RR)}\frac{|a-b|_\infty} {|ab|_\infty^{1/2}}\left|\frac ab\right|_\infty^{s}\mu_{1,\infty}(a)\mu_{2,\infty}(b)\worb\sptilde(f_\infty;t)\rmd t\ll |s|^{-2}
\]
when $\sigma=\Re s$ is fixed, where $t=(\begin{smallmatrix}
            a & 0 \\
            0 & b 
          \end{smallmatrix})$.
\end{proposition}
\begin{proof}
We split the integral into four regions, depending on the sign of $x$ and $y$. More precisely, for any two signs $\epsilon_1$ and $\epsilon_2$ we define
\[
A^{\epsilon_1,\epsilon_2}=\left\{\begin{pmatrix}
                     a & 0 \\
                     0 & b
                   \end{pmatrix}\in \A(\RR)\,\middle|\,\sgn a=\epsilon_1,\sgn b=\epsilon_2\right\}.
\]
Hence the integral can be split into four terms
\[
\int_{Z_+\bs A^{+,+}}\cdots,\ \int_{Z_+\bs A^{+,-}}\cdots,\ \int_{Z_+\bs A^{-,+}}\cdots,\ \int_{Z_+\bs A^{-,-}}\cdots.
\]

In the two regions $Z_+\bs A^{+,-}$ and $Z_+\bs A^{-,+}$, the modified weighted orbital integral is continuous. Hence the above integral in such two regions, as the Fourier transform of smooth and compactly supported functions, are of rapid decay with respect to $s$. Hence it suffices to consider the remaining two regions. For simplicity we only consider the integral
\begin{align*}
I(s)
&=
\int_{Z_+\bs A^{+,+}}\frac{|a-b|_\infty}{|ab|_\infty^{1/2}}\left|\frac ab\right|_\infty^{s}\mu_{1,\infty}(a)\mu_{2,\infty}(b)\worb\sptilde(f_\infty;t)\rmd t\\
&=-\int_{Z_+\bs (\RR_{>0}\times \RR_{>0})}\frac{2}{|ab|_\infty^{1/2}}\int_{\RR}F\begin{pmatrix}   a & x \\  0 & b \end{pmatrix}\log\left(\frac{|a-b|_\infty^2}{|ab|_\infty}+\frac{x^2}{|ab|_\infty}\right)\rmd x\left|\frac ab\right|_\infty^{s}\rmd^\times a\rmd^\times b.
\end{align*}
Since $F$ is $Z_+$-invariant, by making the change of variable $x\mapsto |ab|^{1/2}x$, we obtain
\[
\frac{2}{|ab|_\infty^{1/2}}\int_{\RR}F\begin{pmatrix}   a & x \\  0 & b \end{pmatrix}\log\left(\frac{|a-b|_\infty^2}{|ab|_\infty}+\frac{x^2}{|ab|_\infty}\right)\rmd x=2\int_{\RR}F\begin{pmatrix}   |\frac ab|^{1/2} & x \\  0 & |\frac ba|^{1/2} \end{pmatrix}\log\left(\frac{|a-b|_\infty^2}{|ab|_\infty}+x^2\right)\rmd x.
\]
Hence by \autoref{lem:quotientmeasure} we have
\[
I(s)=-2\int_{0}^{+\infty}\int_{\RR}F\begin{pmatrix}   a^{1/2} & x \\  0 & a^{-1/2} \end{pmatrix}\log\left((a-a^{-1})^2+x^2\right)\rmd x a^s\rmd^\times a.
\]
Let $a=\rme^\xi$. Then we have
\[
I(s)=-2\int_{\RR}\int_{\RR}\phi(x,\xi)\log\left((\rme^\xi-\rme^{-\xi})^2+x^2\right)\rmd x \rme^{s\xi}\rmd \xi,
\]
where $\phi(x,\xi)$ is a smooth and compactly supported function.

The inner integral is smooth when $\xi\in \lopen -\infty,0\ropen$ or $\xi\in \lopen 0,+\infty\ropen$. Hence we may use integration by parts twice on each interval (which is valid by the above lemmas as $\xi\to 0$). We have
\begin{align*}
&\int_{0}^{+\infty}\int_{\RR}\phi(x,\xi)\log\left((\rme^\xi-\rme^{-\xi})^2+x^2\right)\rmd x \rme^{s\xi}\rmd \xi \\
=&-\frac{1}{s}\int_{\RR}\phi(x,0)\log(x^2)\rmd x
+\frac{1}{s^2}\int_{\RR}\left.\frac{\diff}{\diff \xi}\right|_{\xi=0^+}\left[\phi(x,\xi)\log\left((\rme^\xi-\rme^{-\xi})^2+x^2\right)\right]\rmd x\\
+&\frac{1}{s^2}\int_{0}^{+\infty}\int_{\RR}\left.\frac{\diff^2}{\diff \xi^2}\right|_{\xi=0}\left[\phi(x,\xi)\log\left((\rme^\xi-\rme^{-\xi})^2+x^2\right)\right]\rmd x \rme^{s\xi}\rmd \xi
\end{align*}
and
\begin{align*}
&\int_{-\infty}^{0}\int_{\RR}\phi(x,\xi)\log\left((\rme^\xi-\rme^{-\xi})^2+x^2\right)\rmd x \rme^{s\xi}\rmd \xi \\
=&\frac{1}{s}\int_{\RR}\phi(x,0)\log(x^2)\rmd x
-\frac{1}{s^2}\int_{\RR}\left.\frac{\diff}{\diff \xi}\right|_{\xi=0^-}\left[\phi(x,\xi)\log\left((\rme^\xi-\rme^{-\xi})^2+x^2\right)\right]\rmd x\\
+&\frac{1}{s^2}\int_{-\infty}^{0}\int_{\RR}\left.\frac{\diff^2}{\diff \xi^2}\right|_{\xi=0}\left[\phi(x,\xi)\log\left((\rme^\xi-\rme^{-\xi})^2+x^2\right)\right]\rmd x \rme^{s\xi}\rmd \xi
\end{align*}
Since $|\rme^{s\xi}|\ll 1$ when $\sigma$ is fixed and $\xi$ lies in a compact region, we obtain 
\[
I(s)\ll |s|^{-2}
\]
by adding the above two results together. Hence the conclusion holds.
\end{proof}

\subsection{The modified local degree $1$ terms} \label{subsec:modifieddeg1}
For any $v\in \mf{S}$ and $f=\bigotimes_{w\in \mf{S}}'f_w$, we define the \emph{modified local hyperbolic part (of the first kind)} as
\begin{align*}
  \widetilde{J}_{\hyp,v}(f)& =-\frac{1}{2}\sum_{\gamma\in \A(\QQ)_{\reg}}\worb\sptilde(f_v;\gamma)\prod_{w\neq v}\orb(f_w;\gamma)\\
   &=J_{\hyp,v}(f)+\sum_{\gamma\in \A(\QQ)_{\reg}}\log\frac{|\gamma_1-\gamma_2|_v}{|\gamma_1\gamma_2|_v^{1/2}}\int_{\A(\AA)\bs \G(\AA)}f(g^{-1}\gamma g)\rmd g,
\end{align*}
where $\gamma=(\begin{smallmatrix}\gamma_1 &  \\   & \gamma_2 \end{smallmatrix})$.

Also, we define the \emph{modified local unipotent part (of the first kind)}
\[
\widetilde{J}_{\unip,v}(f)=\sum_{z\in \Z(\QQ)}\int_{\QQ_v}F_{z,v}(x_v)\log|x_v|\rmd x_v\cdot\prod_{w\neq v} \int_{\QQ_w}F_{z,w}(x_w)\rmd x_w
\]
and the \emph{modified local continuous part (of the first kind)}
\[
\widetilde{J}_{\cont,v}(f)=-\frac{1}{4\uppi\rmi}\sum_{\mu}\int_{(0)}\widetilde{\Tr} \left(\Ind_{\B(\QQ_v)}^{\G(\QQ_v)}(s,\mu_v)(f_v)\right)\cdot\prod_{w\neq v}\Tr\left(\Ind_{\B(\QQ_w)}^{\G(\QQ_w)}(s,\mu_w)(f_w)\right)\rmd s.
\]

\begin{theorem}\label{thm:modifieddegree1trace}
For any $v\in \mf{S}$ we have
\[
\widetilde{J}_{\hyp,v}(f)+\widetilde{J}_{\unip,v}(f)=\widetilde{J}_{\cont,v}(f).
\]
\end{theorem}
\begin{proof}
We will use the hyperbolic Poisson summation formula on the function 
\[
\varphi(t)=-\frac12\frac{|x-y|_v}{|xy|_v^{1/2}}\worb\sptilde(f_v;t_v)\cdot\prod_{w\neq v}\frac{|x-y|_w}{|xy|_w^{1/2}}\orb(f_w;t_w).
\]
By \autoref{cor:weightedcontinuous}, $\varphi(t)$ is continuous and compactly supported. Hence the condition (i) in \autoref{thm:hyperbolicpoisson} holds. Now we verify (ii). It suffices to show the absolute convergence of
\begin{equation}\label{eq:continuousdeg1weightedabs}
\sum_{\mu}\int_{(0)}\left|\int_{Z_+\bs \A(\AA)}\frac{|x-y|}{|xy|^{1/2}}\left|\frac xy\right|^{s}\mu_1(x)\mu_2(y)\worb\sptilde(f_v;t_v)\cdot\prod_{w\neq v}\orb(f_w;t_w)\rmd t\right|\rmd |s|.
\end{equation}
By Proposition 3.8 of \cite{cheng2025b} and the definition of $\widetilde{\Tr}$, the formula in the absolute value is
\begin{equation}\label{eq:continuousdeg1weightedabsinner}
\widetilde{\Tr} \left(\Ind_{\B(\QQ_v)}^{\G(\QQ_v)}(s,\mu_v)(f_v)\right)\cdot\prod_{w\neq v}\Tr\left(\Ind_{\B(\QQ_w)}^{\G(\QQ_w)}(s,\mu_w)(f_w)\right).
\end{equation}

\emph{\underline{Case 1:}}\ \ $v=p$ is nonarchimedean.

In this case, the function
\[
\Tr\left(\Ind_{\B(\RR)}^{\G(\RR)}(s,\mu_\infty)(f_\infty)\right)
\]
has rapid decay vertically. Also, for each prime $\ell\neq p$ there exists $M_\ell\geq 0$ such that 
\[
\Tr\left(\Ind_{\B(\QQ_\ell)}^{\G(\QQ_\ell)}(s,\mu_\ell)(f_\ell)\right)=0
\]
if $\mu_{1,\ell}$ or $\mu_{2,\ell}$ is not trivial on $1+\ell^{M_\ell}\ZZ_\ell$, and $M_\ell=0$ for almost all $\ell$ (cf. Lemma 3.10 of \cite{cheng2025b}).
Moreover, by \autoref{prop:estimategermnonarchimedean}, there exists $M_p>0$ such that $\mu_{1,v}\mu_{2,v}$ is trivial on $1+p^{M_p}\ZZ_p$. Hence the choice of the global character $\mu_1\mu_2$ is finite. Now we fix such choice so that $\mu_1$ is determined by $\mu_2$. Hence it suffices to bound
\begin{equation}\label{eq:continuousdeg1weightedabs2}
\sum_{\mu_2}\int_{(0)}\left|\int_{Z_+\bs \A(\AA)}\frac{|x-y|}{|xy|^{1/2}}\left|\frac xy\right|^{s}\mu_1(x)\mu_2(y)\worb\sptilde(f_v;t_v)\cdot\prod_{w\neq v}\orb(f_w;t_w)\rmd t\right|\rmd |s|.
\end{equation}

For any nonarchimedean $\ell\neq p$, the $p$-part of $C(\mu_2)$ has only finite many choices. Moreover, for almost all $\ell$, $C(\mu_2)$ is relatively prime to $p$. Hence the choice of the "away from $p$ part" of $C(\mu_2)$ is finite. Now 
\[
C(\mu_2)=rp^j,
\]
where $r$ is fixed in a finite set and $j\in \ZZ_{\geq 0}$ can vary. 

Clearly we have
\begin{equation}\label{eq:trivialboundtrace}
\Tr\left(\Ind_{\B(\QQ_p)}^{\G(\QQ_p)}(s,\mu_p)(f_p)\right)\ll 1.
\end{equation}
Hence by \autoref{prop:estimategermnonarchimedean}, 
\[
\eqref{eq:continuousdeg1weightedabs2}\ll \sum_{r}\sum_{C(\mu_2)=rp^j}\int_{(0)}(1+|s|)^{-N}\rmd |s| \sum_{r}\sum_{C(\mu_2)=rp^j} C(\mu_2)^{-2}\ll \sum_{r}\sum_{C(\mu_2)=rp^j} C(\mu_2)^{-2}.
\]
Note that
\[
\#\{\mu_2\ \text{Dirichlet character} \,|\,C(\mu_2)=n\}\ll n.
\]
Hence
\[
\eqref{eq:continuousdeg1weightedabs2}\ll \sum_{r}\sum_{j=0}^{+\infty} (rp^j)(rp^j)^{-2}\ll 1.
\]
Hence \eqref{eq:continuousdeg1weightedabs} is absolutely convergent. 

\emph{\underline{Case 2:}}\ \ $v$ is archimedean.

In this case, \eqref{eq:continuousdeg1weightedabsinner} is
\[
\widetilde{\Tr} \left(\Ind_{\B(\RR)}^{\G(\RR)}(s,\mu_\infty)(f_\infty)\right)\cdot\prod_{p} \Tr\left(\Ind_{\B(\QQ_p)}^{\G(\QQ_p)}(s,\mu_p)(f_p)\right).
\] 
As in the previous case, there exists $M_p\geq 0$ such that 
\[
\Tr\left(\Ind_{\B(\QQ_p)}^{\G(\QQ_p)}(s,\mu_p)(f_p)\right)=0
\]
if $\mu_{1,p}$ or $\mu_{2,p}$ is not trivial on $1+p^{M_p}\ZZ_p$, with $M_p=0$ for almost all $p$. Hence the sum over $\mu$ is finite. Now we fix such $\mu$ and consider
\[
\int_{(0)}\left|\int_{Z_+\bs \A(\AA)}\frac{|x-y|}{|xy|^{1/2}}\left|\frac xy\right|^{s/2}\mu_1(x)\mu_2(y)\worb\sptilde(f_v;t_v)\cdot\prod_{w\neq v}\orb(f_w;t_w)\rmd t\right|\rmd |s|.
\]
Using \autoref{prop:estimategermarchimedean} and \eqref{eq:trivialboundtrace}, the above is bounded by
\[
\int_{(0)}(1+|s|)^{-2}\rmd |s|\ll 1.
\]
Hence \eqref{eq:continuousdeg1weightedabs} is absolutely convergent. 

By the above two cases, condition (ii) in \autoref{thm:hyperbolicpoisson} holds. 

Now we can use the hyperbolic Poisson summation formula  \eqref{eq:hyperbolicpoisson}. The right hand side of  \eqref{eq:hyperbolicpoisson} is
\begin{align*}
&-\frac{1}{4\uppi\rmi}\sum_{\mu}\int_{(0)}\int_{Z_+\bs \A(\AA)}\frac{|x-y|}{|xy|^{1/2}}\left|\frac xy\right|^{s}\mu_1(x)\mu_2(y)\worb\sptilde(f_v;t_v)\cdot\prod_{w\neq v}\orb(f_w;t_w)\rmd t\rmd s\\
=&-\frac{1}{4\uppi\rmi}\sum_{\mu}\int_{(0)}\widetilde{\Tr} \left(\Ind_{\B(\QQ_v)}^{\G(\QQ_v)}(s,\mu_v)(f_v)\right)\cdot\prod_{w\neq v}\Tr\left(\Ind_{\B(\QQ_w)}^{\G(\QQ_w)}(s,\mu_w)(f_w)\right)\rmd s.
\end{align*}
which is precisely the continuous part. The left hand side is
\begin{align*}
-&\frac12\sum_{\gamma\in \A(\QQ)_\reg}\frac{|\gamma_1-\gamma_2|_v}{|\gamma_1\gamma_2|_v^{1/2}}\worb\sptilde(f_v;\gamma) \cdot\prod_{w\neq v}\frac{|\gamma_1-\gamma_2|_w}{|\gamma_1\gamma_2|_w^{1/2}}\orb(f_w;\gamma)\\
+&\sum_{z\in \Z(\QQ)}\int_{\QQ_v}\int_{K_v}f\left(k_v^{-1}z\begin{pmatrix}
                                      1 & x_v \\
                                      0 & 1 
                                    \end{pmatrix}k_v\right)\rmd k_v \log |x_v|_v\rmd x_v\cdot \prod_{w\neq v}\int_{\QQ_w}\int_{K_w}f\left(k_w^{-1}z\begin{pmatrix}
                                      1 & x_w \\
                                      0 & 1 
                                    \end{pmatrix}k_w\right)\rmd k_w\rmd x_w
\end{align*}
by using \autoref{prop:modifiedorbitalcentral} and the local versions of \eqref{eq:limitunipotent}. The second term is precisely the unipotent part. Since $|\alpha|_{\AA}=1$ for $\alpha\in \QQ^\times$, we know that the first term is precisely the hyperbolic part. Hence we obtain our result.
\end{proof}
\subsection{The modified local hyperbolic part of the second kind}
For nonarchimedean $v$, we define the \emph{modified local weighted orbital integral of the second kind} to be 
\begin{align*}
  \worb\sphat(f;t) & =\int_{\A(\QQ_v)\bs \G(\QQ_v)}f(g^{-1}t g)\left(\alpha(H_\B(wg)+H_\B(g))-2\log|a-b|_v\right)\rmd g \\
   & =\worb(f;t)-2\log|a-b|_v\orb(f;t)
\end{align*}
for $t=(\begin{smallmatrix}  a & 0 \\ 0  & b  \end{smallmatrix})$. Moreover, we define the \emph{modified local hyperbolic part of the second kind} as
\begin{align*}
  \widehat{J}_{\hyp,v}(f)& =-\frac{1}{2}\sum_{\gamma\in \A(\QQ)_{\reg}}\worb\sphat(f_v;\gamma)\prod_{w\neq v}\orb(f_w;\gamma)\\
   &=J_{\hyp,v}(f)+\sum_{\gamma\in \A(\QQ)_{\reg}}\log|\gamma_1-\gamma_2|_v\int_{\A(\AA)\bs \G(\AA)}f(g^{-1}\gamma g)\rmd g\\
   &=\widetilde{J}_{\hyp,v}(f)+\sum_{\gamma\in \A(\QQ)_{\reg}}\log|\gamma_1\gamma_2|_v^{1/2}\int_{\A(\AA)\bs \G(\AA)}f(g^{-1}\gamma g)\rmd g.
\end{align*}
Let 
\[
f^v(g)=\log\mathopen{|}\det g_v\mathclose{|}_v^{1/2}f(g).
\]
Then we have $f^v\in C_c^\infty(Z_+\bs \G(\AA))$ and for $\gamma\in \A(\QQ)_{\reg}$ we have
\[
\log|\gamma_1\gamma_2|_v^{1/2}\int_{\A(\AA)\bs \G(\AA)}f(g^{-1}\gamma g)\rmd g=\int_{\A(\AA)\bs \G(\AA)}f^v(g^{-1}\gamma g)\rmd g.
\]
Hence we obtain
\[
\widehat{J}_{\hyp,v}(f)=\widetilde{J}_{\hyp,v}(f)+I_{\hyp}^{\deg=1}(f^v).
\]

Now let us define
\[
\widehat{J}_{\unip,v}(f)=\widetilde{J}_{\unip,v}(f)+ I_{\unip}^{\deg=1}(f^v)\qquad\text{and}\qquad \widehat{J}_{\cont,v}(f)=\widetilde{J}_{\cont,v}(f)+ I_{\cont}^{\deg=1}(f^v).
\]
Then by \eqref{eq:traceformuladeg1} and \autoref{thm:modifieddegree1trace} we obtain
\begin{theorem}\label{thm:modifieddegree1trace2}
For any nonarchimedean $v\in \mf{S}$ we have
\[
\widehat{J}_{\hyp,v}(f)+\widehat{J}_{\unip,v}(f)=\widehat{J}_{\cont,v}(f).
\]
\end{theorem}

Our final goal of this subsection is to give an explicit formula for $\widehat{J}_{\unip,v}(f)$ and $\widehat{J}_{\cont,v}(f)$.

Note that $f^v=\bigotimes_w' f_w^v$, where $f_v^v(g_v)=\log\mathopen{|}\det g_v\mathclose{|}_v^{1/2}f(g_v)$ and $f_w^v(g_w)=f(g_w)$ for all $w\neq v$. Hence we obtain
\[
I_{\unip}^{\deg=1}(f^v)=\sum_{z\in \Z(\QQ)}\int_{\QQ_v}F^v_{z,v}(x_v)\rmd x_v\cdot\prod_{w\neq v} \int_{\QQ_w}F_{z,w}(x_w)\rmd x_w,
\]
where
\[
F_{z,v}^{v}(x_v)=\int_{K}f_v^v\left(k^{-1}z\begin{pmatrix}
                                                1 & x_v \\
                                                0 & 1 
                                              \end{pmatrix} k\right)\rmd k
=\log\mathopen{|}\det z\mathclose{|}_v^{1/2}F_{z,v}(x_v).
\]
If we write $z=(\begin{smallmatrix}
                  a & 0 \\
                  0 & a 
                \end{smallmatrix})$, then we have
\begin{align*}
\int_{\QQ_v}F_{z,v}(x_v)(\log|x_v|_v+\log\mathopen{|}\det z\mathclose{|}_v^{1/2})\rmd x_v
&=\int_{\QQ_v}\int_{K}f_v\left(k^{-1}\begin{pmatrix}
                                                a & ax_v \\
                                                0 & a 
                                              \end{pmatrix} k\right)\rmd k\log|ax_v|_v\rmd x_v\\
&=\frac{1}{|a|_v}\int_{\QQ_v}\int_{K}f_v\left(k^{-1}\begin{pmatrix}
                                                a & x_v \\
                                                0 & a 
                                              \end{pmatrix} k\right)\rmd k\log|x_v|_v\rmd x_v.
\end{align*}
Hence
\begin{align*}
  &\widehat{J}_{\unip,v}(f)=\widetilde{J}_{\unip,v}(f)+ I_{\unip}^{\deg=1}(f^v) \\
    =&\sum_{z\in \Z(\QQ)}\int_{\QQ_v}F_{z,v}(x_v)(\log|x_v|_v+\log\mathopen{|}\det z\mathclose{|}_v^{1/2})\cdot\prod_{w\neq v} \int_{\QQ_w}F_{z,w}(x_w)\rmd x_w\\
   =&\sum_{z\in \Z(\QQ)}\frac{1}{|a|_v}\int_{\QQ_v}\int_{K_v}f_v\left(k_v^{-1}\begin{pmatrix}
                                                a & x_v \\
                                                0 & a 
                                              \end{pmatrix} k_v\right)\rmd k_v\log|x_v|_v\rmd x_v\\  \times&\prod_{w\neq v}\int_{\QQ_w}\int_{K_w}f_w\left(k_w^{-1}z\begin{pmatrix}
                                                1 & x_w \\
                                                0 & 1 
                                              \end{pmatrix} k_w\right)\rmd k_w\rmd x_w.
\end{align*}

Next we consider the continuous term. We define
\begin{equation}\label{eq:defwidehattrace}
\widehat{\Tr}\left(\Ind_{\B(\QQ_v)}^{\G(\QQ_v)}(s,\mu_v)(f_v)\right) =-2\Tr\left(\Ind_{\B(\QQ_v)}^{\G(\QQ_v)}(s,\mu_v)(f_v^v)\right) +\widetilde{\Tr}\left(\Ind_{\B(\QQ_v)}^{\G(\QQ_v)}(s,\mu_v)(f_v)\right)
\end{equation}
for any nonarchimedean place $v$. Then it is direct to see that
\begin{equation}\label{eq:widehattracecontinuous}
\widehat{J}_{\cont,v}(f)=-\frac{1}{4\uppi\rmi}\sum_{\mu}\int_{(0)}\widehat{\Tr} \left(\Ind_{\B(\QQ_v)}^{\G(\QQ_v)}(s,\mu_v)(f_v)\right)\cdot\prod_{w\neq v}\Tr\left(\Ind_{\B(\QQ_w)}^{\G(\QQ_w)}(s,\mu_w)(f_w)\right)\rmd s.
\end{equation}
Moreover, we have
\begin{equation}\label{eq:widehattrace}
\widehat{\Tr}\left(\Ind_{\B(\QQ_v)}^{\G(\QQ_v)}(s,\mu_v)(f_v)\right)=\int_{ \A(\QQ_v)}\frac{|a-b|_v}{|ab|_v^{1/2}}\left|\frac ab\right|_v^{s}\mu_{1,v}(a)\mu_{2,v}(b)\worb\sphat(f_v;t)\rmd t
\end{equation}
since $\worb\sphat(f_v;t)=\worb\sptilde(f_v;t)-2\log|ab|^{1/2}\orb(f;t)$.

$\widehat{\Tr}$ satisfies similar properties as $\widetilde{\Tr}$.
\begin{proposition}\label{prop:estimategermnonarchimedean2}
There exists $M>0$ such that
\[
\widehat{\Tr} \left(\Ind_{\B(\QQ_v)}^{\G(\QQ_v)}(s,\mu_v)(f_v)\right)=0
\]
if $\mu_{1,v}\mu_{2,v}$ is not trivial on $1+p^M\ZZ_p$. Moreover, we have
\[
\widehat{\Tr} \left(\Ind_{\B(\QQ_v)}^{\G(\QQ_v)}(s,\mu_v)(f_v)\right)\ll C(\mu_{2,v})^{-2},
\]
where $C(\mu_i)$ denotes the conductor of $\mu_i$. The implied constant only depends on $f_v$.
\end{proposition}
\begin{proof}
By \eqref{eq:defwidehattrace}, linearity and \autoref{prop:estimategermnonarchimedean}, it suffices to verify the conditions for 
\[
\widehat{\Tr} \left(\Ind_{\B(\QQ_v)}^{\G(\QQ_v)}(s,\mu_v)(f_v)\right)\quad\text{replaced by} \quad\Tr\left(\Ind_{\B(\QQ_v)}^{\G(\QQ_v)}(s,\mu_v)(f_v^v)\right).
\]

Since $f^v$ is a smooth function, there exists $M>0$ such that $\Tr\left(\Ind_{\B(\QQ_v)}^{\G(\QQ_v)}(s,\mu_v)(f_v^v)\right)=0$ if $\mu_{1,v}$ or  $\mu_{2,v}$ is not trivial on  $1+p^M\ZZ_p$. Hence the above two conditions automatically hold. 
\end{proof}

\begin{proposition}\label{prop:explicitmodifiedtrace2}
Suppose that $v=p\notin S$. Then we have
\[
\widehat{\Tr} \left(\Ind_{\B(\QQ_p)}^{\G(\QQ_p)}(s,\mu_p)(f_p^n)\right)\ll \frac{\log p}{p-1} C(\mu_{2,v})^{-2},
\]
where the implied constant is absolute.
\end{proposition}
\begin{proof}
By \eqref{eq:widehattrace} and \autoref{cor:unramifiedweightedlocalorbital} we have 
\begin{align*}
\widehat{\Tr}\left(\Ind_{\B(\QQ_v)}^{\G(\QQ_v)}(s,\mu_v)(f_v)\right)&=\int_{ \A(\QQ_v)}\worb\sphat(f_v;t)\frac{|a-b|_v}{|ab|_v^{1/2}}\left|\frac ab\right|_v^{s}\mu_{1,v}(a)\mu_{2,v}(b)\rmd t\\
&=2\int_{\ZZ_p\times \ZZ_p,\,v_p(ab)=n_p}\mu_{1,v}(a)\mu_{2,v}(b)\left|\frac ab\right|_v^{s}\sum_{j=1}^{v_p(a-b)}\frac{\log p}{p^j}\rmd^\times a\rmd ^\times b\\
&=2\sum_{k=0}^{n_p}\int_{p^k\ZZ_p^\times\times p^{n_p-k}\ZZ_p^\times}\mu_{1,v}(a)\mu_{2,v}(b)\left|\frac ab\right|_v^{s}\sum_{j=1}^{v_p(a-b)}\frac{\log p}{p^j}\rmd^\times a\rmd ^\times b\\
&=2\sum_{k=0}^{n_p}\int_{\ZZ_p^\times\times \ZZ_p^\times}\mu_{1,v}(p^ka)\mu_{2,v}(p^{n_p-k}b)\left|\frac ab\right|_v^{s}\sum_{j=1}^{v_p(p^ka-p^{n_p-k}b)}\frac{\log p}{p^j}\rmd^\times a\rmd ^\times b. 
\end{align*}
We have a trivial bound
\[
\widehat{\Tr}\left(\Ind_{\B(\QQ_v)}^{\G(\QQ_v)}(s,\mu_v)(f_v)\right)\ll \sum_{j=1}^{+\infty}\frac{\log p}{p^j}\ll \frac{\log p}{p-1}.
\]
for any $\mu$. Hence we may assume that $\mu_2$ is not trivial on $1+p\ZZ_p$, that is, $C(\mu_{2,p})\geq p^2$. Let $\alpha\in \ZZ_p^\times$ be such that $\mu_{2,p}(\alpha)\neq 1$. Now we consider the integral
\[
I_k=\int_{\ZZ_p^\times\times \ZZ_p^\times}\mu_{1,v}(p^ka)\mu_{2,v}(p^{n_p-k}b)\left|\frac {p^ka}{p^{n_p-k}b}\right|^{s}\sum_{j=1}^{v_p(p^ka-p^{n_p-k}b)}\frac{\log p}{p^j}\rmd^\times a\rmd ^\times b.
\]
If $k\neq n_p-k$, then $v_p(p^ka-p^{n_p-k}b)=\min\{k,n_p-k\}=v_p(p^ka-\alpha p^{n_p-k}b)$. Hence
\[
I_k=\int_{\ZZ_p^\times\times \ZZ_p^\times}\mu_{1,v}(p^ka)\mu_{2,v}(\alpha p^{n_p-k}b)\left|\frac {ap^k}{\alpha p^{n_p-k}b}\right|^{s}\sum_{j=1}^{v_p(p^ka-p^{n_p-k}b)}\frac{\log p}{p^j}\rmd^\times a\rmd ^\times b=\mu_{2,p}(\alpha)I_k.
\]
Hence $I_k=0$ if $k\neq n_p-k$. If $n_p$ is odd, then
\[
\widehat{\Tr}\left(\Ind_{\B(\QQ_v)}^{\G(\QQ_v)}(s,\mu_v)(f_v)\right)=0
\]
and our assertion is proved. Now we assume that $n_p$ is even. In this case
\[
\widehat{\Tr}\left(\Ind_{\B(\QQ_v)}^{\G(\QQ_v)}(s,\mu_v)(f_v)\right) =I_{n_p/2}=\int_{\ZZ_p^\times\times \ZZ_p^\times}\mu_{1,v}(p^{n_p/2}a)\mu_{2,v}(p^{n_p/2}b)\sum_{j=1}^{n_p/2+v_p(a-b)}\frac{\log p}{p^j}\rmd^\times a\rmd ^\times b.
\]
Making change of variable $b\mapsto ab$ yields
\[
I_{n_p/2}=\int_{\ZZ_p^\times}(\mu_{1,v}\mu_{2,v})(p^{n_p/2}a)\rmd^\times a\int_{\ZZ_p^\times}\mu_{2,v}(b)\sum_{j=1}^{n_p/2+v_p(1-b)}\frac{\log p}{p^j}\rmd ^\times b.
\]
If $\mu_1\mu_2$ is not trivial on $\ZZ_p^\times$, then the integral with respect to $a$ vanishes. Hence $I_{n_p/2}=0$ and the conclusion holds trivially. 

Now we assume that $\mu_{1,p}\mu_{2,p}$ is trivial. In this case we have
\begin{align*}
&(\mu_{1,p}\mu_{2,p})(p^{n_p/2})^{-1}I_{n_p/2}=\frac{p}{p-1}\int_{\ZZ_p^\times}\mu_2(b) \sum_{j=1}^{n_p/2+v_p(1-b)}\frac{\log p}{p^j}\rmd b\\
=&\frac{p}{p-1}\int_{\ZZ_p^\times-(1+p\ZZ_p)}\mu_{2,p}(b)\rmd b\sum_{j=1}^{n_p/2}\frac{\log p}{p^j}+\frac{p}{p-1}\sum_{i=1}^{+\infty}\int_{1+p^i\ZZ_p^\times}\mu_{2,p}(b)\rmd b \sum_{j=1}^{n_p/2+i}\frac{\log p}{p^j}\\
\ll&\frac{p}{p-1}\int_{\ZZ_p^\times-(1+p\ZZ_p)}\mu_{2,p}(b)\rmd b\sum_{j=1}^{+\infty}\frac{\log p}{p^j}+\frac{p}{p-1} \sum_{i=1}^{+\infty}\int_{1+p^i\ZZ_p^\times}\mu_{2,p}(b)\rmd b \frac{\log p}{p^j} \sum_{j=n_p/2+i+1}^{+\infty} \frac{\log p}{p^j}\\
\ll&\frac{p}{p-1}\sum_{i=\log_pC(\mu_{2,p})}^{+\infty}p^{-i}\sum_{j=n_p+i+1}^{+\infty}\frac{\log p}{p^j}\ll \frac{p\log p}{p-1}\sum_{i=\log_pC(\mu_{2,p})}^{+\infty}p^{-i}\frac{p^{-n_p-i}}{p-1}\ll C(\mu_{2,p})^{-2}\frac{\log p}{p-1}
\end{align*}
by \eqref{eq:conductorcase} and that $C(\mu_{2,p})\geq p^2$.
\end{proof}
\section{The local version of the limit form of the trace formula}\label{sec:limittracelocal}
In this section we assume that \eqref{eq:standardrepresentation} holds. We will derive local versions of the limit form of the trace formula from the global version \autoref{thm:limittraceformula}.

\subsection{Global computations}
 Recall that \autoref{thm:limittraceformula} states that
\[
\mf{C}+\mf{D}+\mf{E}=-(\upgamma-2\log 2-\log\uppi)\mf{B}+\mf{F}.
\]
Suppose that $\mf{B}\neq 0$. Then we have
\begin{equation}\label{eq:limittraceformula}
\frac{\mf{C}}{\mf{B}}+\frac{\mf{D}}{\mf{B}}+\frac{\mf{E}}{\mf{B}}=-(\upgamma-2\log 2-\log\uppi)+\frac{\mf{F}}{\mf{B}}.
\end{equation}

By the expression of $\mf{B},\mf{C}$ and $\mf{D}$ in \cite[Theorem 5.1]{cheng2025c}, it is direct to see that
\begin{align*}
  \frac{\mf{C}}{\mf{B}}= &-\frac12\left(2\upgamma-2\log (2\uppi)+\sum_{i=1}^{r}(1+q_i^{-1})\frac{\log q_i}{1-q_i^{-1}}\right) + \frac{\uppi}{2}\dfrac{\int_{X_1}|1-x|_\infty^{-1}\widehat{\Theta}_\infty(x)|x|_\infty^{-\frac12}\rmd x}{\int_{X_0}|1-x|_\infty^{-1}\widehat{\Theta}_\infty(x)|x|_\infty^{-\frac12}\rmd x} \\
  + & \sum_{i=1}^{r}\sum_{\epsilon_i\in \{0,-1\}}\frac{1-\epsilon_i q_i^{-1}}{1-\epsilon_i}\frac{\log q_i}{1-q_i^{-1}}\dfrac{\int_{Y_{i,\epsilon_i}}|1-y_i|_{q_i}^{-1}\widehat{\Theta}_{q_i}(y_i) |y_i|_{q_i}'^{-\frac12}\rmd y_i}{\int_{Y_{i,1}}|1-y_i|_{q_i}^{-1}\widehat{\Theta}_{q_i}(y_i) |y_i|_{q_i}'^{-\frac12}\rmd y_i}
\end{align*}
and
\[
  \frac{\mf{D}}{\mf{B}}= -\dfrac{\int_{X_0}\log\frac{|1-x|_\infty}{|x|_\infty}|1-x|_\infty^{-1}\widehat{\Theta}_\infty(x)|x|_\infty^{-\frac12}\rmd x}{\int_{X_0}|1-x|_\infty^{-1}\widehat{\Theta}_\infty(x)|x|_\infty^{-\frac12}\rmd x}
   - \sum_{i=1}^{r}\dfrac{\int_{Y_{i,1}}\log\frac{|1-y_i|_{q_i}}{|y_i|_{q_i}'}|1-y_i|_{q_i}^{-1}\widehat{\Theta}_{q_i}(y_i) |y_i|_{q_i}'^{-\frac12}\rmd y_i}{\int_{Y_{i,1}}|1-y_i|_{q_i}^{-1}\widehat{\Theta}_{q_i}(y_i) |y_i|_{q_i}'^{-\frac12}\rmd y_i}.
\]

Moreover, by \eqref{eq:ramifiedproducteisenstein} and the expression of $\mf{E}$ and $\mf{F}$ in \autoref{thm:hyperbolicramifiedcontribution} and \autoref{prop:contributioncontinuous2} respectively, we find that
\[
\frac{\mf{E}}{\mf{B}}=\frac12\sum_{v\in S}\frac{\widetilde{\Tr}\left(\Ind^{\G(\QQ_v)}_{\B(\QQ_v)}(0,\triv)(f_v)\right)}{\Tr(\xi_0(f_v))} \qquad\text{and}\qquad \frac{\mf{F}}{\mf{B}}=\frac12\sum_{v\in S}\frac{\Tr\left(R_v(0,\triv)^{-1}R_v'(0,\triv)\xi_0(f_v)\right)}{\Tr(\xi_0(f_v))}.
\]

Hence by \autoref{prop:archimedeantraceeisenstein} and \autoref{prop:ramifiedtraceeisenstein}, \autoref{thm:limittraceformula} can be restated as follows:
\begin{theorem}[Limit form of the trace formula, global version]\label{thm:limittraceformulamodified}
Let $S=\{\infty,q_1,\dots, q_r\}$ be a finite set of places of $\QQ$ containing $\infty$ and $2$. Let $f_\infty\in C_c^\infty(Z_+\bs\G(\RR))$ and $f_v\in C_c^\infty(\G(\QQ_v))$ for $v\in S$ nonarchimedean. Let the representation $\xi_0$ be defined in \autoref{subsec:traceformula}, $|\cdot|_{q_i}'$ be defined in \autoref{subsec:modifiednorm} and let $\widehat{\Theta}_\infty(x)$ and $\widehat{\Theta}_{q_i}(y_i)$ be defined in \autoref{subsec:singularities}. Suppose that $\Tr(\xi_0(f_v))\neq 0$ for all $v\in S$. Then we have
\begin{align*}
    &\frac12\sum_{v\in S}\frac{\Tr\left(R_v(0,\triv)^{-1}R_v'(0,\triv)\xi_0(f_v)\right)}{\Tr(\xi_0(f_v))}= -\log 2-\frac12\sum_{i=1}^{r}(1+q_i^{-1})\frac{\log q_i}{1-q_i^{-1}}+\frac{\uppi\int_{X_1}|1-x|_\infty^{-1}\widehat{\Theta}_\infty(x)|x|_\infty ^{-\frac12}\rmd x}{\Tr(\xi_0(f_\infty))} \\
   +& \sum_{i=1}^{r}\sum_{\epsilon_i\in \{0,-1\}}\frac{2(1-\epsilon_i q_i^{-1})\log q_i}{(1-\epsilon_i)(1-q_i^{-1})^2}\frac{\int_{Y_{i,\epsilon_i}}|1-y_i|_{q_i}^{-1}\widehat{\Theta}_{q_i}(y_i) |y_i|_{q_i}'^{-\frac12}\rmd y_i}{\Tr(\xi_0(f_{q_i}))}
   \!-\!2\frac{\int_{X_0}\log\frac{|1-x|_\infty}{|x|_\infty}|1-x|_\infty^{-1}\widehat{\Theta}_\infty(x)|x|_\infty^{-\frac12}\rmd x}{\Tr(\xi_0(f_\infty))}\\
   -&\sum_{i=1}^{r}\frac{2}{1-q_i^{-1}}\frac{\int_{Y_{i,1}}\log\frac{|1-y_i|_{q_i}}{|y_i|_{q_i}'}|1-y_i|_{q_i}^{-1}\widehat{\Theta}_{q_i}(y_i) |y_i|_{q_i}'^{-\frac12}\rmd y_i}{\Tr(\xi_0(f_{q_i}))}
   +\frac12\sum_{v\in S}\frac{\widetilde{\Tr}\left(\Ind_{\B(\QQ_v)}^{\G(\QQ_v)}(0,\triv)(f_v)\right)}{\Tr(\xi_0(f_{v}))}.
\end{align*}
\end{theorem}

\subsection{Local computation for $v=p$ odd prime}
Now we pass from global to local.
If $p$ is an odd prime, the local version of the limit form of the trace formula is easy to derive.
\begin{theorem}[Limit form of the trace formula, local version for odd primes] \label{thm:limittraceformulalocal1}
Let $p$ be an odd prime and $f_p\in C_c^\infty(\G(\QQ_p))$. Let $\xi_0$ be defined in \autoref{subsec:traceformula}, $|\cdot|_{p}'$ be defined in \autoref{subsec:modifiednorm}, and $\widehat{\Theta}_{p}(y)$ be defined in \autoref{subsec:singularities}. Then we have
\begin{align*}
    &\frac12\Tr\left(R_p(0,\triv)^{-1}R_p'(0,\triv)\xi_0(f_p)\right)=-\frac12(1+p^{-1})\frac{\log p}{1-p^{-1}}\Tr(\xi_0(f_p))\\
    +&\sum_{\epsilon\in \{0,-1\}}\frac{2(1-\epsilon p^{-1})\log p}{(1-\epsilon)(1-p^{-1})^2}\int_{Y_{\epsilon}}|1-y|_{p}^{-1}\widehat{\Theta}_{p}(y) |y|_{p}'^{-\frac12}\rmd y\\
   -&\frac{2}{1-p^{-1}}\int_{Y_{1}}\log\frac{|1-y|_{p}}{|y|_{p}'}|1-y|_{p}^{-1}\widehat{\Theta}_{p}(y) |y|_{p}'^{-\frac12}\rmd y
   +\frac12\widetilde{\Tr}\left(\Ind_{\B(\QQ_p)}^{\G(\QQ_p)}(0,\triv)(f_p)\right).
\end{align*}
\end{theorem}

\begin{proof}
Let $S_0=\{\infty,2\}$ and $S=S_0\cup\{p\}$. Choose any $f_\infty\in C_c^\infty(Z_+\bs\G(\RR))$ and $f_2\in C_c^\infty(\G(\QQ_2))$ such that $\Tr(\xi_0(f_v))\neq 0$ for all $v\in S_0$.
If $\Tr(\xi_0(f_p))\neq 0$, the theorem follows by using \autoref{thm:limittraceformulamodified} for $S_0$ and $S$ and then taking the difference of them. If $\Tr(\xi_0(f_p))=0$, the equality still holds by a limiting process.
\end{proof}

\subsection{Local computations for $v=\infty$ and $v=2$}
In this subsection we denote $p=2$.
By \autoref{thm:limittraceformulamodified} with $S=\{\infty,p\}$ we have

\begin{proposition}\label{prop:limittraceinfinity2}
We have
\begin{align*}
    &\frac12\sum_{v\in S}\frac{\Tr\left(R_v(0,\triv)^{-1}R_v'(0,\triv)\xi_0(f_v)\right)}{\Tr(\xi_0(f_v))}= -\log 2-\frac12(1+p^{-1})\frac{\log p}{1-p^{-1}}+\dfrac{\uppi\int_{X_1}|1-x|_\infty^{-1}\widehat{\Theta}_\infty(x)|x|_\infty ^{-\frac12}\rmd x}{\Tr(\xi_0(f_\infty))} \\
   +& \sum_{\epsilon\in \{0,-1\}}\frac{2(1-\epsilon p^{-1})\log p}{(1-\epsilon)(1-p^{-1})^2}\dfrac{\int_{Y_{\epsilon}}|1-y|_{p}^{-1}\widehat{\Theta}_{p}(y) |y|_{p}'^{-\frac12}\rmd y}{\Tr(\xi_0(f_{p}))}
   -2\dfrac{\int_{X_0}\log\frac{|1-x|_\infty}{|x|_\infty}|1-x|_\infty^{-1}\widehat{\Theta}_\infty(x)|x|_\infty^{-\frac12}\rmd x}{\Tr(\xi_0(f_\infty))}\\
   -&\frac{2}{1-p^{-1}}\dfrac{\int_{Y_{1}}\log\frac{|1-y|_{p}}{|y|_{p}'}|1-y|_{p}^{-1} \widehat{\Theta}_{p}(y) |y|_{p}'^{-\frac12}\rmd y}{\Tr(\xi_0(f_{p}))}
   +\frac12\sum_{v\in S}\frac{\widetilde{\Tr}\left(\Ind_{\B(\QQ_v)}^{\G(\QQ_v)}(0,\triv)(f_v)\right)}{\Tr(\xi_0(f_{v}))}.
\end{align*}
\end{proposition}

To derive the local version of the limit form of the trace formula for $v=\infty$, we may take $f_p\in C_c^\infty(\G(\QQ_p))$ to be a specific function whose orbital integral can be easily computed and deduce an equality for arbitrary $f_\infty\in C_c^\infty(Z_+\bs \G(\RR))$ (hence we also get a formula for $v=p$). 

We take $f_p$ to be the characteristic function of the standard maximal compact subgroup $\G(\ZZ_p)$ of $\G(\QQ_p)$. For $g\in \G(\QQ_p)$ we denote $T=\Tr g$ and $N=\det g$. In \cite[Theorem 2.7]{cheng2025}, we have computed that $\orb(f_p;g)=0$ unless $T\in\ZZ_p$ and $N\in \ZZ_p^\times$, and in this case we have
\[
\orb(f_p;g)=1+\sum_{j=1}^{k}p^j\left(1-\frac{\chi(p)}{p}\right), 
\]
where $k=\log_p|T^2-4N|'^{-\frac12}$ and $\chi(p)=\omega_p(T^2-4N)$.

From now on we assume that $T\in\ZZ_p$ and $N\in \ZZ_p^\times$.
Recall that
\[
\theta_p(T,N)=\theta_p(g) =\frac{1}{\mathopen{|}\det g\mathclose{|}_p^{1/2}}\left(1-\frac{\chi(p)}{p}\right)^{-1}p^{-k}\orb(f_{p};g)
\]
and
\[
\widehat{\Theta}_p(y)=\int_{\QQ_p^\times}\theta_p\left(z,\frac{z^2(1-y)}{4}\right)\frac{\rmd z}{|z|_p}.
\]
 
Note that $\theta_p(z,z^2(1-y)/4)\neq 0$ only if $z\in \ZZ_p$ and $z^2(1-y)/4\in \ZZ_p^\times$, only if $v_p(z)\geq 0$ and
\[
v_p(1-y)=-v_p(z^2/4)=-2v_p(z)+2.
\]
From this we know that $v_p(1-y)$ must be even and $v_p(1-y)\leq 2$.

\begin{lemma}\label{lem:orbitalintegral2}
If $v_p(1-y)>2$ or $v_p(1-y)$ is odd, we have $\widehat{\Theta}_p(y)=0$. Otherwise:
\begin{enumerate}[itemsep=0pt,parsep=0pt,topsep=0pt, leftmargin=0pt,labelsep=2.5pt,itemindent=15pt,label=\upshape{(\arabic*)}]
  \item If $y\in Y_1$, then $\widehat{\Theta}_p(y)=1$.
\item If $y\in Y_{-1}$, then
\[
\widehat{\Theta}_p(y)=1-\frac13|y|_p'^{\frac12}|1-y|_p^{-\frac12}.
\]
\item If $y\in Y_0$, then
\[
\widehat{\Theta}_p(y)=1-\frac14|y|_p'^{\frac12}|1-y|_p^{-\frac12}.
\]
\end{enumerate}
\end{lemma}
\begin{proof}
The proof is divided into three cases:

\underline{\emph{Case 1:}}\ \ $\chi(p)=1$, so that $y\in Y_1$.

In this case, $\orb(f_p;g)=p^k$. Hence $\theta_p(T,N)=\theta_p(g)=(1-p^{-1})^{-1}$. If $v_p(1-y)>2$ or $v_p(1-y)$ is odd, we have $\widehat{\Theta}_p(y)=0$. Otherwise, we have
\[
\widehat{\Theta}_p(y)=\int_{v_p(z)=1-v_p(1-y)/2}(1-p^{-1})^{-1}\frac{\rmd z}{|z|_p}=1.
\]

\underline{\emph{Case 2:}}\ \  $\chi(p)=-1$, so that $y\in Y_{-1}$.

In this case, we have
\[
\orb(f_p;g)=2\sum_{j=0}^{k-1}p^j+p^k=\frac{2(p^k-1)}{p-1}+p^k=3p^k-2,
\]
where
\[
k=\log_p|T^2-4N|_p'^{-\frac12}=\log_p|z^2-z^2(1-y)|_p'^{-\frac12}=\log_p(|z|_p^{-1}|y|_p'^{-\frac12}).
\]
Hence
\[
\theta_p(T,N)=\theta_p(g)=(1+p^{-1})^{-1}(3-2p^{-k})=2-\frac{4}{3}p^{-k}.
\]

If $v_p(1-y)>2$ or $v_p(1-y)$ is odd, we have $\widehat{\Theta}_p(y)=0$. Otherwise, we have
\begin{align*}
\widehat{\Theta}_p(y)&=\int_{v_p(z)=1-v_p(1-y)/2}\left(2-\frac{4}{3}|z|_p|y|_p'^{\frac12}\right)\frac{\rmd z}{|z|_p}=1-\frac43|y|_p'^{\frac12}\int_{v_p(z)=1-v_p(1-y)/2}\rmd z\\
&=1-\frac{2}{3}|y|_p'^{\frac12}p^{-1+v_p(1-y)/2}=1-\frac13|y|_p'^{\frac12}|1-y|_p^{-\frac12}.
\end{align*}

\underline{\emph{Case 3:}}\ \  $\chi(p)=0$, so that $y\in Y_0$.

In this case, we have
\[
\orb(f_p;g)=1+\sum_{j=1}^{k}p^j=\frac{p^{k+1}-1}{p-1}=2p^k-1,
\]
where $k$ is as in the previous case. Hence
\[
\theta_p(T,N)=\theta_p(g)=2-p^{-k}.
\]

If $v_p(1-y)>2$ or $v_p(1-y)$ is odd, we have $\widehat{\Theta}_p(y)=0$. Otherwise, we have
\begin{align*}
\widehat{\Theta}_p(y)&=\int_{v_p(z)=1-v_p(1-y)/2}\left(2-|z|_p|y|_p'^{\frac12}\right)\frac{\rmd z}{|z|_p}=1-|y|_p'^{\frac12}\int_{v_p(z)=1-v_p(1-y)/2}\rmd z\\
&=1-\frac12|y|_p'^{\frac12}p^{-1+v_p(1-y)/2}=1-\frac14|y|_p'^{\frac12}|1-y|_p^{-\frac12}.\qedhere
\end{align*}
\end{proof}

Now we can compute the nonarchimedean terms in \autoref{prop:limittraceinfinity2}. For simplicity we denote
\[
E=\{y\in \QQ_p\,|\, v_p(1-y)\leq 2,\ v_p(1-y)\ \text{is even} \}.
\]
\begin{proposition}\label{prop:bcompute2}
We have
\[
\Tr(\xi_0(f_p))=1.
\]
\end{proposition}
\begin{proof}
By \autoref{prop:ramifiedtraceeisenstein},  \autoref{lem:orbitalintegral2} and since $|y|_p=|y|_p'$ for $y\in Y_1$, we know that
\[
\Tr(\xi_0(f_p))=\frac{2}{1-p^{-1}}\int_{Y_{1}}|1-y|_{p}^{-1}\widehat{\Theta}_{p}(y) |y|_{p}'^{-\frac12}\rmd y=\frac{2}{1-p^{-1}}\int_{E \cap Y_1}|1-y|_{p}^{-1}|y|_{p}^{-\frac12}\rmd y.
\]
Also, note that $y\in Y_1$ if and only if $y=p^{2n}y_0$, where $n\in \ZZ$ and $y_0\equiv 1\,(8)$.

Now we consider the following three cases:

\underline{\emph{Case 1:}}\ \  $v_p(y)>0$. 

In this case we have $v_p(1-y)=0$. Hence we always have $y\in E$. Hence
\[
\int_{E \cap Y_1,\,v_p(y)>0}|1-y|_{p}^{-1}|y|_{p}^{-\frac12}\rmd y=\sum_{n=1}^{+\infty} 
\int_{p^{2n}(1+8\ZZ_p)}|y|_{p}^{-\frac12}\rmd y=\frac18\sum_{n=1}^{+\infty}p^{-2n}p^n=\frac{1}{8}.
\]

\underline{\emph{Case 2:}}\ \  $v_p(y)=0$.

In this case we have $y\equiv 1\,(8)$ and hence $v_p(1-y)\geq 3$. Therefore 
\[
\int_{E \cap Y_1,\,v_p(y)=0}|1-y|_{p}^{-1}|y|_{p}^{-\frac12}\rmd y=0.
\]

\underline{\emph{Case 3:}}\ \  $v_p(y)<0$.

In this case we have $v_p(1-y)=v_p(y)<0$. Hence $y\in E$ if and only if $v_p(y)$ is even, which  always holds for $y\in Y_1$. Hence
\[
\int_{E \cap Y_1,\,v_p(y)<0}|1-y|_{p}^{-1}|y|_{p}^{-\frac12}\rmd y=\sum_{n=-\infty}^{-1} 
\int_{p^{2n}(1+8\ZZ_p)}|y|_{p}^{-\frac32}\rmd y=\frac18\sum_{n=1}^{+\infty}p^{2n}p^{-3n}=\frac{1}{8}.
\]

Combining the three cases we obtain the desired result.
\end{proof}

\begin{proposition}\label{prop:dcompute2}
We have
\[
\int_{Y_{1}}\log\frac{|1-y|_{p}}{|y|_{p}'}|1-y|_{p}^{-1} \widehat{\Theta}_{p}(y) |y|_{p}'^{-\frac12}\rmd y=\frac12\log 2.
\]
\end{proposition}
\begin{proof}
We have
\[
\int_{Y_{1}}\log\frac{|1-y|_{p}}{|y|_{p}'}|1-y|_{p}^{-1} \widehat{\Theta}_{p}(y) |y|_{p}'^{-\frac12}\rmd y=\int_{E\cap Y_{1}}\log\frac{|1-y|_{p}}{|y|_{p}}|1-y|_{p}^{-1} |y|_{p}^{-\frac12}\rmd y.
\]

As before, we consider the following three cases:

\underline{\emph{Case 1:}}\ \  $v_p(y)>0$. 

In this case we have $v_p(1-y)=0$ and $y\in E$. Hence
\begin{align*}
\int_{E\cap Y_{1},\,v_p(y)>0}\log\frac{|1-y|_{p}}{|y|_{p}}|1-y|_{p}^{-1} |y|_{p}^{-\frac12}\rmd y&=-\sum_{n=1}^{+\infty}\int_{p^{2n}(1+8\ZZ_p)}|y|_p^{-\frac12}\log |y|_p\rmd y=\frac18 \sum_{n=1}^{+\infty}2n p^{-2n}p^n\log p\\
&=\frac14 \sum_{n=1}^{+\infty}n p^{-n}\log p=\frac12\log 2.
\end{align*}

\underline{\emph{Case 2:}}\ \  $v_p(y)=0$.

In this case we have $y\equiv 1\,(8)$ and hence $v_p(1-y)\geq 3$. Therefore 
\[
\int_{E \cap Y_1,\,v_p(y)=0}\log\frac{|1-y|_{p}}{|y|_{p}}|1-y|_{p}^{-1}|y|_{p}^{-\frac12}\rmd y=0.
\]

\underline{\emph{Case 3:}}\ \  $v_p(y)<0$.

In this case we have $v_p(1-y)=v_p(y)<0$. Hence 
$\log(|1-y|_{p}/|y|_{p})=\log 1=0$ and therefore
\[
\int_{E \cap Y_1,\,v_p(y)<0}\log\frac{|1-y|_{p}}{|y|_{p}}|1-y|_{p}^{-1}|y|_{p}^{-\frac12}\rmd y=0.
\]

Combining the three cases we obtain the desired result.
\end{proof}

\begin{proposition}\label{prop:ccompute2}
Let $\epsilon\in \{0,-1\}$, then the integral
\[
\int_{Y_{\epsilon}}|1-y|_{p}^{-1}\widehat{\Theta}_{p}(y) |y|_{p}'^{-\frac12}\rmd y
\]
equals $13/36$ if $\epsilon=-1$ and equals $1/3$ if $\epsilon=0$.
\end{proposition}
\begin{proof}
We first consider the case $\epsilon=-1$, which means that $y\in p^{2n}(5+8\ZZ_p)$ for some $n\in \ZZ$.

By \autoref{lem:orbitalintegral2} and since $|y|_p=|y|_p'$ for $y\in Y_{-1}$, we know that
\[
\int_{Y_{-1}}|1-y|_{p}^{-1}\widehat{\Theta}_{p}(y) |y|_{p}^{-\frac12}\rmd y=\int_{E \cap Y_{-1}}|1-y|_{p}^{-1}|y|_{p}^{-\frac12}\left(1-\frac13|y|_p^{\frac12}|1-y|_p^{-\frac12}\right)\rmd y.
\]

We consider the following three cases:

\underline{\emph{Case 1.1:}}\ \  $v_p(y)>0$. 

In this case we have $v_p(1-y)=0$ and thus $y\in E$. Hence
\begin{align*}
&\int_{E \cap Y_{-1},\,v_p(y)>0}|1-y|_{p}^{-1}|y|_{p}^{-\frac12}\left(1-\frac13|y|_p^{\frac12} |1-y|_p^{-\frac12}\right)\rmd y\\
=&\sum_{n=1}^{+\infty}\int_{p^{2n}(5+8\ZZ_p)}|y|_p^{-\frac12}\rmd y-\frac13\sum_{n=1}^{+\infty}\int_{p^{2n}(5+8\ZZ_p)}\rmd y\\
=&\frac18\sum_{n=1}^{+\infty}p^{-2n}p^n-\frac{1}{24}\sum_{n=1}^{+\infty}p^{-2n}=\frac18-\frac{1}{24} \times \frac13=\frac{1}{9}.
\end{align*}

\underline{\emph{Case 1.2:}}\ \  $v_p(y)=0$.

In this case we have $y\equiv 5\,(8)$ and thus $v_p(1-y)=2$. Hence $y\in E$ and thus
\begin{align*}
\int_{E \cap Y_{-1},\,v_p(y)=0}|1-y|_{p}^{-1}|y|_{p}^{-\frac12}\left(1-\frac13|y|_p^{\frac12}|1-y|_p^{-\frac12}\right) \rmd y=&\int_{5+8\ZZ_p}4\rmd y-\frac13\int_{5+8\ZZ_p}8\rmd y\\
=&\frac12-\frac13=\frac16.
\end{align*}

\underline{\emph{Case 1.3:}}\ \  $v_p(y)<0$.
In this case we have $v_p(1-y)=v_p(y)<0$. Hence $y\in E$ if and only if $v_p(y)$ is even, which  always holds for $y\in Y_{-1}$. Hence
\begin{align*}
&\int_{E \cap Y_{-1},\,v_p(y)<0}|1-y|_{p}^{-1}|y|_{p}^{-\frac12}\left(1-\frac13|y|_p^{\frac12}|1-y|_p^{-\frac12}\right) \rmd y\\
=&\frac23\sum_{n=-\infty}^{-1}\int_{p^{2n}(5+8\ZZ_p)}|y|_p^{-\frac32}\rmd y=\frac{1}{12}\sum_{n=1}^{+\infty}p^{2n}p^{-3n}=\frac{1}{12}.
\end{align*}

Adding the three integrals together we then get the result in this proposition for $\epsilon=-1$.

\smallskip

Next we consider the case $\epsilon=0$. By \autoref{lem:orbitalintegral2} we know that
\[
\int_{Y_{0}}|1-y|_{p}^{-1}\widehat{\Theta}_{p}(y) |y|_{p}'^{-\frac12}\rmd y=\int_{E \cap Y_{0}}|1-y|_{p}^{-1}|y|_{p}'^{-\frac12}\left(1-\frac14|y|_p'^{\frac12}|1-y|_p^{-\frac12}\right)\rmd y.
\]

By definition we have $Y_0=Y_0'\sqcup Y_0''$, where
\[
Y_0'=\{y=p^{2n}y_0\in \QQ_p^\times\,|\, n\in \ZZ\ , y_0\equiv 3\,(4)\}\qquad \text{and} \qquad Y_0''=\{y\in \QQ_p^\times\,|\, v_p(y)\ \text{is odd}\}.
\] 

As before, we consider the following three cases:

\underline{\emph{Case 2.1:}}\ \  $v_p(y)>0$. 

In this case we have $v_p(1-y)=0$ and thus $y\in E$. Recall that $|y|_p'=p^2|y|_p$ for $y\in Y_0'$ and $|y|_p'=p^3|y|_p$ for $y\in Y_0''$. Hence
\begin{align*}
\int_{Y_{0}'}|1-y|_{p}^{-1}\widehat{\Theta}_{p}(y) |y|_{p}'^{-\frac12}\rmd y=&\int_{E \cap Y_{0}',\,v_p(y)>0}\frac12|1-y|_{p}^{-1}|y|_{p}^{-\frac12}\left(1-\frac12|y|_p^{\frac12} |1-y|_p^{-\frac12}\right)\rmd y\\
=&\frac12\sum_{n=1}^{+\infty}\int_{p^{2n}(3+4\ZZ_p)}|y|_p^{-\frac12}\rmd y-\frac14\sum_{n=1}^{+\infty}\int_{p^{2n}(3+4\ZZ_p)}\rmd y\\
=&\frac1{8}\sum_{n=1}^{+\infty}p^{-2n}p^n-\frac{1}{16}\sum_{n=1}^{+\infty}p^{-2n}=\frac1{8}-\frac{1}{16} \times \frac13=\frac{5}{48}.
\end{align*}
and
\begin{align*}
\int_{Y_{0}''}|1-y|_{p}^{-1}\widehat{\Theta}_{p}(y) |y|_{p}'^{-\frac12}\rmd y=&\int_{E \cap Y_{0}'',\,v_p(y)>0}\frac{1}{2\sqrt2}|1-y|_{p}^{-1}|y|_{p}^{-\frac12}\left(1-\frac{\sqrt2}{2}|y|_p^{\frac12} |1-y|_p^{-\frac12}\right)\rmd y\\
=&\frac{1}{2\sqrt{2}}\sum_{n=1}^{+\infty}\int_{p^{2n-1}\ZZ_p^\times}|y|_p^{-\frac12}\rmd y-\frac14\sum_{n=1}^{+\infty}\int_{p^{2n-1}\ZZ_p^\times}\rmd y\\
=&\frac{1}{2\sqrt{2}}\sum_{n=1}^{+\infty}p^{-2n}p^{n-\frac12}-\frac{1}{8}\sum_{n=1}^{+\infty}p^{-2n+1}= \frac1{4}-\frac{1}{4} \times \frac13=\frac{1}{6}.
\end{align*}

\underline{\emph{Case 2.2:}}\ \  $v_p(y)=0$.

In this case we must have $y\in Y_0'$ and thus $y\equiv 3\,(4)$ and $v_p(1-y)=1$. Hence $y\notin E$ and thus
\[
\int_{Y_{0}}|1-y|_{p}^{-1}\widehat{\Theta}_{p}(y) |y|_{p}'^{-\frac12}\rmd y=\int_{E \cap Y_{0}',\,v_p(y)=0}\frac12|1-y|_{p}^{-1}|y|_{p}^{-\frac12}\left(1-\frac12|y|_p^{\frac12} |1-y|_p^{-\frac12}\right) \rmd y=0.
\]

\underline{\emph{Case 2.3:}}\ \  $v_p(y)<0$.

In this case we have $v_p(y)=v_p(1-y)$. Hence we must have $y\in Y_0'$. Therefore
\begin{align*}
\int_{Y_{0}}|1-y|_{p}^{-1}\widehat{\Theta}_{p}(y) |y|_{p}'^{-\frac12}\rmd y=&\int_{E \cap Y_{0}',\,v_p(y)<0}\frac12|1-y|_{p}^{-1}|y|_{p}^{-\frac12}\left(1-\frac12|y| _p^{\frac12}|1-y|_p^{-\frac12}\right) \rmd y
\\=&\frac14\sum_{n=-\infty}^{-1}\int_{p^{2n}(3+4\ZZ_p)}|y|_p^{-\frac32}\rmd y=\frac{1}{16}\sum_{n=1}^{+\infty}p^{2n}p^{-3n}=\frac{1}{16}.
\end{align*}

Combining these three cases we obtain our result for $\epsilon=0$.
\end{proof}

Finally, we deal with the terms
\[
\widetilde{\Tr}\left(\Ind_{\B(\QQ_p)}^{\G(\QQ_p)}(0,\triv)(f_p)\right)\quad\text{and} \quad\Tr\left(R_p(0,\triv)^{-1}R_p'(0,\triv)\xi_0(f_p)\right)
\]
in \autoref{prop:limittraceinfinity2}. Since $f_p$ is spherical, $R_p(s,\triv)$ is the identity. Hence the second term vanishes. Now we only need to consider the first one.

\begin{proposition}\label{prop:ecompute2}
We have
\[
\widetilde{\Tr}\left(\Ind_{\B(\QQ_p)}^{\G(\QQ_p)}(0,\triv)(f_p)\right)=\frac43\log 2.
\]
\end{proposition}
\begin{proof}
By definition,
\begin{align*}
\widetilde{\Tr}\left(\Ind_{\B(\QQ_p)}^{\G(\QQ_p)}(0,\triv)(f_p)\right)&=\int_{ \A(\QQ_p)}\frac{|a-b|_p}{|ab|_p^{1/2}}\worb\sptilde(f_p;t)\rmd t\\
&=\int_{\A(\QQ_p)}\frac{|a-b|_p}{|ab|_p^{1/2}}\worb\sphat(f_p;t)\rmd t+ 2\int_{\A(\QQ_p)}\frac{|a-b|_p}{|ab|_p^{1/2}}\log|ab|_p^{1/2}\orb(f_p;t)\rmd t
\end{align*}
for $t=(\begin{smallmatrix}  a &  \\   & b \end{smallmatrix})$. Since $\orb(f_p;t)\neq 0$ only if $\det t=ab\in \ZZ_p^\times$, the second term vanishes. Hence we only need to compute
\begin{equation}\label{eq:modifiedsecondkindcompute}
\int_{\A(\QQ_p)}\frac{|a-b|_p}{|ab|_p^{1/2}}\worb\sphat(f_p;t)\rmd t.
\end{equation}

By \autoref{cor:unramifiedweightedlocalorbital}, $\worb\sphat(f_p;t)=0$ unless $a,b\in \ZZ_p$ and $ab\in \ZZ_p^\times$ (i.e., $a,b\in \ZZ_p^\times$), and in this case,
\[
\worb\sphat(f_p;t)=\frac{2}{|a-b|_p}\sum_{j=1}^{v_p(a-b)}\frac{\log p}{p^j}.
\]
Hence
\[
\eqref{eq:modifiedsecondkindcompute}=2\int_{\ZZ_p^\times\times\ZZ_p^\times} \sum_{j=1}^{v_p(a-b)}\frac{\log p}{p^j}\rmd^\times a\rmd^\times b.
\]
By making the change of variable $b\mapsto ab$, this equals
\begin{align*}
2\int_{\ZZ_p^\times\times\ZZ_p^\times} \sum_{j=1}^{v_p(1-b)}\frac{\log p}{p^j}\rmd^\times a\rmd^\times b&=2\int_{\ZZ_p^\times}\sum_{j=1}^{v_p(1-b)}\frac{\log p}{p^j}\rmd^\times b=2\sum_{n=1}^{+\infty}\int_{1+p^n\ZZ_p^\times}\sum_{j=1}^{n}\frac{\log p}{p^j}(1-p^{-1})^{-1}\rmd b\\
&=2\sum_{n=1}^{+\infty}p^{-n}(1-p^{-1})\left(1-\frac{1}{p^n}\right)(1-p^{-1})^{-1}\log p\\
&=2\sum_{n=1}^{+\infty}p^{-n}\log p-2\sum_{n=1}^{+\infty}p^{-2n}\log p=\left(2-\frac23\right)\log 2=\frac43\log 2.\qedhere
\end{align*}
\end{proof}

By plugging the results of \autoref{prop:bcompute2}, \autoref{prop:dcompute2}, \autoref{prop:ccompute2} and \autoref{prop:ecompute2} into \autoref{prop:limittraceinfinity2}, and using a limiting process for the possible case $\Tr(\xi_0(f_\infty))=0$, we obtain

\begin{theorem}[Limit form of the trace formula, local version for the archimedean place] \label{thm:limittraceformulaarchimedean}
Let $f_\infty\in C_c^\infty(Z_+\bs\G(\RR))$. Let $\xi_0$ be defined in \autoref{subsec:traceformula}  and $\widehat{\Theta}_{\infty}(x)$ be defined in \autoref{subsec:singularities}. Then we have
\begin{align*}
    &\frac12\Tr\left(R_\infty(0,\triv)^{-1}R_\infty'(0,\triv)\xi_0(f_\infty)\right)= \log 2\Tr(\xi_0(f_\infty))+ {\uppi\int_{X_1}|1-x|_\infty^{-1}\widehat{\Theta}_\infty(x)|x|_\infty ^{-\frac12}\rmd x} \\
   -& 2\int_{X_0}\log\frac{|1-x|_\infty}{|x|_\infty}|1-x|_\infty^{-1}\widehat{\Theta}_\infty(x)|x|_\infty^{-\frac12}\rmd x
   +\frac12\widetilde{\Tr}\left(\Ind_{\B(\RR)}^{\G(\RR)}(0,\triv)(f_\infty)\right).
\end{align*}
\end{theorem}

By the above theorem together with \autoref{prop:limittraceinfinity2}, we obtain the formula for $p=2$:
\begin{theorem}[Limit form of the trace formula, local version for $p=2$] \label{thm:limittraceformulalocal2}
Let $p=2$ and $f_p\in C_c^\infty(\G(\QQ_p))$. Let $\xi_0$ be defined in \autoref{subsec:traceformula}, $|\cdot|_{p}'$ be defined in \autoref{subsec:modifiednorm}, and $\widehat{\Theta}_{p}(y)$ be defined in \autoref{subsec:singularities}. Then we have
\begin{align*}
    &\frac12\Tr\left(R_p(0,\triv)^{-1}R_p'(0,\triv)\xi_0(f_p)\right)=-2\log 2\Tr(\xi_0(f_p))-\frac12(1+p^{-1})\frac{\log p}{1-p^{-1}}\Tr(\xi_0(f_p))\\
    +&\sum_{\epsilon\in \{0,-1\}}\frac{2(1-\epsilon p^{-1})\log p}{(1-\epsilon)(1-p^{-1})^2}\int_{Y_{\epsilon}}|1-y|_{p}^{-1}\widehat{\Theta}_{p}(y) |y|_{p}'^{-\frac12}\rmd y\\
   -&\frac{2}{1-p^{-1}}\int_{Y_{1}}\log\frac{|1-y|_{p}}{|y|_{p}'}|1-y|_{p}^{-1}\widehat{\Theta}_{p}(y) |y|_{p}'^{-\frac12}\rmd y
   +\frac12\widetilde{\Tr}\left(\Ind_{\B(\QQ_p)}^{\G(\QQ_p)}(0,\triv)(f_p)\right).
\end{align*}
\end{theorem}

\bibliography{ref.bib}
\bibliographystyle{amsalpha}

\end{document}